\documentclass[12pt,a4paper]{amsart}

\usepackage[T1]{fontenc}
\usepackage[utf8]{inputenc}
\usepackage[english]{babel}
\usepackage{amsfonts, amssymb, amsmath, amsthm}
\usepackage[all]{xy}
\usepackage{enumerate}

\usepackage{graphicx}
\usepackage{psfrag}
\usepackage{geometry}
\usepackage[colorlinks=true,pdfstartview=FitV,linkcolor=blue,citecolor=green,urlcolor=black,filecolor=magenta]{hyperref}
\usepackage{verbatim}

\newtheorem{theorem}{Theorem}[section]
\newtheorem{definition}[theorem]{Definition}
\newtheorem{prop}[theorem]{Proposition}
\newtheorem{lemma}[theorem]{Lemma}
\newtheorem{cor}[theorem]{Corollary}

\newcommand{\Aa}{{\mathcal A}}
\newcommand{\Bb}{{\mathcal B}}
\newcommand{\Cc}{{\mathcal C}}

\newcommand{\Kk}{{\mathcal K}}

\newcommand{\Ss}{{\mathcal S}}
\newcommand{\Tt}{{\mathcal T}}

\newcommand{\Ww}{{\mathcal W}}

\newcommand{\id}{\mathrm{id}}

\newcommand{\CM}{{\mathbb C}}

\newcommand{\FM}{{\mathbb F}}
\newcommand{\HM}{{\mathbb H}}

\newcommand{\NM}{{\mathbb N}}

\newcommand{\PM}{{\mathbb P}}

\newcommand{\TM}{{\mathbb T}}

\newcommand{\Z}{{\mathbb Z}}

\newcommand{\R}{{\mathbb R}}

\newcommand{\C}{{\mathbb C}}

\newcommand{\hs}{{\mathscr H}}

\newcommand{\rs}{{\mathscr R}}

\newcommand{\kg}{{\rho}}
\newcommand{\rS}{{\tilde\rs}}
\newcommand{\rfS}{{\hat \rs}}
\newcommand{\cfS}{{\hat \cc}}

\newcommand{\Tr}{\mathrm{Tr}}

\newcommand{\ch}{\mbox{\rm ch}}

\newcommand{\one}{\mbox{\rm 1}}

\newcommand{\cc}{\mathfrak c}
\newcommand{\CA}{$C^*$-algebra}

\newcommand{\Us}{{\mathfrak S}}
\newcommand{\Ub}{{\mathfrak F}}

\renewcommand{\rs}{\mathfrak r}
\renewcommand{\ss}{{\mathfrak s}}
\newcommand{\rh}{{\mathfrak h}}

\newcommand{\rf}{\mathfrak f}
\newcommand{\rcl}{\mathfrak l}

\newcommand{\DK}{D\!K}
\newcommand{\Ad}{\mathrm{Ad}}

\newcommand{\st}{\mathrm{st}}

\newcommand{\RA}{$C^{*,r}$-algebra}
\newcommand{\BA}{Banach algebra}

\newcommand{\ot}{\otimes}
\newcommand{\hot}{\hat{\otimes}}

\newcommand{\osi}{OSI}
\newcommand{\osu}{OSU}

\newcommand{\tv}{torsion valued}

\title[Cyclic cohomology for graded $C^{*,r}$-algebras]{Cyclic cohomology for graded $C^{*,r}$-algebras and its pairings with van Daele $K$-theory}
\author{Johannes Kellendonk}
\address{Univerisit\'{e} de Lyon, Universit\'{e} Claude Bernard Lyon 1, Institute Camille Jordan, CNRS UMR 5208, 69622 Villeurbanne, France}
\email{kellendonk@math.univ-lyon1.fr}

\date{\today}

\begin{document}
\maketitle
\begin{abstract}
We consider cycles 
for graded $C^{*,\mathfrak{r}}$-algebras (Real $C^{*}$-algebras) which are compatible with the $*$-structure and the real structure. 
Their characters are cyclic cocycles. We define a Connes type pairing between such characters and elements of the van Daele $K$-groups of the $C^{*,\mathfrak{r}}$-algebra and its real subalgebra. This pairing vanishes on elements of finite order. 
We define a second type of pairing between characters and $K$-group elements which is derived from a unital inclusion of \CA s. It is potentially non-trivial on elements of order two and torsion valued. Such torsion valued pairings yield topological invariants for insulators. The two-dimensional Kane-Mele and the three-dimensional Fu-Kane-Mele strong invariant are special cases of torsion valued pairings.
We compute the pairings for a simple class of periodic models and establish structural results for two dimensional aperiodic models with odd time reversal invariance.   
\end{abstract}

\section{Introduction}
Recent developments in solid state physics, notably the classification of topological phases \cite{Schnyder,Kitaev}, underline the importance of real $K$- and $KK$-theory in physics. As part of this development, the \CA ic approach to solid state systems \cite{Bel86} was extended to describe insulators of different types by including a grading or a real structure on the observable algebra \cite{Kel1} (see also \cite{Thiang} for a related proposal). In this article we discuss the cyclic cohomology (in the formulation as characters of cycles) of such algebras.
 
In the \CA ic approach 
topological quantised transport coefficients are expressed as pairings of a $K$-group element either with a character of a cycle (a Connes pairing) or with a $K$-homology class (an index pairing). 
While the Connes pairing with a character is more directly related to the physical interpretation as a transport coefficient and yields a local formula, the index pairing proves integrality of the coefficient and can be extended to the strong disorder regime \cite{BES}. The two approaches are thus complementary. This theory of pairings has first been developped for the Integer Quantum Hall Effect \cite{Bel86,ConnesBook,BES,BCR1} 
and recently extended to topological insulators of complex type \cite{PS}, and is now in active development for insulators of real type \cite{BCR2,BKR}.
The present work aims to contribute to this development.

Insulators of real type are insulators which transform under an anti-linear automorphism of order $2$, like complex conjugation, and so there is an additional ingredient to take into account: a real structure on the \CA. This leads to the consideration of $K$-groups of real \CA s and brings in a feature which has not been of importance for complex insulators, namely the occurence of elements of finite order (torsion elements) in the $K$-group. While also complex $K$-groups may contain torsion elements, for instance in the case of certain quasicrystals \cite{GHK}, the physical significance of those torsion elements remains unclear so far and so they don't play a big role. But for certain real topological insulators, like the Kane-Mele model \cite{KaneMele}, the most relevant $K$-group elements have order $2$ and show up in experiments \cite{Molenkamp}. 
It is therefore a problem that the Connes pairing is trivial on elements of finite order. 

To overcome this problem we define torsion-valued pairings, somewhat in the spirit of the determinant of de la Harpe-Skandalis. Such a torsion-valued pairing is defined on the kernel $\ker \varphi_*$ of the map induced in $K$-theory by an inclusion $B\stackrel{\varphi}\hookrightarrow \tilde B$ of \CA s. We analyse it more closely in two cases which arise naturally when comparing a real \CA\ with its complexification. They are related to an exact sequence attributed to Wood-Karoubi. It turns out that, in these two cases, the $K$-class of the Hamiltonian of an insulator belongs to the above mentioned kernel if and only if that $K$-class admits a representative (possibly different from the Hamiltonian) which admits an extra symmetry: a spin symmetry, or an imaginary chiral symmetry. We refer to the first case as even, and to the second as odd. We provide explicit local formulae for the torsion valued pairing which involve the representative and its extra symmetry.

The results presented here will permit a formulation 
of the bulk boundary correspondence in the spirit of \cite{KRS,KS2004},
as an equality between torsion-valued pairings of characters of cycles with $K$-group elements.
We intend to describe this in an upcoming publication.
The complementary approach to the bulk boundary correspondence which is based on the index pairing has already been developped to quite some extend \cite{GSchuba,BCR2,BKR}. In its most powerful form it is based on $KK$-theory and the (unbounded) Kasparov product. These technics are very different from what we do here.
\bigskip

Our article is organised as follows. After recalling some preliminaries we explain briefly van Daele's formulation of $K$-theory for real or complex Banach algebras. In Section~\ref{sec-cyclic} we discuss the cyclic cohomology for graded algebras in the framework of cycles and their characters. We define a Connes-type pairing between characters and $K$-group elements. It reduces to the usual Connes pairing in case the grading on the algebra is trivial.  
We introduce the notion of the sign and the parity of a cycle and 
formulate necessary conditions under which the Connes-pairing and the torsion-valued pairing can be non-trivial.
Section~\ref{sec-cyclic} ends with a proof that the pairing is compatible with the suspension construction also in the graded case.

In Section~\ref{sec-tor} we introduce torsion-valued pairings of characters of cycles with $K$-group elements. The basic construction is based on a unital inclusion of one \CA\ into another and we focuss on two such inclusions, the inclusion of a real \CA\ $B$ into its graded tensor product with the Clifford algebra $Cl_{1,0}$, and the inclusion of $B$ into its complexification. We explain the relevance of extra symmetries and provide explicit formulas for the torsion valued pairings. 

In Section~\ref{sec-simple-appl} we compute the pairings for a class of simple periodic models which are used in the literature for modelling topological phases with various symmetries. These are closely related to the Bott element on the torus. Our approach is based on a systematic use of Clifford algebras, their real structeres, and their graded and ungraded representations. 

In Section~\ref{sec-aperiodic} we provide structural results for aperiodic models, the main part here is restricted to two dimensional systems with odd time reversal invariance. 

Our definition of the even torsion valued pairing was largely influenced by the recent work on periodically driven systems with odd time reversal invariance \cite{Gaw,Gawedzki}. We will explain the connection in more detail in the last 1section.

\subsection*{Acknowledgement} I would like to thank Krzysztof Gawedzki,  David Carpentier 
and Michel Fruchart for the very useful discussions about their work \cite{Gaw,Gawedzki}. I am furthermore very thankful to Denis Perrot and Samuel Guerin for pointing out the relevance of the exact sequence of Wood-Karoubi (Thm.~\ref{thm-WK}) \cite{Guerin}
for my construction.

\section{Preliminaries}
We consider here real or complex associative algebras which mostly are equipped with a norm, a $*$-structure (an involution which is anti-linear in the complex case), a grading and sometimes also with a real structure. 

Let $G$ be an abelian group.
Recall that a $G$-grading on an algebra $A$ is a direct sum decomposition of $A=\bigoplus_{g\in G}A_g$ such that the algebra product $ab$ of $a\in A_g$ with $b\in A_h$ lies in $A_{gh}$. The elements of $A_g$ have degree $g$ and we denote that degree by $|a|_G$.
If $A$ is a $*$-algebra we also require that the subspaces $A_g$ are invariant under the $*$-operation.

We are interested in the case that $G=\Z_2$, $G=\Z$, or $G=\Z\times\Z_2$. Since it is the first case which arises most often in the formulas we simplify them by writing $|a|$ for 
$|a|_{\Z_2}$ and when we speak about a grading we mean a $\Z_2$-grading.
 
An alternative way to define a $\Z_2$-grading on a real or complex $*$-algebra is by means of at $*$-automorphism $\gamma$ of order $2$. Then $A_+$, the even elements, are those which satisfy $\gamma(a)=a$ and $A_-$, the odd ones, are those which satisfy $\gamma(a) = -a$.

An {\em odd self-inverse} (\osi ) of a graded algebra $(A,\gamma)$ is an odd element of $A$ which is its own inverse, and we denote by $\Ub(A,\gamma)$  the set of \osi s of $A$
$$\Ub(A,\gamma) =\{x\in A:-\gamma(x)=x=x^{-1}\}.$$ 
If $A$ is a normed algebra we say that two \osi s of $A$ are osi-homotopic if they are homotopic in $\Ub(A,\gamma)$.
Not all graded algebras contain an \osi\ and we say that the grading is balanced if $A$ contains at least one. This requires that $A$ is unital, and if $A$ is not then we will have to add a unit.
An {\em odd self-adjoint unitary} (\osu ) of a graded $*$-algebra $(A,\gamma)$ is a self-adjoint \osi, and we denote by $\Us(A,\gamma)$  the set of \osu s of $A$
$$\Us(A,\gamma) =\{x\in A:-\gamma(x)=x=x^*=x^{-1}\}.$$ 
If $A$ is equipped with a norm then two \osu s of $A$ are osu-homotopic if they are homotopic in $\Us(A,\gamma)$.
For instance, any two anticommuting \osu s $x,y$ of a normed $*$-algebra 
are osu-homotopic, a homotopy being given by $c_t x + s_t y$, $t\in[0,1]$, where $c_t = \cos(\frac{\pi t}2)$ and  $s_t = \sin(\frac{\pi t}2)$.

Examples of graded \CA s are the complex Clifford algebras $\C l_k$. $\C l_k$ is the \CA\ generated by $k$ pairwise anticommuting \osu s $\rho_1,\cdots,\rho_k$.
We denote the grading of Clifford algebras always by $\st$. Concretely, $(\C l_1,\st)$ is isomorphic (as graded algebra) to $\C\oplus\C$ with grading given by exchange of the summands $\st(a,b) = (b,a)$ and $(\C l_2,\st)$ is isomorphic to $M_2(\C)$ with grading given by declaring diagonal matrices even and off-diagonal ones odd. As usual we will denote the generators also by $\rho_1=\sigma_x=\begin{pmatrix}0&1\\1&0\end{pmatrix}$ and $\rho_2=\sigma_y=\begin{pmatrix}0&-i\\i&0\end{pmatrix}$. Then $\sigma_z= -i\sigma_x\sigma_y$ is, of course, even.

The standard extension of a grading $\gamma$ on an algebra $A$ to the algebra of matrices $M_m(A)$ is entrywise, we denote this extension by $\gamma_m$. If $m=2$ there is another grading which plays an important role, namely $\gamma_{ev}$ 
$$\gamma_{ev}\begin{pmatrix} a & b \\ c & d\end{pmatrix}
= \begin{pmatrix} \gamma(a) & -\gamma(b) \\ -\gamma(c) & \gamma(d)\end{pmatrix}.
$$  

A real structure on a  complex \CA\ is an anti-linear $*$-automorphism $\rs$ of order $2$. If the algebra is graded then we require tacitly that the grading and the real structure commute $\rs\circ\gamma=\gamma\circ\rs$. We also call the data $(\gamma,\rs)$ a graded real structure on the algebra $A$. A \RA\ $(A,\rs)$ is a complex \CA\ equipped with a real structure. 
The subalgebra $A^\rs$ of $\rs$-invariant elements is a real \CA\ and referred to as the real subalgebra of $(A,\rs)$. 
Any real \CA\ arises as a sub-algebra of a complex \RA\ in such a way. \RA s are elsewhere also called Real \CA s (with capital R). 

Examples of \RA s are $(\C l_{r+s},\rcl_{r,s})$ where $\rcl_{r,s}$ is the real structure defined by $\rcl_{r,s}(\rho_i) = \rho_i$ for $r$ generators, and $\rcl_{r,s}(\rho_j) = -\rho_j$ for the $s$ other generators. The real subalgebra is thus the algebra generated by the $r$ \osu s 
$\rho_i$ and the $s$ 
odd anti-selfadjoint unitaries $i\rho_{j}$. The latter square to $-1$. This real sub-algebra which we denote\footnote{the notation is not uniform, other authors use  $Cl_{s,r}$ for this algebra}
$Cl_{r,s}$ is a real Clifford algebra. 

\newcommand{\nn}{{\mu}}
Define $\nn(k)$ to be the greatest integer smaller or equal to $\frac{k}2$.
Note that $(-1)^{\nn(k)}$ is the sign of the permutation
$1\cdots k \mapsto k\cdots 1$.
Define $\Gamma_k\in\C l_k$ by
$$ \Gamma_k = i^{-\nn(k)}\kg_1\cdots\kg_k$$
where $\kg_i$ are the generators.  Then $\Gamma_k^*=\Gamma_k$.
$\Gamma_k$ depends on the choice of order of the generators, although only up to a minus sign, and we choose, for $k=2$,
$\kg_1=\sigma_x$ and $\kg_2=\sigma_y$ so that $\Gamma_2 = \sigma_z$.
If the context is clear we also simply write $\Gamma$ for $\Gamma_k$.

We will consider graded tensor products of graded Banach or \CA s with graded finite dimensional algebras where the grading may be a $\Z_2$-grading, in which case we denote 
the tensor product as $A\hot B$, or the $\Z$-grading, in which case  
we denote it as $A\wedge B$. The grading on the graded tensor product is the product grading. By definition, $(1\hot b)(a\hot 1) = (-1)^{|a||b|} a\hot b$,
$(a\hot b)^* = (-1)^{|a||b|} a^*\hot b^*$ and $(1\wedge b)(a\wedge 1) = (-1)^{|a|_\Z|b|_\Z} a\wedge b$, $(a\wedge b)^* = (-1)^{|a|_\Z|b|_\Z} a^*\wedge b^*$.
In the case that one of the algebras is trivially $\Z_2$-graded, the graded tensor product $\hot$ coincides with the ungraded one and for clarity we will simply write $\otimes$ instead of $\hot$ in that case.

The following simple result will be important.
\begin{lemma}[\cite{vanDaele1}]\label{lem-iso}
Let $(A,\gamma)$ be a complex balanced graded algebra and $e$ an \osi\ in $A$.
Then $\psi_e:(A\hot \C l_{2},\gamma\ot\st)\to (M_2(A),\gamma_2)$
$$\psi_e(x\hot 1) = \begin{pmatrix} x & 0 \\ 0 & (-1)^{|x|}exe^{-1} \end{pmatrix} ,\quad
\psi_e(1\hot \sigma_x) = \begin{pmatrix} 0 & e^{-1} \\ e & 0 \end{pmatrix} ,\quad
\psi_e(1\hot i\sigma_y) = \begin{pmatrix} 0 & e^{-1} \\ -e & 0 \end{pmatrix} 
$$
is an isomorphism of graded algebras.
If $(A,\gamma)$ is a $*$-algebra 
  and $e$ an \osu\ then $\psi_e$ is a $*$-isomorphism.
If $(A,\gamma)$ carries a real structure $\rs$ 
 (which commutes with $\gamma$) and $e$ is $\rs$-invariant then
\begin{eqnarray*}
\psi_e\circ (\rs\otimes \rcl_{1,1}) &=& \rs_2\circ \psi_e \\   
\psi_e\circ (\rs\otimes \rcl_{0,2}) &=& \Ad_{ Y } \circ\rs_2\circ\gamma_{ev} \circ \psi_e, \quad Y =\begin{pmatrix} 0 & e^{-1} \\ -e & 0 \end{pmatrix}\\ 
\psi_e\circ (\rs\otimes \rcl_{2,0}) &=& \Ad_{X} \circ\rs_2\circ\gamma_{ev} \circ \psi_e, \quad X =\begin{pmatrix} 0 & e^{-1} \\ e & 0 \end{pmatrix}
\end{eqnarray*}
\end{lemma}
\begin{proof} This is a direct calculation.
\end{proof}
\section{Van Daele $K$-theory}
In \cite{Kel1} we developped the point of view that an insulator corresponds to a self-adjoint invertible element in the observable \CA\ $A$ and that the symmetry type of an insulator 
is described by a grading, or a real structure, or a graded real structure on $A$. Any
self-adjoint invertible element of a \CA\ is homotopic (via a continuous path of invertible self-adjoint elements) to a self-adjoint unitary. By taking into account the grading (which has to be put in by tensoring with $\C l_1$ if the insulator does not have chiral symmetry)  
the topological phases for a given symmetry type and observable algebra $A$ are classified by homotopy classes of \osu s in $A$ (or $A\ot \C l_1$), and these define the van Daele $K$-group of the 
graded algebra $C^*$- or $C^{*,r}$-algebra $A$. 

We recall the basic definitions of van Daele $K$-theory for graded Banach algebras
refering the reader for details to the original articles by van Daele \cite{vanDaele1,vanDaele2}.
Let $(A,\gamma)$ be a balanced graded normed algebra. We choose an \osi\ $e\in A$ and define the semigroup
$$V_e(A,\gamma):=\bigsqcup_{n\in \NM} \Ub(M_n(A),\gamma_n) / \sim_h^e$$ 
where $M_m(A)\ni x\sim_h^e y\in M_l(A)$ if, for some $n\in\NM$,  $x\oplus e_{n-m}$
is osi-homotopic to $y\oplus e_{n-l}$. 
Here $e_n = {e\oplus\cdots\oplus e}$ ($n$ summands), and semigroup addition is given by the direct sum $[x]+[y]=[x\oplus y]$.
The van Daele $K$-group $\DK_e(A,\gamma)$ is 
the Grothendieck group of $V_e(A,\gamma)$. 
We refer to the choice of $e$ as a choice of base-point. 
$ \DK_e(A,\gamma)$ depends on $e$ only up to isomorphism. If $e$ is osi-homotopic to its negative $-e$ then $V_e(A,\gamma)$ is already a group with neutral element $[e]$ and inverse map $[x]\mapsto [-exe]$. In that case $\DK_e(A,\gamma)=V_e(A,\gamma)$. It follows from 
Lemma~\ref{lem-iso} that 
$\DK_e(A,\gamma)\cong 
\DK_{1\hot \sigma_x}(A\hot C l_{1,1},\gamma\otimes\st)$ and so the r.h.s.\ may be taken as the definition of the van Daele $K$-group if $A$ is unital but the grading on $A$ not balanced or even trivial. For non-unital algebras the $K$-group is defined as usual as the kernel of the map induced by the epimorphism $A^+\to \C$ (or $\R$) from the minimal unitization $A^+$ to the one-dimensional algebra.

If $A$ is a \CA\ then any element of $\Ub(A,\gamma)$ is osi-homotopic to an element of 
$\Us(A,\gamma)$ and moreover, any two osi-homotopic \osu s are osu-homotopic \cite{vanDaele1}. For \CA s one may therefore  replace in the definition of $V_e(A,\gamma)$ the sets $\Ub(M_n(A),\gamma_n)$
by $\Us(M_n(A),\gamma_n)$ and osi-homotopy by osu-homotopy.
The grading on a \CA\ is hence balanced if the algebra contains an \osu.
This is  what was done in \cite{Kel1} to describe topological insulators, but here the greater flexibility of working with \osi s will be convenient.

If $(A,\gamma,\rs)$ is a graded \RA\ then its van Daele $K$-group is by definition the van Daele $K$-group of its real subalgebra $(A^\rs,\gamma)$.

If the grading is clear then we simply write $\DK_e(A)$ for $\DK_e(A,\gamma)$, or even only
$\DK(A)$, in case the dependence on $e$ is not important.

The $K$-groups in other degrees are defined as 
$$K_{1-i}(A,\gamma) := \DK_{e\hot 1}(A\hot\C l_{i},\gamma\otimes\st) $$
in the complex case, and
$$K_{1-r+s}(A,\gamma) := \DK_{e\hot 1}(A\hot C l_{r,s},\gamma\otimes\st) $$
in the real case. 

If $A$ is trivially graded then the van Daele $K$-group $K_i(A,\id)$ is isomorphic to the standard $K$-group of $A$ which we denote, for complex $A$ also by $KU_i(A)$, and for a  real $A$ by $KO_i(A)$. We recall also that if $(A,\gamma)$ is inner graded, that is, $\gamma$ is given by conjugation with a self-adjoint unitary from $A$, then $K_i(A,\gamma)\cong K_i(A,\id)$ \cite{Kel1}. Thus, under the assumption that chiral symmetry, when it appears, is inner, only standard real or complex $K$-theory is needed to describe topological phases. We find it however extremely useful in the following to use van Daele's formulation of $K$-theory also in the trivially graded case. 
\subsection{Suspensions}
A grading $\gamma$ on a \BA\ $A$ can be extended pointwise to the cone 
$CA:=\{f:[0,1]\to A:f(0)=0\}$ and the suspension $SA=\{f:[0,1]\to A:f(0)=f(1)=0\}$ of $A$,
and this is the extension we will always use. We thus have an exact sequence of graded algebras
\begin{equation}\label{ses-S}
0 \to  SA \to CA \to A \to 0.
\end{equation}
There are two natural ways to extend a real structure $\rs$ from $A$ to $SA$: Pointwise, i.e.\ $\rS(f)(t) = \rs(f(t))$, or flip-pointwise $\rfS(f)(t) = \rs(f(1-t))$. However, only the pointwise extension also extends to the cone $CA$ and thus turns (\ref{ses-S}) into an exact sequence of \RA s. 
The standard suspension of the \RA\ $(A,\rs)$ is the \RA\ $(SA,\tilde \rs)$ and occurs as the ideal in (\ref{ses-S}). Its real subalgebra $(SA)^{\tilde\rs}=SA^\rs$ is the ideal in the exact sequence of real \CA s $0 \to  SA^\rs \to CA^\rs \to A^\rs \to 0$. The boundary map in van Daele's formulation of $K$-theory, which we denote $\beta$, reads as follows.
Denote $c_t = \cos(\frac{\pi t}2)$ and $s_t = \sin(\frac{\pi t}2)$. For an \osu\ $x\in M_m(A)$ define $\nu(x):[0,1]\to M_m(A)\hot Cl_{1,0}$ to be the function
\begin{equation}\label{eq-Bott-dual} 
t\mapsto \nu_t(x)=c_t \hot 1 + s_t x\hot \kg
\end{equation}
\begin{theorem}[\cite{vanDaele2}] \label{thm-susp} Let $(A,\gamma)$ be a balanced graded 
\BA\ and $e$ a basepoint.
The map $\beta:\DK_e(A)\to \DK_{1\hot\kg}(SA\hot Cl_{1,0})$,
$$\beta[x] = [ \nu(x) \nu^{-1}(e_m) (1\hot \kg)_m \nu(e_m) \nu^{-1}(x)] $$
is an isomorphism. 
\end{theorem}

For trivially graded \CA s $A$ this reduces to the standard isomorphisms   \cite{Hu}
$$\beta :KU_i(A) \to KU_{i-1}(SA),\quad \beta :KO_i(A^\rs) \to KO_{i-1}(SA^\rs).$$
For later use we recall the details in the case $i=0$ where the relevant algebra is 
$(A\ot \C l_1,\id\ot\st)$, in the complex, and $(A^\rs\ot C l_{1,0},\id\ot\st)$ in the real case.
We choose the basepoint to be $e=1\ot \kg$ and so to be an odd element which commutes with any other odd element of $A\ot \C l_1$. Hence any \osu\ $x$ of $M_m(A)\hot Cl_{1,0}$ is of the form $x = he_m$ where $h$ is a self-adjoint unitary in 
$M_m(A)$ and thus has the form $h=2p-1$ for a projection $p$. It follows that
\begin{eqnarray*}
\nu_t(x) \nu_t^{-1}(e_m) (1\hot \kg)_m \nu_t(e_m) \nu_t^{-1}(x) 
& = &
\cos(\pi t (h-1))(1\hot\rho) + \sin(\pi t (h-1))(e\hot 1) \\
& = &
\cos(-2\pi t p^\perp)(1\hot\rho) + \sin(-2\pi t p^\perp)(e\hot 1)
\end{eqnarray*}
As $1\hot\rho$ anticommutes with $e\hot 1$ the loop 
$t\mapsto  \cos(-2\pi t p^\perp)(1\hot\rho) + \sin(-2\pi t p^\perp)(e\hot 1)$ is osu-homotopic 
to 
\begin{equation}\label{eq-Bott}
Y(t) : = \cos(-2\pi t p^\perp)(e\hot 1) - \sin(-2\pi t p^\perp)(1\hot \rho)
\end{equation}
(see also \cite{Hu}). Furthermore, by 
identifying $e\hot 1$ and $1\hot\rho$  with $\sigma_x,\sigma_y\in M_2(\C)$ we obtain the expression
\begin{equation}\label{eq-BS}
\beta([x]) = [Y],\quad Y(t) = \begin{pmatrix} 0 & e^{-2\pi i t p^\perp} \\
e^{2\pi i t p^\perp} & 0 \end{pmatrix}.
\end{equation}
The upper right corner corresponds to the usual formula for the Bott map applied to $[p^\perp]$.

\section{Cyclic cohomology for graded \RA s}\label{sec-cyclic}
\subsection{Definition for graded algebras}
The definition of cyclic cohomology for a graded algebra $(\Aa,\gamma)$ can be found in \cite{Kassel86}. 
It is a straightforward generalisation of the ungraded case and follows naturally
if one takes into account the sign rule for the transposition of two elements in the graded tensor product: the sign of the transposition $a\hot b\mapsto b\hot a$ is $(-1)^{1+|a||b|}$.  The cyclic permutation 
$\lambda_{(n)}:\Aa^{\hot^{n}}\to \Aa^{\hot^{n}}$ must then be used with an extra sign coming from the transpositions:
$$  \lambda_{(n)}(a_0\hot\cdots \hot a_{n}) = (-1)^{n+|a_n|(\sum_{i=0}^{n-1}|a_i|)}
  (a_n\hot a_0 \hot \cdots a_{n-1}).$$
The Hochschild complex $(C^n,b)$ of $\Aa$ consists of 
the modules $C^n$ of linear maps 
$\xi:\Aa^{\hot^{n+1}}\to \C$ together with the boundary maps  
$$ 
b = \sum_{i=0}^{n+1} (-1)^i \delta_i$$
where $\delta_i:C^n\to C^{n+1}$ is given by 
$$\delta_i f(a_0\hot\cdots \hot a_{n+1}) = \xi(a_0\hot\cdots a_ia_{i+1}\cdots \hot a_{n+1}),\quad 
\delta_{n+1} = (-1)^{n+1}\lambda\delta_0 $$
with $\lambda:C^n\to C^n, \lambda\xi = \xi\circ\lambda_{(n)}$.
The cyclic cohomology is the cohomology of the subcomplex $(C^n_\lambda,b)$ of the Hochschild complex of cyclic cochains, i.e.\ $\xi\in C^n$ which satisfy $\lambda\xi = \xi$.
A cyclic cochain of $C^n$ which satisfies $b\xi = 0$ is called a cyclic cocycle of dimension $n$.
\subsection{Cycles and their characters}
The cyclic cohomology of algebras can be described by means of characters of cycles. 
Recall from \cite{ConnesBook}
that an $n$-dimensional chain
over an ungraded algebra $\Aa$ is a triple $(\Omega,d,\int)$, a $\Z$-graded algebra $\Omega=\oplus_{k\in\Z} \Omega_k$ with differential 
$d:\Omega_k\to \Omega_{k+1}$, an algebra homomorphism $\varphi:\Aa\to \Omega_0$, and a graded trace $\int:\Omega_n\to\C$ (or $\R$). In the intereste of clarity we denote this trace also by $\int_\Aa$.
A chain is a called a cycle if the trace is closed, that is, it vanishes on the image of $d$. 
We adapt this definition to graded algebras 
by requiring in addition that $\Omega$ is $\Z\times \Z_2$-graded,
the differential $d$ has degree $(1,0)$, $\varphi$ preserves the $\Z_2$-grading and $\int$ is graded cyclic in the sense that 
\begin{equation}\label{eq-graded}
\int \omega\omega' = (-1)^{|\omega|_{\Z}|\omega'|_{\Z}+|\omega||\omega'|}\int \omega'\omega
\end{equation} 
where $|\omega|_{\Z}$ is the $\Z$-degree and $|\omega|$ is the 
$\Z_2$-degree of $\omega$.
With these additional requirements $(\Omega,d,\int)$ is called a chain, or cycle resp., over the graded algebra $\Aa$. 
In the applications below the graded trace $\int$ is non-trivial only on elements of a fixed $\Z_2$-degree $\nu$. We call this $\nu$ the parity of the chain.

The character $\xi$ of the 
chain $(\Omega,d,\int)$ is defined to be 
$$\xi(a_0\hot\cdots\hot a_n) := \int \varphi(a_0)d\varphi(a_1)\cdots d\varphi(a_n).$$

As in the ungraded case \cite{ConnesBook}[Chap.~III.1.$\alpha$, Prop.~4] one shows that the character of a cycle is a cyclic cocycle of dimension $n$ and that conversely, any $n$-dimensional cyclic cocycle arrises from an $n$-dimensional cycle in the above way. 

The cyclic cohomology of \CA s is too poor for the applications in solid state physics which we have in mind. One way to overcome this problem is to consider characters of cycles whose 
differential $d$ and graded trace $\int$ are perhaps only densely defined but which still 
have a well-defined pairing with the $K$-groups of $A$. Here the holomorphic functional calculus will play a role.

Recall  that a unital normed complex algebra $\Aa$ is closed under holomorphic functional calculus if for any $x\in\Aa$ and any function $f$ which is holomorphic in a neighbourhood of the spectrum of $x$ (the set of $\lambda\in\C$ such that $x-\lambda 1$ is not invertible in the completion of $\Aa$) we have $f(x)\in\Aa$. If $\Aa$ is not unital then we say that it is  closed under holomorphic functional calculus if this is the case for its unitization $\Aa^+$. Such algebras are called local Banach algebras \cite{Bla}. This has the following consequences. 

\begin{lemma}\label{lem-HFC}
Let $(\Aa,\gamma) $ be a unital normed complex graded algebra whose even part is closed under holomorphic functional calculus. If $x,y \in \Ub(\Aa,\gamma)$ satisfy $\|x-y\|<2\|x\|^{-1}$ then there exists a smooth path $[0,1]\ni t\mapsto x(t)\in \Ub(\Aa,\gamma)$ with $x(0)=x$ and $x(1)=y$. Moreover, if $\Aa$ is a dense subalgebra of a \CA\ and $x$ and $y$ self-adjoint then $x(t)$ can be chosen in $\Us(\Aa,\gamma)$ and 
$\|x(t) - x\|\leq C(\|x-y\|)$ for all $t\in [0,1]$ where $C:\R^+\to\R^+$ is a continuous function with $C(0)=0$. 
\end{lemma}
\begin{proof} We follow \cite{vanDaele1}[Prop.~3.2] to see that $v=\frac12(1+yx)$ satisfies
$\|1- v\|<1$, and hence is invertible in the completion of $\Aa$, and that $vxv^{-1} = y$. Furthermore $\|1- v\|<1$ implies that the spectrum of $v$ lies in the domain of the analytic extension to $\{z\in\C:\Re z>0\}$ of the natural logarithm and so 
$v^t = \exp(t\log(\frac12(1+yx)))$ is an element of $\Aa$ for all $t$. Hence $x(t) = v^t x v^{-t}$ is a path in $\Aa$ linking $x=x(0)$ to $y=x(1)$.  
Clearly the path is arbitrarily many times differentiable in $\Aa$. 

We now assume that $x$ and $y$ are self-adjoint. This implies that they have norm $1$ and that $v$ is a normal element. Since $\Aa$ is closed under polar decomposition \cite{Bla} $\hat x(t) = x(t) (x(t)^* x(t))^{-\frac12}$ belongs to $\Aa$ and so defines a smooth path in $\Us(\Aa,\gamma)$ joining $x$ with $y$.  
Let $b=\|x-y\|$. Then $\|1-v\|\leq \frac{b}2$ and hence $\|v^{-1}\|\leq \frac1{1- \frac{b}2}$. Hence 
$\|v^txv^{-t}-x\| \leq 2\|v^t-1\|\|v^{-t}\| \leq \frac{b}{1- \frac{b}2}$ for all $t\in [0,1]$.
Using $\|x^*(t)x(t)-1\| = \|(x^*(t)-x(t))x(t)\|\leq \|(x^*(t)-x(t))\|\|x(t)\|$ we see that 
$\|(x^*(t)x(t))^\frac12-1\|$ tends to $0$ if $b$ tends to $0$. Thus $\|\hat x(t)-x\|$ tends to $0$ if $b$ tends to $0$.
\end{proof}

\begin{cor}\label{cor-HFC}
Let $(A,\gamma) $ be a balanced complex graded \CA\ and $\Aa$ a dense $*$-subalgebra 
 whose even part is closed under holomorphic functional calculus. Assume furthermore that within distance $\frac12$ of any \osu\ of $A$ there is an \osu\ of $\Aa$. If two \osu s $x,y$ of $\Aa$ are homotopic in $\Us(A,\gamma)$ then there exists a continuous, piecewise continuously differentiable path $x(t)\in \Us(\Aa,\gamma)$ with $x(0)=x$ and $x(1)=y$.
\end{cor}
\begin{proof}
If two \osu s $x,y$ of $\Aa$ are homotopic in $\Us(A,\gamma)$ then there is a finite collection $(x_i)_{i=0,\cdots,N}$ of \osu s in $A$ such that $x=x_0$, $y=x_N$ and $\|x_i-x_{i+1}\|<1$. By the assumption we may move the $x_i$ a bit so that they are \osu s of $\Aa$ and    
$\|x_i-x_{i+1}\|<2$. Now the result follows from the last lemma.
\end{proof}

\begin{definition}\label{def-unb}
An $n$-dimensional chain $(\Omega,d,\int)$ over a graded \CA\ $A$ with domain algebra 
$\Aa$ is a $\Z\times \Z_2$-graded algebra $\Omega$ with a graded $*$-homomorphism $\varphi:A\to\Omega_0$, a densely defined differential $d$ of degree $(1,0)$, a densely defined linear functional $\int:\Omega_n\to \C$ and a dense $*$-subalgebra $\Aa$ of $A$, referred to as the domain algebra of the chain, such that 
\begin{enumerate}
\item[(C1)] $\varphi(\Aa)$ lies in the domain of $d$. 
\item[(C2)] $\varphi(\Aa) (d\varphi(\Aa))^n$ lies in the domain of $\int$ and  $\int$ is graded cyclic in the sense of (\ref{eq-graded}) on $\varphi(\Aa) (d\varphi(\Aa))^n$. 
If the chain is a cycle we require $\int$ to vanish on $(d\varphi(\Aa))^n$.
\item[(C3)] The even part of $\Aa\hot\C l_k$ is closed under functional holomorphic calculus, for all $k\geq 0$.
\item[(C4)] In any neighbourhood of an \osu\ of $M_n(A\hot \C l_k)$ there is an \osu\ of $M_n(\Aa\hot \C l_k)$, for all $n\geq 1$ and $k\geq 0$.
\end{enumerate}

If $A$ is balanced we require in addition that
\begin{enumerate}
\item[(C5)] $\Aa$ contains an \osu\ $e$ such that $de=0$.
\end{enumerate}
\end{definition}


The character of such a cycle defines a cyclic cocycle over the domain algebra $\Aa$. It is similar to what Connes calls a higher trace in \cite{ConnesBook} for ungraded algebras, but instead of requiring the norm estimates of \cite{ConnesBook}[Chap.~III.6.$\alpha$, Def.~11] (see also \cite{KS2004})  we require directly closedness under functional holomorphic calculus.

Note that if $A$ is balanced then by Lemma~\ref{lem-iso} the conditions (C3) and (C4) will hold for all positive $k$ if  they hold for $k=0,1$. If $A$ is unital then $A\hot \C l_1$ is balanced and hence conditions (C3) and (C4) will hold for all positive $k$ if  they hold for $k=0,1,2$. We will see below that for trivially graded $A$ condition (C4) follows from condition (C3). It would be interesting to know whether this is also the case for general graded algebras.

We simplify our notation by surpressing the homomorphism $\varphi$, which should not create confusion, as for all our cycles below $\varphi$ is injective.
 
\subsection{Extension of cycles to $M_m(A)\hot \C l_k$}
\renewcommand{\jmath}{\kappa} 
Kassel establishes a K\"unneth formula for the cyclic cohomology of graded algebras. It implies that $HC^n( A\hot \C l_k)$ contains $\bigoplus_{i+j=n} HC^i(A)\otimes HC^j(\C l_k)$. He shows moreover that $HC^0(\C l_k)\cong \C$ and that the (up to normalisation unique)
linear map $\jmath_k:\C l_k\to\C$ which is non-zero only on the product of all generators $\rho_1\cdots\rho_k$
provides a generator for $HC^0(\C l_k)$. 
All further elements in $HC^{ev}(\C l_k)$
are images of $\jmath$ under Connes' $S$-operator and $HC^{odd}(\C l_k)$ vanishes.
We use the cup product with the above generator 
on the level of chains to extend characters from $A$ to $A\hot \C l_k$.
We normalise the generator as follows. Let  
$$ \kappa = \sqrt{2} e^{\frac{\pi i}4},$$
a square root of $2i$, and
$$\jmath_k (\rho_1\cdots \rho_k) = \kappa^k,\quad \jmath_k (\rho_{i_1}\cdots \rho_{i_j}) = 0 \:\:\mbox{if } j<k.$$
Note that $\jmath_k$ is a graded trace on $\C l_k$, i.e.\ $\jmath(c_1c_2) = (-1)^{|c_1||c_2|} \jmath(c_2c_1)$. Thus $(\C l_k,0,\jmath_k)$ is a cycle over $\C$. 
Furthermore, $(M_m(\C),0,\Tr_m)$ is a cycle over  $M_m(\C)$ where $\Tr_m$ is the standard trace on $m\times m$ matrices. 
We extend chains over $A$ to chains over $M_m(A)\hot \C l_k$ by taking their product with the above cycles. 
\begin{definition} \label{def-j} Let $\Aa$ be a graded algebra and
$(\Omega,d,\int)$ a chain  over $\Aa$. The extension of this chain  to $M_m(\Aa)\hot\C l_k$ is the chain  $(M_m(\Omega)\hot\C l_k,d\otimes\id,\int\circ\Tr_m\circ\jmath_k)$. Here $d$ is extended to $M_m(\Omega)$ entrywise and
$\jmath_k(\omega\hot c) =  \jmath_k(c)\omega$ for 
$\omega\in M_m(\Omega)$ and $c\in\C l_k$.
\end{definition}
We denote the character of the extension by $\xi\#\Tr_m\#\jmath_k$ where
$\xi$ is the character of $(\Omega,d,\int)$, or simply also by $\xi\#\jmath_k$, or even $\xi$, as its entries make clear what the values for $m$ and $k$ are. We have $\jmath_k\#\jmath_l = \jmath_{k+l}$. 

Chains over a graded \CA\ $A$ with domain algebra $\Aa$ are extended similarily, the domain algebra for $A \hot \C l_k$ being $\Aa \hot \C l_k$.
 
\subsection{Connes pairing with van Daele $K$-groups}
We start by considering cycles on a balanced graded normed algebra $(\Aa,\gamma)$. 
In later applications this algebra will be the domain algebra of a graded \CA.

\begin{lemma} \label{lem-homotopy}
Let $(\Omega,d,\int)$ be an $n$-dimensional cycle over a graded normed algebra 
$(\Aa,\gamma)$. Suppose that there exists an element $e\in\Ub(\Aa,\gamma)$ which satisfies $de=0$. Let $t\mapsto x(t)$ be a continuously differentiable path in $\Ub(\Aa,\gamma)$. Then
$$ \int (x(t)-e)(dx(t))^n$$ does not depend on $t$.
\end{lemma}

\begin{proof} 
Given any derivation $\delta$ 
the identity $x^2=1$ implies that
$(\delta x)^i x = (-1)^i x (\delta x)^i$ for all $i\in \NM$. Let $z=x-e$. Clearly 
$\int (x-e)(dx)^n=\int z(dz)^n$.
Using $\dot{x} = -x\dot{x} x$ and
the graded cyclicity of $\int$ we get
$$ \int\dot{z}(dx)^n = - \int x\dot{x}x(dx)^n = - (-1)^{|x| |\dot{x}x(dx)^n|}
\int \dot{x} x (dx)^n x = - \int\dot{z}(dx)^n$$
%
as $|x|_\Z=0$. Hence $ \int\dot{z}(dx)^n$ must vanish.
Furthermore, $zdz + (dz)z$ is a total derivative. Hence 
$$z(dz)^\nu d\dot{z} = (-1)^\nu (dz)^\nu z d\dot{z} +R_1  = (-1)^{\nu+1} (dz)^{\nu+1} \dot{z} + R_1+R_2$$
where $R_1$ and $R_2$ are total derivatives and hence vanish under the trace $\int$. Thus 
$$\int z(dz)^\nu d\dot{z} (dz)^{n-\nu-1} = \epsilon \int (dz)^{\nu+1} \dot{z} (dz)^{n-\nu-1} = \epsilon' \int \dot{z}(dz)^{n}$$
for certain $\epsilon,\epsilon'\in\{\pm 1\}$.  
Hence the derivative of $\int (x(t)-e) (dx(t))^{n}$ w.r.t.\ $t$ vanishes. 
\end{proof}

\subsubsection{Pairing for balanced graded \CA s}
\begin{definition}\label{def-pairing} Let $(A,\gamma)$ be a unital balanced graded \CA\ and $(\Omega,d,\int)$ an $n$-dimensional cycle over $A$ with character $\xi$ and domain algebra $\Aa$. 
Let $e\in \Us(\Aa,\gamma)$ satisfy $de = 0$.
The pairing of $\xi$ with an element
 $[x]\in\DK_{e}(A)$ is defined to be
 $$\langle\xi,[x]\rangle = \int\Tr_m(x-e_m)(dx)^n$$
where we take a representative $x$ for $[x]$ which lies in $\Us(M_m(\Aa),\gamma_m)$.
\end{definition}

Cor.~\ref{cor-HFC} and Lemma~\ref{lem-homotopy} guarantee that the pairing is well-defined. Indeed, 
the assumptions on $\Aa$ assure that $(\Omega,d,\int)$ restricts to a bounded cycle over $\Aa$, that the hypothesis of Cor.~\ref{cor-HFC} are satisfied and that $[x]$ admits a representative in $\Us(M_m(\Aa),\gamma_m)$.

Since $(x_1\oplus x_2) (d(x_1\oplus x_2))^n = 
x_1 (d x_1)^n\oplus x_2 (d x_2)^n$  the map $V_{e}(A)\ni [x]\mapsto \langle\xi,[x]\rangle\in\C$ is a homomorphism of semi-groups. Since the pairing with the  class of $e$ is $0$ the homomorphism induces a homomorphism of groups $\DK_{e}(A)\ni [x]\mapsto \langle\xi,[x]\rangle\in\C$.
Note that
$$\int\Tr_m(x-e_m)(dx)^n=\int\Tr_m x(dx)^n$$
provided $n>0$. 
In the case $n=0$ the pairing is a priori only independent of the choice of basepoint $e$ if it is osu-homotopic to its negative.

The above formulation of the pairing does not make reference to whether we work with complex or with real \CA s. We find it very convenient, however, to work with \RA s where real \CA s appear as subalgebras of elements which are invariant under the real structure. 
The complex pairing defined for the \RA\ will then define by restriction to real 
homotopy classes a pairing with the subalgebra of real elements. 

To define pairings with higher $K$-groups or for trivially graded algebras we consider tensor products with $\C l_k$ and extend the cycle as in Def.~\ref{def-j}. We remark that $k$ must be strictly larger than $1$ in case $A$ is not balanced graded.
\begin{definition}\label{def-pairing-higher} Let $(A,\gamma)$ be a (possibly trivially) graded unital \CA\ and $(\Omega,d,\int)$ an $n$-dimensional (possibly unbounded) cycle over $A$ with character $\xi$ and domain algebra $\Aa$.  
Let $e\in \Aa\hot \C l_k$ be an \osu\ which satisfies $de = 0$.
The pairing of $\xi$ with an element
 $[x]\in\DK_{e}(A\hot \C l_k)\cong K_{1-k}(A)$ is, by definition, the pairing of the extension $\xi\#\jmath_k$ with $[x]$, 
 $$\langle\xi,[x]\rangle = \int\Tr_m\jmath_k((x-e_m)(dx)^n)$$
for $x\in \Us(M_m(\Aa\hot\C l_k),(\gamma\ot\st)_m)$.
\end{definition}
For convenience we extend this definition to invertible odd self-adjoint elements $x$ by setting $\langle\xi,[x]\rangle = \langle\xi,[\hat x]\rangle$ where $\hat x = x |x|^{-1}$ is the spectral flattening of $x$.

If $A$ is a balanced graded algebra then by Lemma~\ref{lem-iso}
$(A\hot\C l_2,\gamma\ot\st)$ is isomorphic to $(M_2(A),\gamma_2)$. 
We therefore have {\it a priori} two ways to pair an element of the $K$-group with a character: one involves the formula with $k$ and the other with $k+2$. That these two ways yield the same answer is the following result.

\begin{lemma} \label{lem-j-tr}
Let $\Aa$ be a balanced graded algebra and $e$ an \osi\ in $\Aa$. Let $(\Omega,d,\int)$ be an $n$-dimensional cycle over $\Aa$ and $\omega\in \Omega_n\hot\C l_2$. Then 
$$ \int\jmath_2(\omega) = \int \Tr_2(\psi_e(\omega))$$
where $\psi_e: \Omega_n\hot\C l_2 \to M_2(\Omega)$ is defined as in Lemma~\ref{lem-iso}
 with $A$ replaced by $\Omega_n$ and
$\Tr_2$ is the matrix trace.
\end{lemma}
\begin{proof}
Expand $\omega = \omega_0\hot 1 +  \omega_1\hot \sigma_x+  \omega_2\hot \sigma_y+  \omega_3\hot \sigma_z$ with $\omega_i\in\Omega_n$. Then
$\jmath_2(\omega) = -i\kappa^2 \omega_3=2\omega_3$. On the other hand
\begin{eqnarray*}
\Tr_2(\psi_e(\omega)) &=& 
\Tr_2 \begin{pmatrix} \omega_0+\omega_3 & (\omega_1-i\omega_2) e\\
e ((-1)^{|\omega_1|}\omega_1 +i (-1)^{|\omega_2|}\omega_2) & e((-1)^{|\omega_0|}\omega_0 - (-1)^{|\omega_3|}\omega_3)e\end{pmatrix}\\
&=& \omega_0+(-1)^{|\omega_0|}e\omega_0e + \omega_3 - (-1)^{|\omega_3|}e\omega_3e
\end{eqnarray*}
Now $\int e\omega e = (-1)^{|\omega|+1} \int \omega$, by graded cyclicity. Hence 
$$\int\Tr_2(\psi_e(\omega)) = 2 \int \omega_3.$$
\end{proof}
\begin{cor} \label{cor-j-tr}
Let $\Aa$ be a balanced graded algebra and $e$ an \osi\ in $\Aa\hot \C l_k$. Let $\xi$ be the character of an $n$-dimensional cycle over $\Aa$ and $x\in M_m(\Aa)\hot\C l_2\hot \C l_k$ an \osi. Then  
$$\int\Tr_m\jmath_{2+k} ((x-\tilde e_m)(dx)^n)
=  \int\Tr_{m+2} \jmath_k((\psi_e(x)-e_{2m})(d\psi_e(x))^n).$$
where $\tilde e\in  \Aa\hot\C l_2\hot \C l_k$ corresponds to $e\hot\sigma_z\in \Aa\hot\C l_2\hot \C l_k$ under the isomorphisms $\C l_2\hot \C l_k\to \C l_k\hot \C l_2$.
\end{cor}
\begin{proof} We apply Lemma~\ref{lem-j-tr} to $a = (x-\tilde e_m)(dx)^n\in \Omega_n\hot \C l_{2+k}$. The result follows as
$\psi_e$  commutes with $d$ and $\psi_e(\tilde e)=e_2$.
\end{proof}

\subsubsection{Pairing for nonunital algebras}
If $A$ is non-unital then its van Daele $K$-group is a subgroup of the $K$-group of the unitization $A^+$ of $A$. 
More precisely we consider the exact sequence of graded \CA s
$$0\to (A\hot \C l_2,\gamma\ot\st) \to (A^+\hot \C l_2,\gamma^+\ot\st)
\stackrel{q}\to
(\C\ot \C l_2,\id\ot\st)\to 0$$
and take $e=1\hot \sigma_x$ as base point for $A^+\hot \C l_2$ and $\C\ot \C l_2$.
Then the elements of $\DK(A)$ are by definition the homotopy classes $[x]$ of \osu's 
$x\in  M_m(A^+)\hot \C l_2$ such that $q(x)$ is homotopic to $e_m$. 
Let now $\xi$ be a  character of a cycle over $(A,\gamma)$.
If $\xi$ extends to $\Aa^+$, the unitization of the domain algebra $\Aa$, 
then we may directly apply Def.~\ref{def-pairing}
for its pairing. If $\xi$ does not extend to $\Aa^+$ then the formula makes only sense if $x-e_m\in \ker q =M_m(\Aa)\hot \C l_2$. Van Daele proves in 
\cite{vanDaele1}[Prop.~3.7] that any $[x]\in\DK(A)$ has a representive $x$ which satisfies 
$x-e_m\in\ker q = M_m(A)\hot \C l_2$. A closer look at his proof (based again on the construction of a holomorphic logarithm as in the proof of Lemma~\ref{lem-HFC}) shows that any $[x]\in\DK(A)$ 
has a representive which even satisfies $x-e_m\in M_m(\Aa)\hot \C l_2$.
We thus define the pairing of $\xi$ with $[x]$ by the same formula as in Def.~\ref{def-pairing}
$\langle\xi,[x]\rangle = \int\Tr(x-e_m)(dx)^n$ but require that 
$x$ is such that $x-e_m\in M_m(\Aa)\hot \C l_2$. 
The homotopy invariance of this pairing can be shown 
in two steps: For differentiable paths in $\Us( M_m(\Aa)\hot \C l_2)$ of the form $x(t) - e_m$ we conclude
as in Lemma~\ref{lem-homotopy} that the derivative of $\int\Tr(x(t)-e_m)(dx(t))^n$ w.r.t.\ $t$ vanishes. Indeed, the argument using 
the graded cyclicity can be employed as 
$\dot{x} x (dx)^n$ and $\dot{z}(dz)^{n-\nu-1}$ lie in $ M_m(\Aa)\hot \C l_2$, and furthermore the total derivatives $R_1$, $R_2$ are derivatives of elements from $ M_m(\Aa)\hot \C l_2$. Now in a second step we need to make sure that if $y$ is osu-homotopic to $x$ and also satisfies $y-e\in  M_m(\Aa)\hot \C l_2$ then we can  even find a homotopy
$x(t)$ from $x$ to $y$ such that $x(t)-e$ remains in $M_m(\Aa)\hot \C l_2$ for all $t$. To see this, let $x(t)$ be a continuously differentiable homotopy between $x(0)=x$ and $x(1)=y$ in $\Us( M_m(\Aa^+)\hot \C l_2)$ and consider the path $c(t)=q(x(t))$ in 
$\Us(M_m(\C)\hot \C l_2)$. The path $c(t)$ can be represented as $\begin{pmatrix}0 & U(t)^* \\ U(t) & 0  
\end{pmatrix}$ where $U(t)$ is a continuous path of unitaries in $M_m(\C)$ which is $1$ at $t=0$ and $t=1$. 
Let $W(t) = \begin{pmatrix}U(t)^{\frac12} & 0 \\ 0&U(t)^{-\frac12}  
\end{pmatrix}$ where $U(t)^{\frac12}$ a square root of $U(t)$ which is continuous in $t$ for $0\leq t \leq 1$ and equal to $1$ at $t=0$. It follows that $W(t)c(t)W^*(t) = \begin{pmatrix}0 & 1 \\ 1 & 0  
\end{pmatrix}$. Moreover, $W(1)$ commutes with $c(1)$ and has eigenvalues $\pm 1$. Since $U_m(\C)$ is contractible we can find a path of unitaries $V(t)\in M_{2m}(\C)$, $t\in[1,2]$ which commutes with $c(1)$ and connects $W(1)$ to $1\in M_{2m}(\C)$. Let
$$\tilde x(t) = \left\{\begin{array}{cc}
s(W(t)) x(t) s(W(t))^* &\mbox{for } 0\leq t\leq 1\\
s(V(t)) x(1) s(V(t))^* &\mbox{for } 1\leq t\leq 2
\end{array}\right.
$$
where $s:\C\ot \C l_2\to A^+\hot \C l_2$ be a section, i.e.\ $q\circ s = \id$.
Then $\tilde x(t)$ is an osu-homotopy between $x$ and $y$ which satisfies $q(\tilde x(t)) = e_m$ for all $t\in[0,2]$.

\subsubsection{Pairing for trivially graded \CA s} 
In the context of trivially graded $A$ the above pairing is an adaptation of Connes pairing to van Daele's formulation of (complex) $K$-theory, as we shall see now. We begin with a couple of remarks.

If $(\Omega,d,\int)$ is a cycle over a trivially graded algebra then only the even part of 
$\Omega$ enters into the definition of the character and so we may assume without loss of generality that $\Omega$ is trivially graded. 

If $A$ is a graded \CA\ with 
domain algebra $\Aa$ then the condition that $\Aa$ is closed under holomorphic functional calculus implies that the even parts of $\Aa\ot\C l_k$ are for all positive $k$ also closed under holomorphic functional calculus. This is trivially the case for $k=1$, follows for $k=2$ from the description of the even part as the diagonal matrices of $M_2(\Aa)$, and for $k>2$ from the fact that $A\ot\C l_1$ (or $A^+\ot\C l_1$ if $A$ is not unital) is balanced.

Furthermore, closedness under holomorphic functional calculus of $\Aa$ implies (C4). Indeed, for $k=0$ the condition (C4) is empty. For $k=1$ any \osu\ of $A\ot \C l_1$ has the form $(2p-1)\ot\kg$ where $p$ is a projection. 
Arbitrarily close to $p$ we can find an element $p'\in\Aa $ which is perhaps not a projection, but its spectrum lies in a small neighbourhood $U$ of the set $\{0,1\}$.  
There exists a holomorphic function $f$ on $U$ which is $1$ near $1$ and $0$ near $0$ and hence $f(p')$ is a projection and $f(p) = p$. By continuity of  $f$, $f(p')$ is close to $f(p)$ and hence $(2p-1)\ot\kg$ osu-homotopic to $(2p'-1)\ot\kg$ in $A\hot\C l_1$.
Finally, in the case $k=2$ any \osu\ of $A\ot \C l_2$ has the form $\begin{pmatrix}0 & U^* \\ U & 0 \end{pmatrix}$ for some unitary $U\in A$. Since the invertible elements of $A$ are open we find an invertible element $Q\in\Aa$ close to $U$. Since $\Aa$ is closed under functional calculus we can polar decompose $Q$ in $\Aa$ to see that $Q$ is close to a unitary $U'$ in $\Aa$. It follows that $U'$ is homotopic to $U$ in the set of unitaries of $A$. 

To summarize, for trivially graded \CA s we can replace conditions (C3) and (C4) in Def.~\ref{def-unb} by the single condition that $\Aa$ is closed under holomorphic functional calculus.
It seems in interesting question to ask whether this is also true for non-trivially graded \CA s.
\bigskip

The following result is specific to trivially graded algebras. 
\begin{lemma} \label{lem-triv}
Let $(\Omega,d,\int)$ be an $n$-dimensional cycle over a trivially graded normed algebra $A$.
Let $e\in \Ub(A\ot \C l_k)$ with $de=0$. 
 If $k+n$ is even then $\int\Tr\jmath((x-e_m)(dx)^n)=0$ for all $x \in\Ub(A\ot\C l_k,\id\ot\st)$.
\end{lemma}
\begin{proof}
We can write $(x-e_m)(dx)^n=\omega\ot \rho_1\cdots \rho_k +R$ where $R$ belongs to the kernel of $\jmath_k$.
Now the $\Z_2$-degree of $(x-e_m)(dx)^n$ is $n+1$ (mod $2$) and that of $\omega\ot \rho_1\cdots \rho_k$ equal to $k$, as $\Omega$ is trivially $\Z_2$-graded (one can choose it that way). It follows that $n+1+k$ must be even or $\omega=0$.
\end{proof}
We now show that, for a trivially graded unital \CA\ $A$, the pairing of the character of an $n=2j+1$-dimensional cycle with 
$KU_1(A)=K_1(A,\id)\cong \DK_{1\ot \sigma_x}(A\ot\C l_2,\id\ot\st)$  corresponds to the usual pairing as defined by Connes \cite{ConnesBook}.
Any \osu\  $x\in M_2(M_m(A))\cong M_m(A)\ot \C l_2$ is of the form 
$$ x = \begin{pmatrix} 0 & U^* \\ U & 0 \end{pmatrix}$$
for some unitary $U \in M_m(A)$. 
Let $e = 1\ot \sigma_x = \begin{pmatrix} 0 & 1 \\ 1 & 0 \end{pmatrix}$. Then
\begin{eqnarray*} 
(x-e_m)dx(dx)^{n-1} &=&  \begin{pmatrix} (U^*-1) dU & 0 \\  0& (U-1)dU^*  \end{pmatrix}\begin{pmatrix} dU^* dU & 0 \\  0& dUdU^*  \end{pmatrix}^j \\
&=& 
\frac12\big((U^*-1)dU(dU^* dU)^j - (U-1)dU^*(dU dU^*)^j\big)\ot \sigma_z + R
\end{eqnarray*}
where $R$ lies in the kernel of $\jmath_2$. $U\mapsto \int\Tr (U^*-1)dU(dU^* dU)^j$ is the usual pairing between unitaries and the character of the cycle and known to induce a group homomorphism on $KU_1(A)$ \cite{ConnesBook}.
In particular $\int\Tr (U-1)dU^*(dU dU^*)^j = -\int\Tr (U^*-1)dU(dU^* dU)^j$ and, since $\jmath_2(\sigma_z) = 2$ we have 
$$\int\Tr_m\jmath_2((x-e_m)(dx)^n) =  2 \int\Tr_m (U^*-1)dU(dU^* dU)^j.$$

Next we show that the pairing of the character of an even dimensional cycle with $KU_0(A)=K_0(A,\id)$ corresponds to the usual pairing as defined by Connes \cite{ConnesBook} if $A$ is trivially graded unital.
Recall that $K_0(A,\id) \cong \DK_{e}(A\ot \C l_1,\st)$ where we may take 
$e=-1\ot \kg$ ($\kg$ the generator of $\C l_1$) as basepoint. 
Since $e$ is not homotopic to its negative $\DK_{e}(A\ot \C l_1,\st)$ is the Grothendieck group of the semigroup defined by homotopy classes of \osu s. Any \osu\ has the form
$x = h\ot \kg$ where $h\in M_m(A)$ is a selfadjoint unitary.
The map $h\mapsto p=\frac{h+1}2$ 
induces an isomorphism between $\DK_e(A\ot \C l_1)$ and the standard picture for the $K_0$-group of $A$. Then $(x-e_m) (dx)^n = (h+1) (dh)^n \kg^{n+1}$ and hence  $\jmath_1((x-e_m) (dx)^n) =  \kappa 2p(2dp)^n$ provided $n$ is even and $0$ otherwise. Hence, if $n$ is even,
$$\int\Tr\jmath_1((x-e_m)(dx)^n) = \kappa 2^{n+1} \int\Tr p(dp)^n.$$

\newcommand{\gr}{\lambda}
\subsection{Compatibility with the $*$-structure and the real structure}
We now formulate compatibility conditions on chains or cycles of \RA s which will, in the context of pairings with their real subalgebras, lead to criteria under which the pairing can be non-trivial. 

\newcommand{\SA}{$(*,\rs)$-algebra}
\newcommand{\tosi}{sign}
A $(*,\rs)$-algebra is a $*$-algebra with a real structure $\rs$, that is, an anti-linear $*$-automorphism of order $2$.  
An $\R$-linear map $\varphi$ on a $(*,\rs)$-algebra  is a $(*,\rs)$-morphism 
if $\varphi(a^*) = (\varphi (a))^*$ and $\varphi( \rs(a)) = \rs(\varphi (a))$.
\begin{definition}
A chain $(\Omega,d,\int)$ over a $(*,\rs)$-algebra $A$ is called a $(*,\rs)$-chain of \tosi\ 
$\ss\in \{+1,-1\}$ if  $\Omega$ is a $(*,\rs)$-algebra, 
$\varphi:A\to \Omega_0$ a $(*,\rs)$-homomorphism, $d$ a $(*,\rs)$-derivation
and the graded trace satisfies
$$ \int \tilde\rs (\omega)^* = \ss \int \omega$$
where $\tilde \rs$ is the real structure on $\Omega$.
\end{definition}
Recall that the chain has parity $\nu$ if the graded trace vanishes on elements of parity $\nu+1$.
\newcommand{\re}{{\mathfrak e}}
\begin{prop}\label{cor-sp-graded}
Let $(\Omega,\tilde\rs,d,\int)$ be an $n$-dimensional $(*,\rs)$-chain over a graded \RA\ $(A,\rs)$.
Let $u,x,y,z\in A$ be $\rs$-invariant odd self-adjoint elements, $u$ being in addition unitary.
A necessary condition for $\int z (dx)^{n}$
to be non-zero is that the \tosi\ of the cycle is $(-1)^{n}$ and its parity is $n+1$. 
A necessary condition for $\int zu (du)^{n}$ and
$ \sum_{k=0}^n (-1)^{k} \int z (dx)^{k}y (dx)^{n-k}$ 
 to be non-zero is that the \tosi\ of the cycle is $-1$ and its parity is $n$. 
\end{prop}
\begin{proof}
The $\Z_2$-degree of $z(dx)^n$ is $ n+1$ (mod $2$) and thus, if the parity of the cycle is different from $n+1$ then $\int z (dx)^{n} = 0$. 
If the parity is $n+1$ then
$$\ss\int z(dx)^n= \int\tilde\rs(z (dx)^n)^* = \int (dx)^n z = (-1)^n \int z(dx)^n$$ 
from which we deduce the first claim. 

The $\Z_2$-degree of $zx(dx)^n$ is $n$ and thus if the parity of the cycle is $n$ then
$$ \int \tilde\rs((zx (dx)^n)^*) = 
(-1)^n \int x(dx)^n z = - \int zx (dx)^n$$ 
where we have used $xdx = - dx x$.  

The  $\Z_2$-degree  of $\sum_{k=0}^n (-1)^k z(dx)^k y (dx)^{n-k}$ is  $n$ and 
if the cycle has parity $n$ then
\begin{eqnarray*} \sum_{k=0}^n (-1)^k \int \tilde \rs \big( z(dx)^ky (dx)^{n-k}\big)^*
&=& \sum_{k=0}^n (-1)^k \int  (dx)^{n-k}y (dx)^{k}z \\
&=& \sum_{k=0}^n (-1)^{k+n+1} \int z (dx)^{n-k}y (dx)^{k}\\
&=& \sum_{k=0}^n (-1)^{k+1} \int z (dx)^{k}y (dx)^{n-k}
\end{eqnarray*}
from which we deduce the last claim.
\end{proof}

Recall that $\nn:\Z\to \Z$ is given by $\nn(n) = \lfloor{\frac{n}2}\rfloor$.
\begin{lemma}\label{lem-p-ext}
Consider a $(*,\rs)$-chain of parity $\nu$ and  
\tosi\ $\ss$ over $(A,\rs)$.
The extension to  $(A\hot\C l_{r+s},\rs\ot \rcl_{r,s})$ is a $(*,\rs)$-cycle of parity $\nu+r-s$ and
\tosi\ $(-1)^{\nn(r-s)+\nu (r-s)}\ss$.
\end{lemma}
\begin{proof} The parity of $\int\jmath_{r+s}$ is the sum of the parities of $\int$ and $\jmath_{r+s}$, hence equal to $\nu+r-s$.  
By definition of the real structures $\rcl_{r,s}$ on $\C l_{r+s}$ we find
$$ \rcl_{r,s}(\kg_1\cdots\kg_{r+s})^*  = (-1)^{\nn(r-s)} \kg_1\cdots\kg_{r+s} .$$
If the following expression is non-zero then $\omega$ must have parity $\nu$ and hence
\begin{eqnarray*}
\int \jmath_{r+s} (\tilde\rs(\omega)\hot \rcl_{r,s}(\kg_1\cdots\kg_{r+s}))^* &=&
 (-1)^{\nu(r+s)}\int \jmath_{r+s} (\tilde\rs(\omega)^*\hot \rcl_{r,s}(\kg_1\cdots\kg_{r+s})^*)\\
&=&   (-1)^{\nn(r-s)+\nu(r-s)} \ss {\int \jmath_{r+s}(\omega\otimes\kg_1\cdots\kg_{r+s})}
\end{eqnarray*} 
\end{proof}
\begin{cor}\label{cor-cond}
Let $(\Omega,\tilde\rs,d,\int)$ be an $n$-dimensional $(*,\rs)$-cycle over a graded \RA\ $(A,\rs)$ with parity $\nu$ and \tosi\ $\ss$.  A necessary condition for its character to pair non-trivially with
$\DK_e(A^\rs\hot Cl_{r,s})$ is that  
$\ss = (-1)^{\nn(r-s)+n(r-s+1)}$ and $\nu=n+1-r+s$.
\end{cor}
\begin{proof} By Lemma~\ref{lem-p-ext} the extension of the cycle to $A^\rs\hot Cl_{r,s}$ has 
parity $\nu+r-s$ and
\tosi\ $(-1)^{\nn(r-s)+\nu (r-s)}\ss$. The result follows now from Prop.~\ref{cor-sp-graded}, taking into account that $\nu(r-s) = n(r-s)$.\end{proof}
Note that $(-1)^{\nn(i)+n(i+1)}=-(-1)^{\nn(i+2)+n(i+2+1)}$. Hence if
$\langle \xi,\DK(A^\rs\hot Cl_{r,s}) \rangle\neq 0$ then $\langle \xi,\DK(A^\rs\hot Cl_{r',s'}) \rangle = 0$ where $r'+s'=r+s\pm2$. 
For trivially graded real algebras we get the following stronger conditions. 
\begin{cor} \label{cor-cond-triv}
Let $A$ be a trivially graded \RA. Necessary conditions for an $n$-dimensional $(*,r)$-cycle of parity $\nu$ and \tosi\ $\ss$ over $A$ to pair non-trivially with $KO_i(A^\rs)$ are that 
$\ss = (-1)^{\nn(1-i)+i}$ and $n+i$ is even. 
In particular, when these conditions are met the cycle pairs trivially with
$KO_{i+k}(A^\rs)$ for $k\notin 4\Z$.
\end{cor}

\subsection{Examples of pairings}\label{sec-Examples}
We discuss a variety of simple examples.
\subsubsection{Trace on $\C$}\label{sec-examples1} Consider the \RA\ $(\C,\cc)$ with trivial grading. One easily sees that 
$(\C,0,\id)$ is a $0$-dimensional $(*,\rs)$-cycle of even parity and \tosi\ $+1$ over $(\C,\cc)$. 
Since $\C$ is trivially graded only pairings with even $K$-group elements may be non-zero.

We pair with $KU_0(\C)=\DK_e (\C\ot\C l_1)$. 
Let $\kg$ be the generator of $\C l_1$ and choose $e = -1\ot \kg$ as base point. Any \osu\ is of the form $x = h\ot\kg$ for some self-adjoint unitary $h \in M_m(\C)$ which we can write as $h = 2p -1$ with some projection $p$. Thus the character $\xi$ of $(\C,0,\id)$ pairs as $\langle \xi , [x] \rangle = \kappa 2 \Tr(p)$.
Absorbing the constant $2\kappa$ we denote 
$$\ch_0 = \frac1{2\kappa}\xi$$ 
calling it the standard chern character on $\C$. Then $\langle\ch_0,KO_0(\R)\rangle = \Z$.

We consider pairings with the even $KO$-groups of the real subalgebra $\R$.
As $(-1)^{\nn(0)} = +1$ pairings with elements of $KO_i(\R)$ for $i=2$ and $6$
have to vanish by Cor.~\ref{cor-cond-triv}. This is to be expected as $KO_2(\R)$ is pure torsion and  $KO_{6}(\R)=0$.

For pairings with elements of $KO_0(\R)=\DK_{1\ot \kg}(\R\ot Cl_{1,0})$ the analysis is exactly as in the complex case $\langle \ch_0 , [x] \rangle = \Tr(p)$ where $x = (2p-1)\ot \kg$, the only difference being that $p$ is a projection in $M_m(\R)$. In particular, $\langle\ch_0,KO_0(\R)\rangle = \Z$.

Finally we consider pairing with elements from  $KO_{4}(\R)\cong\DK(\R\ot Cl_{0,3})$.  Let $e = 1\hot \kg \in \C l_2\hot \C l_1$. By Lemma~\ref{lem-iso} the isomorphism 
$\psi_e:(\C l_2\hot \C l_1,\st\ot\st)\to (M_2(\C)\ot\C l_1,\id_2\ot\st)$ intertwines the real structure $\rcl_{0,2}\ot\rcl_{0,1}$ with the real structure $\rh\ot \rcl_{1,0}$ where $\rh = \Ad_{\sigma_y}\circ \cc$. Thus $KO_{4}(\R)\cong\DK_{1\otimes \rho}(\HM\ot Cl_{1,0})$ 
where $\HM=M_2(\C)^\rh$ is the (trivially graded) algebra of quaternions.
The pairing is given now by $\langle \ch_0, [x]\rangle = \Tr(P)$ where $P = \frac12 (\psi_e(x)+1)$ is a projection in $M_{m}(\HM)$. $M_m(\HM)$ contains only projections of even rank as any eigenvalue of a self-adjoint element in $M_{2m}(\C)$ which is invariant under $\rh_m$ has to be evenly degenerate (this is Kramer's degeneracy in physics). Thus $\langle\ch_0,KO_4(\R)\rangle = 2\Z$.

\subsubsection{Winding number cycle}
\newcommand{\Int}{\int_{[0,1]}}
A simple (trivially graded) example of a chain is given by $(\Omega([0,1]),d_{[0,1]},\Int)$, where $\Omega([0,1])$ are the differential forms on the closed interval $[0,1]$, $d_{[0,1]}$ the exterior derivative, and $\Int$ the integral over $1$-forms, normalized so as $\Int dt = 1$. It is a $1$-dimensional chain over $C([0,1])$ with domain algebra $C^1([0,1])$.
On the subalgebra 
of functions which vanish satisfy $f(1)=f(0)$ the integral is closed and hence the chain restricts to a cycle over $C(S^1)$. 
It can pair non-trivially only with odd $K$-group elements.

We consider its pairing with $KU_1(C(S^1)) = \DK_{\sigma_x}(C(S^1)\ot \C l_2)$.
Upon writing an \osu\ $x\in C(S^1)\ot \C l_2$ as $x = \begin{pmatrix} 0 & U^* \\ U & 0\end{pmatrix}$, with some unitary $U(t)$ in $M_m(\C)$, we obtain for the pairing of the character $\xi$ of the above cycle with $[x]$, 
$$\langle \xi,[x]\rangle = \Int \Tr_m\jmath_2 ((x-e_m)dx)=  2 \int_0^{1} \Tr_m((U(t)^*-1) U'(t)) dt.$$
We recognise this as $4\pi i$ times the winding number of the determinant of $U$ which is why we normalize $\Ww:= \frac{1}{2\pi \kappa^2} \xi$ and call
the above cycle the winding number cycle. As a result 
$\langle \Ww, KU_1(C(S^1)) \rangle =\Z$.
Since $C(S^1)$ is the unitization of the suspension $S\C=C_0((0,1),\C)$ of $\C$ the above cycle is also a cycle for $S\C$.

There are two important real structures on $C(S^1)$, point-wise complex conjugation $\cc(f) = \overline{f}$ and the composition of the latter with the automorphism induced by a reflection
 $t\mapsto 1-t$, $\cfS(f)(t) = \overline{f(1-t)}$. They both restrict to the suspension $S\C$ and
 the real subalgebra of $(S\C,\cc)$ is the (standard) suspension $S\R$ of $\R$, while 
 the real subalgebra of $(S\C,\cfS)$ is, by definition, the dual  suspension $\tilde S\R$
of $\R$.

Let us consider first the real structure $\cc$ given by point-wise complex conjugation. It is rather straighforward to see that by putting on $\Omega([0,1])$ the real structure given by point-wise complex conjugation the cycle becomes a $(*,r)$ cycle of \tosi\ $+1$. 
We consider its pairings with elements of the real $K$-groups $KO_{i}(C(S^1)^\cc)$. 
There are four cases to consider, and 
Cor.~\ref{cor-cond-triv} predicts that only two are potentially non-trivial, namely those for which 
$(-1)^{\nn(1-i)} = -1$, that is, $i=-1$ and $i=3$. 
This can be understood on a more elementary level if one considers the characterisation of the $KO$-group elements via constraints on the unitary $U$ \cite{Kel1,Loring}: If $i=1$ then $U(t)$ must be real and hence its determinant must be constant and so cannot have a non-trivial winding number (a similar argument involving quaternions holds for $i=5$). If $i=-1$ then $\overline{U(t)} = U(t)^*$. This condition does not  prevent the determinant to have trivial winding number, it is, for instance, satisfied by 
$U(t) = e^{2\pi it}$ and hence we see that 
$$\langle\Ww,KO_{-1}(C(S^1,\R))\rangle = \Z.$$
If $i=3$ then $\overline{U(t)} = - U(t)^*$. This implies that the dimension of $U(t)$ has to be even
and $\overline{\det(U-\lambda 1)} = \det(\overline{U}-\overline{\lambda} 1) = 
(-1)^{2m} \det(U^*+\overline{\lambda} 1) = \overline{\det(U+\lambda 1)}$, and hence that with $\lambda(t)$ also $-\lambda(t)$ is an eigenvalue of $U(t)$. Since both have the same winding number we see that the winding number of $U$ must be even. Thus
$$\langle\Ww,KO_{3}(C(S^1,\R))\rangle = 2 \Z.$$

Let us now consider the real structure $\cfS$ defining the twisted suspension $\tilde S\R$.
As $df = \partial_t f dt$ we have $d\cfS(f) = -\cfS(\partial_t f) dt$. In order for the differential to commute with the real structure $\tilde \cfS$ we thus need to set
$$\tilde\cfS(f dt) = -\cfS(f)dt.$$
In other words, $\tilde\cfS$-invariant $1$-forms are of the form $i f dt$ where $f\in \tilde S\R$.
It follows that 
$(\Omega([0,1]),\tilde\cfS,d_{[0,1]},\Int)$ is a $(*,r)$-cycle of \tosi\ $\ss=-1$.

Cor.~\ref{cor-cond-triv} predicts now that the pairing is only non-trivial for odd $i$  which satisfy 
$(-1)^{\nn(1-i)} = 1$, that is, $i=1$ and $i=5$. The elements of $KO_1(\tilde S\R)$ correspond to unitaries satisfying $\overline{U(t)} = U(1-t)$, a condition which is satisfied by $U(t) = e^{2\pi i t}$ and hence 
$$\langle\Ww,KO_{1}(\tilde S \R)\rangle = \Z.$$

\subsection{Suspension of cycles}
\newcommand{\got}{\wedge}
\newcommand{\dS}{{d^s}}

Given a \CA\ $A$, the algebra  $SA=C_0((0,1),A)$ is called the suspension of $A$. It is a subalgebra of the unital algebra $C([0,1],A)$. 
Let $(\Omega,d,\int)$ be  an $n$-dimensional chain of parity $\nu$  over a graded \CA\ $A$. One can extend this chain to an $n+1$-dimensional chain of parity $\nu$ over 
$C([0,1],A)\cong A\ot C([0,1])$.  For that we consider
the (graded) product of $(\Omega,d,\int)$ with the chain discussed above to obtain
$(\Omega\got\Omega([0,1]),\dS ,\int_{C([0,1],A)})$ where 
\begin{eqnarray*} \dS (\omega\got \mu) &=& d\omega\got \mu + (-1)^{|\omega|_{\Z}}\omega\got d_{[0,1]} \mu\\
\int_{C([0,1],A)}(\omega \got \mu) &=&   \int \omega \Int\mu .
\end{eqnarray*}
Here $\Omega\got\Omega_{[0,1]}$ is the graded tensor product w.r.t.\ the $\Z$-degree and hence $(\omega\got \mu)^* = (-1)^{|\omega|_\Z |\mu|_\Z}\omega^*\got \mu^*$. The character of the extended chain is denoted by $\xi_{[0,1]}$. It is given by
\begin{equation}\label{eq-char-ext}
\xi_{[0,1]}(f_0,\cdots, f_{n+1}) = \sum_{k=0}^n (-1)^k  \int_0^1
\int f_0 df_1 \cdots df_k \partial_t f_{k+1} d  f_{k+2} \cdots d f_{n+1} dt .
\end{equation}
This expression can only be non-trivial if the dimension $n$ of the cycle $(\Omega,d,\int)$ has the same parity as the cycle.

When $(\Omega,d,\int)$ is a cycle then the above extension chain restricts to a cycle
on $SA:=C_0((0,1),A)$, the so-called suspension cycle.

Let $C^\infty([0,1],\Aa)$ be the dense subalgebra of continuous, piecewise smooth functions from $[0,1]$ into $\Aa$ which vanish at the end points. This is a domain algebra for the suspension of the original cycle. Indeed, (C1) and (C2) are directly shown and 
(C3) follows as $C^\infty([0,1],\Bb)$ is, for any complex algebra $\Bb$ which is closed under holomorphic functional calculus, also closed under holomorphic functional calculus. To see that it satisfies (C4) let $f \in SA$ and so $f$ is a continuous function with values in $\Ss\Aa$. Let $b>0$. We can partition $[0,1]$ into intervals $[t_i,t_{i+1}]$, $0=t_0 < t_1 \cdots t_N=1$ such that for all $t\in  [t_i,t_{i+1}]$ we have $\|f(t)-f(t_i)\|< b$. Since $\Aa$ is a domain algebra, there are $\tilde f_i\in \Us(\Aa,\gamma)$ such that $\|\tilde f_i-f(t_i)\|< b$.
By Lemma~\ref{lem-HFC} there is a smooth path $\tilde f_i(s)$ from $\tilde f_i = \tilde f_i(0)$ to 
 $\tilde f_{i+1} = \tilde f_i(1)$ in $\Us(\Aa,\gamma)$
 such that $\|\tilde f_i(s)-\tilde f_i\|\leq C(3b)$ for all $s\in [0,1]$ where $C:\R^+\to\R^+$ is a continuous function with $C(0)=0$. Concatenating the paths $\tilde f_0,\cdots,\tilde f_{N-1}$ to one path $\tilde f$ yields a continuous, piecewise smooth loop in $\Us(\Aa,\gamma)$ which has distance $\sup_t \|\tilde f(t) - f(t)\| \leq 3b + C(3b)$ from $f$. This proves (C4).
 
If $\xi$ is the character of the original cycle then we denote by $\xi^S$ the character of its suspension.

If $A$ is a \RA\  then the suspension of an $n$-dimensional $(*,r)$-cycle is an $n+1$-dimensional  $(*,r)$-cycle with the same \tosi\ and parity.
\begin{cor}\label{cor-V}
Let $(\Omega,\tilde\rs,d,\int)$ be an $n$-dimensional $(*,\rs)$-chain over a balanced graded \RA\ $(A,\rs)$ with parity $\nu$ and \tosi\ $\ss$. 
Let $\xi$ be its character and $\xi_{[0,1]}$ its extension. If
$$\ss = (-1)^{n+1}\quad\mbox{or}\quad\nu = n\:\:mod\:\:2 $$
then $\int z (dx)^{n}=0$ for all odd self-adjoint elements $x,z\in A^\rs$. If 
$$\ss = +1$$
then $\int zu (du)^{n}=0$ for all odd self-adjoint elements $u,z\in A^\rs$, $u$ unitary, and
$\xi_{[0,1]}(f, \cdots,f) = 0$ for any odd self-adjoint $f\in C^1([0,1],A^\rs)$.
\end{cor}
\begin{proof}
This follows from (\ref{eq-char-ext}) and Prop.~\ref{cor-sp-graded} upon taking $x = f(t)$,
$z=x-e_m$, and $y = \dot f(t)$.
\end{proof}

The following lemma is a generalisation to graded \CA s of Pimsner's formula \cite{Pimsner}. 
\begin{lemma}\label{lem-Pimsner} Let $A$ be a graded complex \CA\ and $\xi$ be the character of a $n$-dimensional cycle $(\Omega,d,\int_A)$ over $A$. Then
$$\langle \xi^S,\beta[x]\rangle = c_n \langle \xi,[x]\rangle$$ 
where 
$$ c_n = (-1)^{n+1} \kappa \pi(n+1) \alpha_{n,n},\quad \alpha_{n,n} = \int_0^1\sin^n(\pi t) dt $$
\end{lemma}
\newcommand{\wu}{w}
\begin{proof} We first need the following observation: If $e$ is an \osu\ in $A$ and $w$ a even unitary in $A$ then $[wew^{-1}]= -[w^{-1}ew]$ as elements in $\DK_e(A)$. Indeed, $w\oplus w^{-1}$ is homotopic to $1\oplus 1$ in the set of even unitaries of $A$, and therefore  $wew^{-1}\oplus w^{-1}ew$ osu-homotopic to $e\oplus e$, which represents the neutral element in $\DK_e(A)$. 
Using this we obtain from Theorem~\ref{thm-susp}
$$\beta[x] = -[Z], \quad
Z=\nu(e_m)\nu^{-1}(x) (1\hot \kg_m)  \nu(x) \nu^{-1}(e_m).$$
We abbreviate $\wu_t = \nu_t(x)$ and $r_t = \nu_t(e_m)$ where $\nu$ is given in (\ref{eq-Bott-dual}).
Taking into account that $x\hot 1$ and $1\hot \kg_m$ anticommute,
$$ b_t:= \wu_t^{-1} (1\hot \kg_m)  \wu_t = 1\hot \kg_m (c_t\hot 1 + x\hot \kg_m s_t)^2 = c_{2t}\hot \kg_m -xs_{2t}\hot 1.$$
A direct calculation shows that
$\dot Z = r_t\big(r_t^{-1}\dot r_t b_t + \dot b_t - b_tr_t^{-1}\dot r_t\big)r_t^{-1}$. Hence
\begin{eqnarray*} 
Z(d^s Z)^{n+1} &=& \sum_{k=0}^{n} r_t\big(b_t (db_t)^k ( r_t^{-1}\dot r_t b_t  +\dot b - b_tr_t^{-1}\dot r_t) dt (db)^{n-k}\big)r_t^{-1}
\end{eqnarray*}
When applying the graded trace $\int$ then, by graded cyclicity, the conjugation with ${r_t}$ simply drops out. 
Furthermore $r_t^{-1}\dot r_t = \frac{\pi}2 e_m\hot \kg_m$ and $\dot b_t = -\pi(x c_{2t}\hot 1 + s_{2t}\hot \kg)$ so that 
$$  \big(r_t^{-1}\dot r_t b_t + \dot b_t - b_tr_t^{-1}\dot r_t\big) = 
\pi\big( c_{2t}(e-x)\hot 1 + (\frac{xe + ex}2 - 1) s_{2t}\hot \kg\big)$$ 
Now all terms which are even powers in $\kg$ lie in the kernel of $\jmath_1$.
Furthermore, not counting the derivatives $dx$, all terms with even powers of $x$ yields terms which are total derivatives, because $x^2=1$, an identity which can be used after permuting $x$ with $dx$ or permuting $x$ cyclicly.
Thus the only terms from  $b_t (db_t)^k (  [r_t^{-1}\dot r_t, b_t] +\dot b)$ are not in the kernel of $\int^s \jmath_1$ are
$$\pi(c_{2t}\hot \kg) (db_t)^k (-xc_{2t}\hot 1) - \pi(xs_{2t}\hot 1)(db_t)^k (-s_{2t}\hot \kg) =  \pi(x\hot 1) (db_t)^k(1 \hot \kg)$$
where we have used $(1\hot \rho) (db_t)^k (x\hot 1) = (-1)^{k+1} (db_t)^k (x\hot 1)(1\hot \rho)  = - (x\hot 1)(db_t)^k (1\hot \rho)$.
We thus get, using $db_t = -s_{2t} dx\hot 1$,
\begin{eqnarray*} 
\int_{C([0,1],A)} \jmath_1\big(Z(d^s Z)^{n+1}\big)  &=& \pi\sum_{k=0}^{n}\int_{C([0,1],A)} \jmath_1\big( (x\hot 1 )(db_t)^k (dt\hot\rho) (db_t)^{n-k}\big)\\
&=& \pi\sum_{k=0}^{n}\int_{C([0,1],A)} \jmath_1\big((-s_{2t})^n x(dx)^n dt\hot\rho\big)\\
& = &  \kappa \pi(n+1)\int_0^1 \sin^n(\pi t) dt \int_A x(dx)^n
\end{eqnarray*}
For the second equality we used that $dt$ anticommutes with $dx$ as the latter has $\Z$-degree $1$ while $1\hot\rho$ anticommutes with $dx\hot 1$ because the latter has $\Z_2$-degree $1$. 
Since $\langle\xi,[x]\rangle = \int x(dx)^n$ and 
$\langle \xi^S,\beta[x]\rangle = -  \int_{C([0,1],A)} \jmath_1\big(Z(d^s Z)^{n+1}\big)$ we arrive at
$$c_n = (-1)^{n+1} \kappa \pi (n+1)\int_0^1 \sin^n(\pi t) dt.$$ 
\end{proof}
The integral in the definition of $c_n$ is given by
\begin{equation}\label{eq-alpha}
\alpha_{n,n}=\int_0^1 \sin^n(\pi t) dt = \left\{\begin{array}{cc} \frac{1}{2^n} \left(n \atop \frac{n}2 \right) & \mbox{if $n$ is even}\\ 
& \\
 \frac{2^{n}}{n \pi} \left( n-1\atop \frac{n-1}2 \right)^{-1} & \mbox{if $n$ is odd} 
 \end{array}\right. .
 \end{equation}


\section{A torsion valued pairing with $K_n(A^\rs,\gamma)$}\label{sec-tor}
\newcommand{\eqm}{\equiv}
\newcommand{\im}{\mbox{\rm im}}
\newcommand{\Cl}{\C l}

\newcommand{\gen}[1]{{\sigma^{(#1)}}}
\newcommand{\Gen}[1]{{\Sigma^{(#1)}}}
\newcommand{\Pplus}[1]{{\Pi_{+}^{#1}}}
\newcommand{\Aplus}[1]{A_{\Pplus{#1}}}

Connes pairing is additive and takes values in the linear space $\C$. In particular we have 
$\langle\xi,[x]\rangle= \frac12 \langle\xi, 2[x]\rangle$ showing that the pairing must vanish on $2$-torsion elements. We now provide a construction which is geared to see such torsion elements.
\subsection{General construction}
Consider a unital inclusion $\varphi$ of (real or complex) balanced graded \CA s 
$$ B \stackrel{\varphi}\hookrightarrow \tilde B.$$
The relative cone of the inclusion is the algebra 
$$\Cc(\varphi)=\{f\in C([0,1],\tilde B) | f(0)\in 0, f(1)\in \varphi(B)\}$$ 
It fits into the short exact sequence
$$ 0 \to S\tilde B \stackrel{i}\to \Cc(\varphi) \stackrel{ev_1}\to B \to 0$$ 
which gives rise to the LES ($6$ periodic in the complex and $24$-periodic in the real case)
$$ K_{i}(S\tilde B)\stackrel{i_*}\to K_i(\Cc(\varphi)) \stackrel{{ev_1}_*}\to K_i(B) \stackrel{\delta}\to K_{i-1}(S\tilde B) \to \cdots $$
and when followed by the inverse of the Bott isomorphism $\beta:K_{i}(\tilde B)\to K_{i-1}(S\tilde B)$ the boundary map $\delta$ becomes the map induced on $K$-theory by $\varphi$
$$\begin{array}{ccc}
K_i(B) & \stackrel{\varphi_*}\to & K_{i}(\tilde B)\\
\| & & \downarrow \beta \\
K_i(B) & \stackrel{\delta}\to & K_{i-1}(S\tilde B)
\end{array}$$

\newcommand{\cV}{$(\mathrm V)$}
We formulate a vanishing condition for characters $\tilde \xi$ of chains over $\tilde B$.
\begin{itemize}
\item[\cV] 
For any $m$ and differentiable function $g:[0,1]\to \Us(M_m(\varphi(B)))$
$$\tilde\xi_{[0,1]}\#\Tr_m (g,\cdots,g) =0.$$
\end{itemize}
In the following we suppose that $B$ contains a basepoint $e$ which 
is homotopic to its negative in $\Us(B)$.
\begin{definition} Let $\varphi:B\to \tilde B$ be a unital inclusion.
Let $(\Omega,d,\int_{\tilde B})$ be an $n$-dimensional chain 
with character $\tilde \xi$ over a graded \CA\ $\tilde B$ 
which satifies \cV. Suppose that $B$ contains a basepoint $e$ which 
is homotopic to its negative in $\Us(B)$ and satisfies $d\varphi(e)=0$. 
The \tv\ pairing is given by the homomorphism
$\Delta^\varphi_{\tilde\xi}:\ker \varphi_*\cap  \DK_e(B)\to \C / \langle\tilde\xi\#\jmath_1,\DK_{\varphi(e)}(\tilde B\hot Cl_{0,1})\rangle$
$$\Delta^\varphi_{\tilde\xi}([x]) \eqm
 c_{n}^{-1} \tilde \xi_{[0,1]}\#\Tr_m (F - \varphi(e)_m,\cdots,F - \varphi(e)_m)$$
where $x \in \Us(M_l(B)\otimes Cl_{r,s})$ and
$F$ is a continuous path from $\varphi(e)_m$ to $\varphi(x)\oplus \varphi(e)_{m-l}$ in $\Us(M_m(\tilde B))$, for some $m\geq l$.
\end{definition}
Here $\eqm$ means equality in the quotient group, that is, the value is to be understood modulo the subgroup $\langle\tilde\xi\#\jmath_1,\DK_{\varphi(e)}(\tilde B\hot Cl_{0,1})\rangle$.
We deduce from (\ref{eq-char-ext}) that
$\Delta^\varphi_{\tilde\xi}$ can only be non-trivial if the dimension and the parity of  the chain $(\Omega,d,\int_{\tilde B})$ coincide. 

We argue why $\Delta^\varphi_{\tilde\xi}$ is well defined:
$\ker \varphi_*$ is generated by elements whose representatives $x$ belong to 
$\Us(M_l(B))$
for some $l$ and such that $\varphi(x)\oplus \varphi(e)_{m-l}$ is
homotopic to $\varphi(e)_m$ in 
$\Us(M_m(\tilde B))$, for some $m\geq l$. $F$ is such a homotopy. 
If we take a different homotopy $F'$ then the composition of $F$ with $F'$ run backwards yields a loop $L$ with values in $\Us(M_m(\tilde B))$ starting at $\varphi(e)_m$, and hence its homotopy class $[L]$ defines an element of $\DK_{\varphi(e)}(S\tilde B)$. 
The ambiguity is thus given by 
$ c_{n}^{-1} \tilde \xi_{[0,1]}\#\Tr_m (L - e_m,\cdots,L - e_m)$ which is an element of 
$ c_{n}^{-1} \langle \tilde \xi_{[0,1]},\DK_{\varphi(e)}(S\tilde B)\rangle$. Now
\begin{eqnarray*}
 \langle \tilde \xi_{[0,1]},\DK_{\varphi(e)}(S\tilde B)\rangle & = &  \langle \tilde \xi_{[0,1]}\#\jmath_2 ,\DK_{\varphi(e)}(S\tilde B\hot Cl_{1,1})\rangle \\
 & = &   \langle \tilde \xi_{[0,1]}\#\jmath_1\#\jmath_1 ,\DK_{\varphi(e)}(S\tilde B\hot Cl_{0,1}\hot Cl_{1,0})\rangle \\
 & = &  c_n \langle \tilde \xi \#\jmath_1 ,\DK_{\varphi(e)}(\tilde B\hot Cl_{0,1})\rangle
\end{eqnarray*}
The ambiguity is thus taken care of by moding out $\langle \tilde \xi \#\jmath_1 ,\DK_{\varphi(e)}(\tilde B\hot Cl_{0,1})\rangle$.

If we take another representative $x'$ which is homotopic to $x$ then we can prolong $F$ in $\Us(M_m(\varphi( B)))$
from $x$ to $x'$ and condition \cV\ insures that this does not change the value of 
$ \tilde\xi_{[0,1]}\#\Tr_m(F-e_m,\cdots,F-e_m)$.
\bigskip

We apply the above to two particular cases, an even and an odd one. They naturally occur if we have a \RA. 
\subsection{Even \tv\ pairing}
Let $B$ be a real graded \CA\ and $B_\C=B\otimes_\R\C$ its complexification. On $B_\C$  we consider the real structure $\cc$ given by complex conjugation on the second factor. 
Then, of course, $B$ is the real sub-algebra of $(B_\C,\cc)$. 
Let $$ \tilde B = B\hot Cl_{0,1}$$ and $j:B\hookrightarrow B\hot Cl_{0,1}$ be given by 
$$j(b) = b\hot 1.$$ 
\begin{theorem}\label{thm-TV-ev}
Let $B$ be a real graded \CA\ and $(\Omega,d,\int_B)$ an $n$-dimensional $(*,\rs)$-cycle over $B_\C$ with character $\xi$. 
Let $\tilde\xi=\xi\#\jmath_1$ be its extension to $B\hot \Cl_{1}$. 
Suppose that $B$ contains a basepoint $e$ which 
is homotopic to its negative in $\Us(B)$ and satisfies $de=0$.
Then 
$\Delta_{\tilde\xi}^j : \ker j_*\cap \DK_e(B) \to \C/ \langle\xi\#\jmath_2,\DK_e(B\hot Cl_{0,2})\rangle$ 
$$\Delta_{\tilde\xi}^j ([x]) \eqm  c_{n}^{-1} \xi_{[0,1]}\#\jmath_1\#\Tr_m (F - e_m\hot 1,\cdots,F - e_m\hot 1)$$
where $F$ is a \osu-homotopy in $M_m(B)\hot Cl_{0,1}$ between $e_m\hot 1$ and $x\hot 1$,
is a well-defined homomorphim. 
Necessary conditions for $\Delta_{\tilde\xi}^j $ to be non-trivial are that the dimension $n$, 
parity $\nu$, and \tosi\ $\ss$ of $(\Omega,d,\int)$ satisfy
\begin{equation}\label{eq-cond+}
\ss = (-1)^{\nu} \quad\mbox{and}\quad \nu = n+1 \:\: mod \:\: 2.
\end{equation}
Moreover, under these conditions $\langle\xi,\DK_e(B)\rangle=0$.
\end{theorem}
\begin{proof}
We first show \cV. Indeed, $\jmath_1$ vanishes on the image of $j$. Therefore  $\tilde\xi$ vanishes on the image of $j$ which implies \cV.

By Lemma~\ref{lem-p-ext}
the extension $\tilde\xi=\xi\#\jmath_1$ has dimension $n$, parity $\nu+1$, and \tosi\ $(-1)^{\nu +\nn(-1)}\ss$ on $B\hot Cl_{0,1}$. We deduce from (\ref{eq-char-ext}) that, if $n$ has parity different from $\nu+1$ then the expression $\tilde\xi_{[0,1]}(f,\cdots,f)$ vanishes. It then follows from 
from Cor.~\ref{cor-V} (applied to $\tilde \xi$) that if $(-1)^{\nu}\ss =-1$ then $\tilde\xi_{[0,1]}(f,\cdots,f)$ vanishes. 

Finally we conclude from Cor.~\ref{cor-V} applied to $\xi$ that 
under the above conditions  $\langle\xi,\DK_e(B)\rangle=0$.
\end{proof}

\subsection{Odd \tv\ pairing}
Let again $B$ be a real \CA\ but $\tilde B=B_\C$ be its complexification. Viewed differently, 
we can start with a complex \RA\ $(\tilde B,\rs)$ and set $B = \tilde B^\rs$. Let $c:B\to \tilde B$ be  the complexification map $c(b) = b$. 
\begin{theorem}\label{thm-TV-odd}
Let $B$ be a real graded \CA\ and $(\Omega,d,\int)$ an $n$-dimensional $(*,\rs)$-cycle over $B_\C$ with character $\tilde\xi$. Suppose that $B$ contains a basepoint $e$ which 
is homotopic to its negative in $\Us(B)$ and satisfies $de=0$. 
If the \tosi\ satisfies $\ss = +1$, then
$\Delta_{\tilde\xi}^c : \ker c_*\cap \DK_e(B) \to \C/\langle\tilde \xi\#\jmath_1,\DK_e(B_\C\hot \C l_1)\rangle$
$$\Delta_{\tilde\xi}^c ([x]) \eqm  c_{n}^{-1} \tilde\xi_{[0,1]}\#\Tr_m (F - e_m,\cdots,F - e_m)$$
where $F$ is an \osu-homotopy in $M_m(B_\C)$ between $e_m$ and $x$,
 is a well-defined homomorphim.
A necessary condition for $\Delta_{\tilde\xi}^c $ to be non-trivial 
is that the dimension $n$ and the parity $\nu$ of the cycle satisfy
\begin{equation}\label{eq-cond-}
\nu = n \:\: mod \:\: 2.
\end{equation}
Moreover, under these conditions $\langle\tilde\xi,\DK(B)\rangle=0$.
\end{theorem}
\begin{proof}
By Cor.~\ref{cor-V} the first condition $\ss = +1$ implies that the vanishing condition \cV\ is satisfied 
for the inclusion $B\stackrel{c}\hookrightarrow B_\C$.  
Condition (\ref{eq-cond-}) is necessary for the non-vanishing of 
$\tilde\xi_{[0,1]}(f,\cdots,f)$ where $f$ is an odd self adjoint differentiable function with values in $B_\C$
($f$ is not necessarily real). It also implies that $\langle\tilde\xi,\DK_e(B)\rangle=0$.
\end{proof}

We discuss furtherdown that $\ker c_* = \im j_*$ and that $2\DK(B)\subset \ker j_*$.
Hence $\ker c_*$ contains only $2$-torsion elements.

\subsection{The \tv\ pairings for the algebra $\R$}
As a simple application we show that the torsion part of the $K$-theory of $\R$ can be detected by the \tv\ pairings, notably the odd pairing is non-trivial on $K_1(\R)$ and the even pairing is non-trivial on $K_2(\R)=\Z_2$. For that we use the up to normalisation only non-trivial cycle $(\Omega,d,\int) = (\C,0,\id)$ over $\C$, its character is $\xi=\id$.

We have $K_1(\R) = \DK(M_2(\R)\otimes Cl_{1,1})\cong\Z_2$ and $e=\sigma_z\otimes \sigma_x$ is a basepoint which is homotopic to its negative. The generator of $\DK(M_2(\R)\otimes Cl_{1,1})$ is given by the class of $x=1_2\otimes\sigma_x$. Let $\Sigma=1_2\otimes \sigma_y$. Then the concatination of the two paths,
$F_1(t) = \cos(\frac{\pi}2 t) e +  \sin(\frac{\pi}2 t) \Sigma$ with 
$F_2(t) = \cos(\frac{\pi}2 t)  \Sigma + \sin(\frac{\pi}2 t) x$, provides a homotopy from $e$ to $x$ in $M_2(\C)\otimes \C l_2$. 
Thus $[x]$ lies in the kernel of $c_*$ and
\begin{eqnarray*}
\Delta_{\id}^{c} ([x]) &\eqm & c_0^{-1} \int_0^1 \Tr_2\#\jmath_2 (F_1(t)\dot F_1(t)
+F_2(t)\dot F_2(t)) dt \\
& = & \frac{\kappa^{-1}}{2}  \Tr_2\#\jmath_2 (\sigma_z\otimes \sigma_x\sigma_y + 1_2\otimes\sigma_y\sigma_x) = -\kappa
\end{eqnarray*}
The subgroup to be devided out is given by the values of
$$\frac{\kappa^{-1}}{\pi} \int_0^1 \Tr_2\#\jmath_2 (F(t)\dot F(t)) dt$$
where $F:[0,1]\to M_2(\C)\otimes \C l_2$ is an osu-valued loop. 
In particular, $F$ has the form
$$F(t) = \begin{pmatrix} 0 & Z(t) \\ Z^*(t) & 0 \end{pmatrix}$$
where $Z(t)\in M_2(\C)$ is a unitary loop. Then  
\begin{eqnarray*}
\int_0^1 \Tr_2\#\jmath_2 (F(t)\dot F(t)) dt &=& \int_0^1 \Tr_2\#\Tr_2 \begin{pmatrix} 1_2 & 0  \\ 0 & -1_2 \end{pmatrix}\begin{pmatrix} Z(t) \dot Z^*(t) & 0 \\ 0 & Z^*(t) \dot Z(t) \end{pmatrix} dt \\
& = & 2 \int_0^1 \Tr_2  Z(t) \dot Z^*(t)  dt = -4\pi i \mathcal W(Z)
\end{eqnarray*}
where $W(Z)$ is the winding number of $Z$. Thus we mod out the group $2\kappa\Z$ and the odd \tv\ pairing is injective. 

We have $K_2(\R) = \DK(M_2(\R)\otimes M_2(\R)\otimes Cl_{0,1})\cong\Z_2$ and take as basepoint $e=\sigma_z\otimes \sigma_y\otimes \rho_1$ where $\rho_1$ is the generator of $\Cl_{1}$ (hence $\sigma_y\otimes \rho_1 \in M_2(\R)\otimes Cl_{0,1}$). 
Again $e$ is homotopic to its negative. A generator of $K_2(\R)$ is given by the class of $x=1_2\otimes\sigma_y\otimes \rho_1$. Let $\Sigma=1_2\otimes \sigma_y\otimes\rho_2$ where $\rho_2$ is the second generator of $\C l_2$. Hence $\Sigma\in M_2(\R)\otimes M_2(\R)\otimes C l_{0,2}$.
Again the concatenation of 
$F_1(t) = \cos(\frac{\pi}2 t) e +  \sin(\frac{\pi}2 t) \Sigma$ with 
$F_2(t) = \cos(\frac{\pi}2 t)  \Sigma + \sin(\frac{\pi}2 t) x$ yields  a homotopy from $e$ to $x$  in $M_2(\R)\otimes M_2(\R)\otimes C l_{0,2}$.
Thus $[x]\in\ker j_*$ and, similar to the above,
\begin{eqnarray*}
\Delta_{\id}^{j} ([x]) 
&\eqm & c_0^{-1} \int_0^1 \Tr_2\#\jmath_2 (F_1(t)\dot F_1(t)+F_2(t)\dot F_2(t)) dt \\
&\eqm & \frac{\kappa^{-1}}{2}  \Tr_4\#\jmath_2 (\sigma_z-1)\otimes 1_2 \otimes \rho_1\rho_2  = 2\kappa
\end{eqnarray*}
The subgroup to be devided out is generated by the values of
$\frac{\kappa^{-1}}{\pi} \int_0^1 \Tr_2\#\jmath_2 (F(t)\dot F(t)) dt$
where $F:[0,1]\to M_2(\R)\otimes C l_{0,2}$ is an osu-value loop which also in this case must have the form $F(t) = \begin{pmatrix} 0 & Z(t) \\ Z^*(t) & 0 \end{pmatrix}$
so that these values are given by $-4\pi i \mathcal W(Z)$. 
But now the reality condition $F(t)\in M_2(\R)\otimes C l_{0,2}$ implies that
$Z(t)$ has to be a unitary in $M_2(\C)$ which satisfies
$$ \overline{Z(t)} = -Z^*(t)$$
(entrywise complex conjugation). In particular, the spectrum of $Z$ is invariant under mutliplication with $-1$. Thus its winding number must be even. It follows that we mod out the group $4\kappa \Z$ and the even \tv\ pairing is injective as well. 

\subsection{Definition of Property $Y$}
The simplest way to construct an \osu-valued homotopy between two \osu s $x$ and $x'$ is to find a third \osu\ $y$ which anti-commutes with $x$ and with $x'$. Indeed, 
with $c(t) = \cos(\frac{\pi}2 t)$ and $s(t) = \sin(\frac{\pi}2 t)$,
the path
$F_{x,y}(t) := c(t)x + s(t) y$ is an \osu-homotopy between $F_{x,y}(0)=x$ and $F_{x,y}(1)=y$ and so its concatenation with $F_{y,x'}$ provides a homotopy between $x$ and $x'$.
This will be the basis of the explicit formulae we develop below for 
the even and odd \tv\ pairing where, for abstract $K$-theoretic reasons, such an extra \osu\ $y$ always exist. 
Moreover, in the context of insulators $y$ can be interpreted as an extra symmetry.

\begin{definition}
Let $(A,\rs)$ be a graded \RA.
We say that an element $x\in \Us(A^\rs)$ 
satisfies property $Y^{(\pm)}$,  if $iA^\rs$  
contains a self-adjoint unitary $\Gen{\pm}$ of degree $\pm$ such that 
$$ x\Gen{\pm} = \pm \Gen{\pm}x.$$
If $e$ is a base point in $A^\rs$ such that also
$ e\Gen{\pm} = \pm \Gen{\pm}e$ we say that $x$ satisfies $Y_e^{(\pm)}$. 
\end{definition}
$\Gen{+}$ is an even imaginary self-adjoint unitary which commutes with $x$. 
$\Gen{-}$ is an odd imaginary self-adjoint unitary which 
anti-commutes with $x$. 
We call $\Gen{\pm}$ the generator of an extra symmetry of $x$. We emphazise that $\Gen{\pm}$ will not commute or anti-commute with all elements of the homotopy class of $x$, an extra symmetry is thus not a protecting symmetry of the topological phase. 

\begin{lemma}\label{lem-Y}
Let $(A,\rs)$ be a balanced graded \RA\ with base point $e$. 
\begin{itemize}
\item If $x\in \Us(A^\rs)$ satisfies property $Y_e^{(+)}$,
then $x\otimes 1$ is homotopic to $e\otimes 1$ in 
$\Us(A^\rs\hot Cl_{0,1})$ and hence $j_*([x]) = 0$.
\item If $x\in \Us(A^\rs)$ satisfies property $Y_e^{(-)}$,
then $x$ is homotopic to $e$ in 
$\Us(A)$ and hence $c_*([x]) = 0$.
\end{itemize}
\end{lemma}
\begin{proof}
$\Gen{+}\otimes \kg$ 
belongs to $\Us(A^\rs\hot Cl_{0,1})$ and anticommutes with $x\otimes 1$ and $e\otimes 1$. By the standard argument this means that $\Gen{+}\otimes \kg$ is homotopic to $x\otimes 1$ and $e\otimes 1$ in $\Us(A^\rs\hot Cl_{0,1})$.

$\Gen{-}$ belongs to $\Us(A)$ and anticommutes with $x$ and $y$. By the standard argument this means that it is homotopic to $x$ and $e$ in $\Us(A)$.
\end{proof}

\begin{cor}
The kernel $\ker j_*$ contains $2 \DK(A^\rs)$. In particular $\im j_*$ contains only $2$-torsion elements. 
\end{cor}
\begin{proof}
Let $[x]\in \DK(A^\rs)$. Then $2[x]$ has a representative of the form $\begin{pmatrix}x & 0 \\ 0 & x \end{pmatrix}$ which commutes with the imaginary self adjoint unitary  
$\begin{pmatrix}0 & -i \\ i & 0 \end{pmatrix}$ and hence satisfies $Y^{(+)}_{e_2}$.
\end{proof}
\subsection{The exact sequence relating $j_*$ and $c_*$}
Let $B$ be a graded real \CA\ and $r:B_\C = B + iB \to M_2(B)$ be given by 
$$r(a+ib) = \begin{pmatrix}
a & -b \\ b & a \end{pmatrix} .$$
$r$ is a unital homomorphism of graded real \CA s referred to as realification (or forgetting the complex structure). It hence induces a homomorphism $r_*$ on $K$-theory. 
Composed with the Bott isomorphism $b=\beta^2$ and together with $j_*$ and $c_*$ it forms an exact sequence relating the $K$-theory of $B$ to that of its complexification $B_\C$.
\begin{theorem}\label{thm-WK}
Let $B$ be a graded real \CA. The following sequence is exact.  
$$ \cdots\to K_{i+2}(B_\C) \stackrel{r_*\circ b^{-1}}\to K_i(B)   \stackrel{j_*}\to K_{i+1}(B)
\stackrel{c_*}\to K_{i+1}(B_\C) \to \cdots $$
\end{theorem}
This theorem, attibuted to Wood and Karoubi in \cite{Sch}
(in the commutative case), has been proven for trivially graded \CA s in \cite{Boersema1,Boersema2} and generalised to equivariant $KK$-theory in \cite{Schick} where however the map in the middle is given by Kasparov multiplication with the generator of $K_1(\R)$. Guerin provides a proof for graded \CA s in which he identifies the map in the middle with $j_*$ \cite{Guerin}.
 
Recall that $K_i(A)$ is defined as $\DK_e(A\hot Cl_{r,s})$ where $r-s=1-i$. We can always choose $r$ large enough so that $A\hot Cl_{r,s}$ contains an \osu\ $e$  which is homotopic to its negative. 
Recall that a differential $d$ of a cycle over $A$ extends to a differential on 
$M_m(A)\hot Cl_{r,s}$ 
using the identity on the right factor and entrywise extension to matrices.
In the same way
the maps $r$, $j$, and $c$ extend to $M_m(A)\hot Cl_{r,s}$. Note that $d$ commutes with  $r$, $j$, and $c$.

\begin{prop}\label{prop-Y}
Let $A$ be a graded real \CA. Suppose that $A\hot Cl_{r,s}$ contains a basepoint $e$ which is homotopic to its negative. Let $d$ be a differential of a cycle over $A$.
\begin{enumerate}
\item Any element of $\ker j_*\cap \DK_e(A\hot Cl_{r,s})$ admits a representative which satisfies property $Y_e^+$. Moreover, the corresponding generator $\Sigma^{(+)}$ can be chosen to satisfy $d\Sigma^{(+)}=0$. 
\item Any element of $\ker c_*\cap \DK_e(A\hot Cl_{r,s})$ admits a representative which satisfies property $Y^-$. The corresponding generator $\Sigma^{(-)}$ can be chosen to satisfy $d\Sigma^{(-)}=0$. If moreover $e=e'\hot 1$ for some
$e'\in A\hot Cl_{r,s-1}$ then the representative 
satisfies property $Y^-_e$. 
\end{enumerate}
\end{prop}
\begin{proof} 
By Theorem~\ref{thm-WK} $\ker j_* = \im  r_*$. Since $e$ is homotopic to its negative 
any element of $\ker j_*$ is represented by an \osu\ $r(z)$, for some \osu\ $z$ in $M_{m}(A_\C)\hot Cl_{r,s}$. Furthermore, viewing $e$ as an element of $A_\C$ we have $r(e) = e_2$. Clearly  
$z$ and $e_m$ commute with $i 1_m\otimes 1\in M_m(A_\C)\hot Cl_{r,s}$, where $1_m$ is the unit matrix in $M_m(\C)$. Thus $\Sigma^{(+)}:=ir(i 1_m)$ is an imaginary self-adjoint unitary which commutes with $r(z)$ and $r(e_m)=e_{2m}$. Furthermore, $d\Sigma^{(+)}=ir(id1_m) = 0$.

By Theorem~\ref{thm-WK} $\ker c_* = \im  j_*$. Therefore any element of $\ker c_*$ is represented by an \osu\ $j(z)=z\hot 1$, for some \osu\ $z$ in $M_{m}(A)\hot Cl_{r,s-1}$. Clearly  
$z\hot 1$ anti-commutes with the imaginary odd self-adjoint unitary $\Sigma^{(-)}:=i 1_m\hot \kg$ ($\kg$ is the generator of $\Cl_{1}$) and so satisfies property $Y^-$. 
We have $d\Sigma^{(-)}:=i d1_m\hot \kg=0$.
If $e_m=e'_m\hot 1$ then also $e_m$ commutes with $i 1_m\hot \kg$.
\end{proof}
\subsection{Formulae for \tv\ pairing}\label{sec-secondary}
\newcommand{\ixi}{\tilde \xi_{[0,1]}}
We derive now formulae for the \tv\ pairings related to the character $\tilde\xi$  of a cycle $(\Omega,d,\int)$ over a unital \CA\ $\tilde B$.

Let $F:[0,1]\to \Us(\tilde B)$ be a continuous path from $F(0)=E$ to $F(1) = X$ where we suppose that $dE=0$. We want to compute
$$\ixi(F-E,\cdots ,F-E) = \int_{C([0,1],\tilde B)} (F-E) (\dS F)^{n+1}.$$
The integral can be split into two summands, because $\tilde\xi$ is densely defined on $\tilde B$.
Using $F dF = - (dF) F$ we obtain for the first summand
$$\int_{C([0,1],\tilde B)} F(\dS F)^{n+1} 
=  (n+1) \int_0^1 \int_{\tilde B} F \dot F (dF)^{n} dt$$
where $\dot F$ is derivative w.r.t.\ $t$.
Using $dE=0$ the second summand is the boundary term
\begin{equation}\label{eq-bdry-term}
\int_{C([0,1],\tilde B)} E(\dS F)^{n+1} = \int_{C([0,1],\tilde B)} \dS (EF(\dS F)^{n}) = \left.
 \int_{\tilde B} EF(dF)^n \right|_0^1 
= \int_{\tilde B} EX (dX)^n.
\end{equation}
Suppose that $F$ has the form $F_{x,y}(t) = c(t)x + s(t) y$ for two anti-commuting \osu s.
Then $F_{x,y} \dot F_{x,y} = \frac{\pi}2 yx$ and
\begin{eqnarray*}
\int_{C([0,1],\tilde B)} F_{x,y} (\tilde d F_{x,y})^{n+1} 
&=& 
 \frac{(n+1)\pi}2 \int_0^1 \int y x (c(t)dx+s(t)dy)^n  dt\\
&=& 
 \frac{(n+1)\pi}2\sum_{k=0}^n \alpha_{k,n}   \int y x P_k(dx,dy)
\end{eqnarray*}
where
$$ P_k(a,b) =  \sum_{\pi\in S_{n,k}} 
c^\pi_1\cdots c^\pi_n, \quad
c^\pi_i = \left\{
\begin{array}{cl}
a & \mbox{if } \pi^{-1}(i)\leq k \\
b & \mbox{if } \pi^{-1}(i) > k
\end{array}\right.
$$
$S_{n,k}$ being the subgroup of permutations of $n$ elements which preserve the order of the first $k$ and of the last $n-k$ elements. 
($P_k(dx,dy)$ is the sum over all products of $k$ factors $dx$ with $n-k$ factors $dy$ in all possible orders) and 
$$\alpha_{k,n} = \int_0^1 c^k(t) s^{n-k}(t) dt$$
We have $\alpha_{k,n} = \alpha_{n-k,n} $ and
$\alpha_{k,n} = \alpha_{k-2,n-2} - \alpha_{k-2,n} $.

\newcommand{\sS}{\Sigma}
Suppose now that $E$ and $X$ anticommute with an \osu\ $\sS\in\tilde B$.
Then we can
take the concatenation of $F_1:=F_{E,\sS}$ with $F_2:=F_{\sS,X}$ for $F$ to compute the first above summand, 
\begin{eqnarray} \nonumber 
c_n^{-1} \int_{C([0,1],\tilde B)} F(\dS F)^{n+1} &=& 
\frac{\kappa^{-1}}2 \left(  \sum_{k=0}^n \beta_{k,n}\int_{B} X\sS P_k(dX,d\sS) + \int_{B} \sS E (d\sS)^n \right) 
\end{eqnarray}
where $ \beta_{k,n} = \frac{\alpha_{k,n}}{\alpha_{n.n}}$.

We apply this to the even and the odd \tv\ pairing. 
\subsubsection{Even \tv\ pairing} 

Recall that the even \tv\ pairing is defined for the inclusion 
$B\stackrel{j}\hookrightarrow B\hot Cl_{0,1}$ for a real graded \CA\ $B$. 
Here $\xi$ is the character of an $n$-dimensional cycle over $B$ and $\tilde\xi = \xi\#\jmath_1$ its extension to $\tilde B = B\hot Cl_{0,1}$.
To evaluate the pairing on a class $[x]\in \DK(B)$ we
choose a representative $x$ which satisfies property $Y^+_e$ with generator $\Gen{+}$ and set  
$$E=e\hot 1, \quad \sS  = \Gen{+}\hot\kg, \quad X= x\hot 1.$$
The boundary term $\frac12\int E X (dX)^{n}$ vanishes as $\int\circ\jmath_1$ vanishes when evaluated on elements of the image of $j$. We have 
$X d\sS P_k(dX,d\sS ) = xd\Gen{+} \tilde P_k(dx,d\Gen{+})\hot \kg^{n-k+1}$ where
 $$ \tilde P_k(a,b) = (-1)^k\sum_{\pi\in S_{n,k}} 
\mbox{\rm sign}(\pi) c^\pi_1\cdots c^\pi_n, \quad
c^\pi_i = \left\{
\begin{array}{cl}
a & \mbox{if } \pi^{-1}(i)\leq k \\
b & \mbox{if } \pi^{-1}(i) > k
\end{array}\right.
$$
Now
$$\jmath_1(X\sS P_k(dX,d\sS)) = \left\{ \begin{array}{cl}
\kappa x\sS^{(+)} \tilde P_k(dx,d\sS^{(+)}) & \mbox{if $n-k$ is even}  \\
0 & \mbox{otherwise}
\end{array}\right. .
$$
from which we deduce that, if $e$ is homotopic to its negative and $de=0$, 
$$\Delta_{\xi\#\jmath_1}^{j}:\ker j_*\cap \DK(B)\to
\C / \langle \xi\#\jmath_2, \DK(B\hot Cl_{0,2})\rangle$$
 is given by 
\begin{equation} \label{eq-formula-even}
\Delta_{\xi\#\jmath_1}^{j} ([x])  \eqm 
\frac{1}2 \left(  \sum_{k=0\atop n-k\:even}^n  \beta_k \int_{B} x\Gen{+} \tilde P_k(dx,d\Gen{+}) - \delta^{ev}_{n}\int_{B} e\Gen{+}  (d\Gen{+})^n \right) 
\end{equation}
where $\delta^{ev}_n =1$ if $n$ is even and $0$ otherewise. 
Note that the above expression changes sign if one replaces $\sS^{(+)}$ by $-\sS^{(+)}$. Hence $2\Delta_{\tilde \xi}^{j} ([x])  \eqm 0$ showing that $\im \Delta_{\tilde \xi}^{j}$ contains only $2$-torsion elements. If $d\sS^{(+)} = 0$ the formula simplifies enormously,
\begin{equation} \label{eq-formula-even-b}
\Delta_{\xi\#\jmath_1}^{j} ([x])  \eqm 
\frac{1}2  \int_{B}\Gen{+} \left(x (dx)^n - \delta_{n 0} e \right).
\end{equation}
The group which is quotiented out is 
$ \langle \xi\#\jmath_2, \DK(B\hot Cl_{0,2})\rangle = \langle \xi , K_2(B) \rangle$, is it generated by the elements of the form 
$$\int_{B} \left(x' (dx')^n - \delta_{n 0} e\right) $$
where $x' $ is an osu in $B\hot Cl_{0,2}$.

\subsubsection{Odd \tv\ pairing}

We come back to the situation where the inclusion map is given by the complexification,
$\varphi=c$, $\tilde B$ the complexification of a real graded \CA\ $B$. 
Now $\tilde\xi$ is simply $\xi$, the character of an $n$-dimensional cycle over $B_\C$ whose \tosi\ is $\ss=1$.
We choose a representative $x$ which satisfies property $Y^-_e$ with generator $\Gen{-}$ and set  
$$E=e, \quad \sS=\Gen{-},\quad X= x.$$ 
Note that $P_k(dx,d\Gen{-})$ is self-adjoint and
$\tilde\rs ( P_k(dx,d\Gen{-})) = (-1)^{n-k} P_k(dx,d\Gen{-})$. 
Hence
\begin{eqnarray*}
\int \tilde \rs (x\Gen{-} P_k(dx,d\Gen{-}))^* &=& 
(-1)^{n-k+1} \int (x\Gen{-} P_k(dx,d\Gen{-}))^* \\
&=& (-1)^{n-k}  \int x\Gen{-} P_k(dx,d\Gen{-})
\end{eqnarray*}
from which we conclude that the expression vanishes if $\ss\neq (-1)^{n-k} $, that is,
$n-k$ is odd.
It follows that
$\Delta^c_{\xi}:\ker c_*\cap \DK(B) \to \C / \langle \xi\#\jmath_1, \DK(B\hot \C l_1)\rangle
$ is given by
\begin{equation} \label{eq-formula-odd}
\Delta_\xi^{c} ([x]) \eqm
\frac{\kappa^{-1}}2 \left(  \sum_{k=0\atop n-k\: even}^n \beta_k\int_{B} x\Gen{-}  P_k(dx,d\Gen{-}) - \delta^{ev}_{n}\int_{B} e\Gen{-}  (d\Gen{-})^n \right) 
\end{equation}
Indeed, the boundary term $\frac12\int e x (dx)^{n}$ vanishes as the \tosi\ is $+1$ 
(see Cor.~\ref{cor-V}).  
  As above we see that $2\Delta_\xi^{c} ([x]) \eqm 0$, that is, the elements of $\im \Delta_\xi^{c}$ are $2$-torsion. If $d\sS^{(-)} = 0$ the formula simplifies enormously,
\begin{equation} \label{eq-formula-odd-b}
\Delta_{\xi}^{c} ([x])  \eqm 
\frac{\kappa^{-1}}2  \int_{B} \left( x\Gen{-} (dx)^n - \delta_{n 0} e \right).
\end{equation}
The group which is quotiented out is 
$ \langle \xi\#\jmath_1, \DK(B_\C\hot \Cl_{1})\rangle = \langle \xi , K_0(B_\C)\rangle$, is it generated by the elements of the form 
$$\int_{B_\C} \jmath_1 \left(x' (dx')^n - \delta_{n 0} e\right) $$
where $x' $ is an osu in $B_\C\hot \Cl_1$.

\newcommand{\kk}{\gamma}
\renewcommand{\hs}{h}
\newcommand{\TRI}{\mathrm{TRI}}
\section{Explicit calculations for a simple class of periodic tight binding models}
\label{sec-simple-appl}
We consider here a class of periodic models for which the pairings can be computed explicitely. This allows us also to demonstrate that our formulae reproduce the known formulae in the literature for these models and hence the Kane-Mele and the Fu-Kane-Mele invariant. 

After a Bloch transformation a periodic $d$-dimensional tight binding model 
(without external magnetic field) is described by a self-adjoint Hamiltonian on 
$L^2(\TM^d,\C^N)$. Here $\TM^d$ is the Brillouin zone and $\C^N$ an internal Hilbert space for the degrees of freedom at the lattice sites (including spin). 
We consider here Hamiltonians of the form
\begin{equation}\label{eq-Ham}
 h = \sum_{i=0}^d \kk_i \hs_i 
 \end{equation}
where $\kk_i\in M_N(\C)$, $i=0,\cdots,d$ are pairwise anticommuting self-adjoint unitary matrices, 
the $\hs_i\in C(\TM^d,\R)$ act as multiplication operators by
$$\hs_i(k) = \sin(k_i)\:\:\mbox{\rm for $i>0$},\quad \hs_0(k) = m(k), $$
and $m\in C(\TM^d,\R)$ is an even real function of class $C^1$. The spectrum of such an operator is rather simple as $h^2 = \sum_{i=0}^d b_i^2 \one_N$. Hence $h$ is invertible if and only if $m(k)\neq 0$ for all $k\in \TRI := \{k\in\TM^d\,|\,\forall i : k_i\in \{0,\pi\}\}$.
Under this condition $h$ defines a
topological phase w.r.t.\ a observable algebra $A$. The choice of $A$ has to be made on physical grounds and we adopt here the point of view that the topological phase is protected by the lattice symmetry which means that deformation is only allowed in the algebra of periodic operators. The observable algebra is therefore 
\begin{equation}\label{eq-alg}
A=C(\TM^d,M_N(\C)).
\end{equation}
It will be fruitful to consider the matrices $\gamma_i$ as the images of the generators $\rho_i$ of the (complex) Clifford algebra of $d+1$ generators in a representation on $\C^N$, i.e.\
$\gamma_i = \varphi(\rho_i)$ for some algebra morphism $\varphi:\Cl_{d+1} \to M_N(\C)$. 

A protecting symmetry playing the role of chiral symmetry can be introduced by specifying a grading on $M_N(\C)$ which then extends pointwise to $A$. 
For the Hamiltonians of the form (\ref{eq-Ham}) we choose the grading in such a way that the $\gamma_i$ are odd. In other words the representation morphism $\varphi$ becomes a morphism of graded algebras.

When further protected by a real symmetry and/or a chiral symmetry the above models are simple examples of  insulators with topologically non-trivial phases without external magnetic field. They have been studied for instance in \cite{Golterman,RosenbergFranz,PS,QHZ}, the first two in a field theory context.
But also from the purely mathematical perspective the above models are interesting, as with a particular choice for $m$ they correspond to the Bott element on $\TM^d$. 
The Bott element on the sphere $S^d=\{y\in\R^{d+1}:\sum_{i=0}^d y_i^2=1\}$ is the van Daele class of the \osu\ $b_{S^d} \in C(S^d,\Cl_{d+1})$,
$$b_{S^d}(y) = \sum_{i=0}^d \kg_i y_i$$
where $\rho_i$ are the self-adjoint generators of the graded Clifford algebra $\Cl_{d+1}$.
It is known to be the generator of $\DK(C_0(\R^d,\Cl_{d+1}))$, which is $KU_0(C_0(\R^d))$ for even and $KU_1(C_0(\R^d))$ for odd $d$. 
The Bott element on $\TM^d$ is the van Daele class of the pull-back of $b_{S^d}$ under the composition $f=f_2\circ f_1$
of maps $\TM^d = ([-\pi ,\pi ] / \sim)^d \stackrel{f_1} \to \R^d \cup \{\infty\} \stackrel{f_2}\to S^d\subset \R^{d+1}$, where $f_1$ stretches the fundamental domain $[-\pi,\pi]^d$ and identifies the $d-1$-skeleton of the torus with the point at infinity $\{\infty\}$,
$$ f_1(k_1,\cdots,k_d) = \left( \tan \frac{k_1} 2 ,  \cdots, \tan \frac{k_d} 2\right) $$
and $f_2$ is a version of the inverse stereographic projection,
$$f_2(x_1,\cdots,x_d) = \left(\frac{1-x^2}{1+x^2},\frac{2 x_1}{1+x^2} ,\cdots ,\frac{2 x_d}{1+x^2}\right).$$
The pull back of $b_{S^d}$ is by definition the \osu\
$f^*(b_{S^d}) = \sum_{i=0}^d \kg_i f_i$, and, for $i\geq 1$,
\begin{eqnarray*}
f_i(k_1,\cdots,k_d) &=&  
\frac{A_i(k)}{(1-d)A(k) +\sum_{j=1}^d A_j(k)}\,\sin k_i \\
f_0(k_1,\cdots,k_d) &=&  
\frac{(1+d)A(k) -\sum_{j=1}^d A_j(k)}{(1-d)A(k) +\sum_{j=1}^d A_j(k)}
\end{eqnarray*}
where $A_i = \sum_{j\neq i} \cos^2 \frac{k_j}2$ and $A = \sum_{j = 1}^d \cos^2 \frac{k_j}2$.
For $i\geq 1$ the expression $f_i(k)$ vanishes only at $k=0$ and on the boundary $\partial [-\pi ,\pi ]^d$. On the other hand $f_0(0)=1$ while on the boundary $f_0(k)$ is $-1$. 
Note that $1-d+ \sum_{i=1}^d \cos k_i$ takes the value $1$ at $k=0$ and is strictly negative at all other points of $\TRI$. 
We can therefore deform the factor $\frac{A_i(k)}{(1-d)A(k) +\sum_{j=1}^d A_j(k)}$ in $f_i(k)$   along a straight line homotopy  to $1$ and then 
$\frac{(1+d)A(k) -\sum_{j=1}^d A_j(k)}{(1-d)A(k) +\sum_{j=1}^d A_j(k)}$ along a straight line homotopy to $1-d+ \sum_{i=1}^d \cos k_i$ 
without closing the gap in the spectrum of $f^*(b_{S^d})$. 
Hence $f^*(b_{S^d})$  is homotopic in the set of invertible odd self-adjoint elements of $C(\TM^d,\Cl_{d+1})$ to
\begin{equation}\label{eq-Bott-torus}
b_{\TM^d}(k) : =  \kg_0\left(1-d+ \sum_{i=1}^d\cos k_i \right)+
\sum_{i=1}^d\kg_i \sin k_i  .
\end{equation}
The van Daele class of its spectral flattening $\hat b_{\TM^d} = b_{\TM^d}|b_{\TM^d}|^{-1}$ is the Bott element on the torus $\TM^d$.

We associate to $b_{\TM^d}$ a Hamiltonian of the type (\ref{eq-Ham}) by considering two types of representations for $\Cl_{d+1}$. If $d$ is odd then we consider the Clifford algebra as graded (the grading is given by the grading operator $\Gamma_{d+1}$) and define $h_{Bott}^{(d)}=\varphi_{odd}(b_{\TM^d})$ using the bijective representation $\varphi_{odd}:\Cl_{d+1}\to M_N(\C)$ where $N = 2^{\frac{d+1}2}$. Interpreting the grading operator as the generator of a chiral symmetry $h_{Bott}^{(d)}$ is thus anti-invariant under chiral symmetry.
Its topological phase is classified by its class in $\DK(C(\TM^d)\otimes \Cl_{d+1})$ \cite{Kel1}. 

If $d$ is even then we use the bijective representation 
$\Cl^u_{d}\to M_N(\C)$ of the ungraded Clifford algebra $\Cl^u_d$ with $N=2^{\frac{d}2}$ to obtain an ungraded representation $\varphi_{ev}$ of $\Cl_{d+1}$ on $\C^N$, namely $\varphi_{ev}(\kg_i) = \kk_i$ where $\gamma_1,\cdots,\gamma_d$ are the representatives of the generators of $\Cl^u_d$ and 
$$  \gamma_0 = (-i)^{\frac{d} 2} \gamma_1\cdots\gamma_d.$$ 
Again we define $h_{Bott}^{(d)}=\varphi_{ev}(b_{\TM^d})$ in this representation, but this time it has no chiral symmetry.
In both cases
$$h_{Bott}^{(d)} = \kk_0\left(1-d+ \sum_{i=1}^d\cos k_i\right) + \sum_{i=1}^d\kk_i \sin k_i . $$
We call $h_{Bott}^{(d)}$ the Bott Hamiltonian. 
It can be interpreted as a tight binding operator with nearest neighbor interaction. For odd $d$ it has chiral symmetry. These models have been studied in a more general context (with external magnetic field and contracting disorder) in \cite{PS}, for $d=2$ the model is also referred to as the half BHZ model or the Qui-Wu-Zhang-model\footnote{at a particular value of its parameter $u$} \cite{Bernevig}.

Note that the above can also be formulated as follows: If $d$ is odd then $\Cl_{d+1}$ is isomorphic to $M_{2^{\frac{d+1}2}}(\C)$ with standard even grading and under the isomorphism $b_{\TM^d}$ becomes $h_{Bott}^{(d)}$. If $d$ is even then 
$$ \kg_i \mapsto \gamma_i\otimes \rho,\quad \mbox{with }  \gamma_0 = (-i)^{\frac{d} 2} \gamma_1\cdots\gamma_d$$ 
defines a graded isomorphism between $\Cl_{d+1}$ and the tensor product $\Cl^u_{d}\otimes \Cl_1$ and $b_{\TM^d}$ becomes $h_{Bott}^{(d)}\otimes\rho$. 
\subsection{Real protecting symmetries}\label{sec-RPS}
A real protecting symmetry playing the role of TRS or PHS can be introduced as follows. Choose a real structure $\rs'$ on $M_N(\C)$ and then define the real structure on 
$A=C(\TM^d, M_N(\C))$ through
$$\rs(a)(k) := \rs'(a)(-k),$$
for $a \in A$. $h$ has TRS if $\rs(h)=h$ and PHS if $\rs(h) = -h$. 
Up to conjugacy, and if $N=2K$ is even, there are only two distinct choices of real structures $\rs'$ on $M_N(\C)$ which lead to a simple real subalgebra. These are, entrywise complex conjugation which we denote $\cc$ , and $\rh = \Ad_{\sigma_2\otimes \id_K}\circ \cc$. In the first case the real subalgebra is $M_N(\R)$ and the second it is $M_{K}(\HM)$ where $\HM$ is the algebra of quaternions. The algebra $M_N(\R)$ is Morita equivalent to $\R$ and we call the algebra $A^\rs$ of real type and the symmetry even\footnote{Complex conjugation induces the reference real structure $\rf$ on $A$ in the sense of \cite{Kel1}.}. Likewise 
$M_{K}(\HM)$ is Morita equivalent to $\HM$ and we say that $A^\rs$ is of quaternionic type and the symmetry odd. 

For Hamiltonians of the type (\ref{eq-Ham}) we will more specifically define (or constrain) the real structure by means of a real structure on $\Cl_{d+1}$ demanding that the representation preserves the real structure. 
There are two useful real structures for this purpose. The first one declares $\rho_0$ to be real and all $\rho_i$ with $1\leq i\leq d$ to be imaginary. We denote this real structure by $\rcl_{1,d}$. 
Hence $\gamma_0$ is real and all $\gamma_i$, $i\geq 1$ imaginary and $h$ is invariant under this real structure ($h$ has TRS). 
The second option is to declare $\rho_0$ to be imaginary and all $\rho_i$ with $1\leq i\leq d$ to be real. This real structure is denoted $\rcl_{d,1}$.  
Then $h$ is anti-invariant (it has PHS). It should be noted that these prescriptions do not fix the real structure on $M_N(\C)$ in case the representation is not surjective. We now analyse these options for the two representations used so far. 
 
If $d$ is odd the graded representation 
$\varphi_{odd}:\Cl_{d+1}\to M_N(\C)$, $N=2^{\frac{d+1}2}$ is bijective. Therefore the real structure on $M_N(\C)$ is uniquely determined by the real structure on $\Cl_{d+1}$, and the real subalgebra $M_N(\C)^{\rs'}$ is simply isomorphic to $Cl_{1,d}$, in the TRS case, or to $Cl_{d,1}$ in the PHS case. Since $\rcl_{d,1} \circ\rcl_{1,d}$ is the grading automorphism, the two structures are not independent and for the classification of the topological phases with two protecting symmetries it suffices to take into account only one real structure, let's say $\rcl_{1,d}$ (TRS), and whether the grading operator $\Gamma_{d+1}$ is real or imaginary under $\rcl_{1,d}$ \cite{Kel1}. Clearly $\Gamma_{d+1}$ is real if $d+\mu(d+1)$ is even, which is the case for $d=1$ mod $4$. Likewise, $\Gamma_{d+1}$ is imaginary if $d=3$ mod $4$. 
The real subalgebras $M_N(\C)^{\rs'}$ are, $Cl_{1,1} \cong M_2(\R)$,
 $Cl_{1,3} \cong M_2(\HM)$, $Cl_{1,5} \cong M_4(\HM)$, $Cl_{1,7} \cong M_{16}(\R)$.
From this we can conclude that the Bott-Hamiltonian in $d=1$ has even TRS with real chiral symmetry, in $d=3$ odd TRS with imaginary chiral symmetry, in $d=5$ odd TRS with even chiral symmetry, and in $d=7$ even TRS with imaginary chiral symmetry.

In the case that $d$ is even we have considered the representation 
$\varphi_{ev}:\Cl^u_{d+1} \to M_N(\C)$ with $N=2^{\frac{d}2}$. In particular there is no grading.
This representation is surjective but not faithful so that we have constraints on the dimension in which $\varphi_{ev}$ can preserve $\rcl_{1,d}$ or $\rcl_{d,1}$. Indeed, since $\gamma_0 = (-i)^\frac{d}2\gamma_1\cdots\gamma_d$ we can only have $\rcl_{1,d}$ (TRS) if $d+\frac{d}2$ is even, which is the case for $d=0$ mod $4$. Likewise we can only have $\rcl_{d,1}$ (PHS) if $d=2$ mod $4$.  The real subalgebras $M_N(\C)^{\rs'}$ are, $Cl^u_{0,0}\cong\R$, $Cl^u_{0,4}\cong M_2(\HM)$, for TRS, and $Cl^u_{2,0}\cong M_2(\R)$, $Cl^u_{6,0}\cong M_4(\HM)$, for PHS. From this we can conclude that the Bott-Hamiltonian in $d=0$ has even TRS, in $d=2$ even PHS, in $d=4$ odd TRS, , and in $d=6$ odd PHS.

Imposing TRS via a real structure $\rs$, the topological phase of $h$ is classified by the van Daele class $[h\otimes\rho]\in \DK(A^\rs \otimes Cl_{1,0})$ if the system has no chiral symmetry, by
$[h]\in \DK(A^\rs \otimes Cl_{1,1})$ if the generator of chiral symmetry is real, and by
$[h]\in \DK(A^\rs \otimes Cl_{0,2})$ if the generator is imaginary.
Imposing PHS by a real structure $\rs$ the topological phase is classified by
$[h\otimes\rho]\in \DK(A^\rs \otimes Cl_{0,1})$ if there is no chiral symmetry. 
PHS with chiral symmetry is related to TRS with chiral symmetry as we mentionned above. 
For these results see \cite{Kel1}.

\subsection{Top dimensional Chern characters}
We will pair the above van Daele classes with 
the $d$-dimensional Chern character over the Brillouin zone $\TM^d$. It is defined as follows.

Consider the complexified exterior differential algebra over $\TM^d$ which we denote by 
$(\Omega(\TM^d),d)$. It is $\Z$-graded, but trivially $\Z_2$-graded. 
Integration of $d$-forms $\int_{\TM^d}$ is a graded trace on $\Omega(\TM^d)$.    
$(\Omega(\TM^d),d,\int_{\TM^d})$ is thus a cycle over $C(\TM^d)$,
the algebra of smooth functions over $\TM^d$ is a domain algebra for the cycle.

The $*$-structure on $C(\TM^d)$ which is given by complex conjugation extends uniquely to a $*$-structure on $\Omega(\TM^d)$. The $*$-map flips the order of a product of differential forms. 
We thus have $(dk_1\cdots dk_d)^*= (-1)^{\nn(d)} dk_1\cdots dk_d$. The real structure $\rs$ flips the sign of the coordinates $k_i$ and hence we have $\rs(dk_i) = -dk_i$. It follows that
$$\int_{\TM^d} \rs(h_1 dh_2 \cdots d h_d)^* = (-1)^{d+\nn(d)} \int_{\TM^d} h_1 dh_2 \cdots d h_d $$
and hence the cycle has \tosi\ $\ss = (-1)^{d+\nn(d)}$.
Since $\Omega(\TM^d)$ is trivially $\Z_2$-graded the cycle has even parity.  
It follows from Cor.~\ref{cor-cond-triv} that this cycle can
pair non-trivially with $KO_i(A^\rs)$ only if $\mu(1-i)-\mu(d)$ and $d-i$ are even. 
These conditions are equivalent to $d=i$ mod $4$.

We denote the character of 
$(\Omega(\TM^d),d,\int_{\TM^d})$ by $\ch_d'$ and extend it to $A\otimes \Cl_k$ in the way we introduced above, notably by $\ch'_d\#\jmath_k$.
We discuss the two kind of pairings introduced above, the Connes pairing and the torsion valued pairing.

\subsection{Connes pairing with $\ch_d$} 
We compute the Connes pairing of $\ch'_d$ with the van Daele class $[x]\in\DK(C(\TM^d)\otimes \Cl_{d+1})$ where $x = \sum_{i=0}^d \rho_i \hat \hs_i$ with $\hs_i$ as for (\ref{eq-Ham}). 
We consider $d>0$, as the case $d=0$ has already been looked at above.
The pairing is given by Def.~\ref{def-pairing-higher}
$$\langle\ch'_d,[x]\rangle =  \int_{\TM^d} \jmath_{d+1} x (d x)^d = \kappa^{d+1} 
\int_{\TM^d}\epsilon^{i_0\cdots i_d} \hat \hs_{i_0} d\hat \hs_{i_1}\cdots d\hat \hs_{i_d}$$
where $\epsilon^{i_0\cdots i_d}$ is the totally antisymmetric $\epsilon$-tensor and we have used the sum convention. 
As before, $\hat h= h |h|^{-1}$ is the spectrally flattened Hamiltonian and $\hat b_i = b_i |h|^{-1}$.

If $d$ is odd then $\Cl_{d+1}$ is isomorphic to $M_N(\C)$ with $N = 2^{\frac{d+1}2}$ and grading defined by the grading operator $\Gamma_d$. We thus have $x = \hat h$ with $h$ as in (\ref{eq-Ham}) and $\gamma_i=\rho_i$ so that, by Lemma~\ref{lem-j-tr} 
$$\langle\ch'_d,[x]\rangle = 
\int_{\TM^d} \Tr_N \hat h (d\hat h)^d. $$
On the other hand if $d$ is even, then $\Cl^u_d$ is is isomorphic to $M_N(\C)$ with $N = 2^{\frac{d}2}$ (with trivial grading). We 
then have $x = \hat h\otimes\rho$ with $h$ as in (\ref{eq-Ham}) and 
$\rho_i=\gamma_i\otimes\rho$ 
where and $\gamma_0 = (-i)^{\frac{d} 2} \gamma_1\cdots\gamma_d$. This leads to
$$\langle\ch'_d,[x]\rangle 
= \kappa \int_{\TM^d} \Tr_N \hat h (d\hat h)^d.$$
\begin{lemma} \label{lem-jup}
Let $d$ be any strictly positive integer. For $i=0,\cdots d$ let $b_i$ be differentiable functions on a manifold and
$|b|:= \sqrt{\sum_{i=0}^d |b_i|^2}$. Suppose that $|b|$ is invertible and let $\hat b_i = b_i |b|^{-1}$. Then
$$\epsilon^{i_0\cdots i_d} \hat \hs_{i_0} d\hat \hs_{i_1}\cdots d\hat \hs_{i_d}
= |b|^{-d-1}\epsilon^{i_0\cdots i_d}  \hs_{i_0} d \hs_{i_1}\cdots d \hs_{i_d}.$$
\end{lemma}
\begin{proof}
Let $x = d \hs_{i_1}\cdots d \hs_{i_j}$ for some $0\leq j < d$. 
Since $\hs_i$ commutes with $\hs_j$ and $d \hs_{j}$ we have $\hat \hs_1 x  \hs_2 = \hat \hs_2 x  \hs_1$.Therefore
\begin{eqnarray*}
\hat \hs_1  x d\hat \hs_2 - \hat \hs_2 x d\hat \hs_1 &=& (\hat \hs_1 x d \hs_2 - \hat \hs_2 x d \hs_1)|b|^{-1} + 
(\hat \hs_1 x \hs_2 - \hat \hs_2 x \hs_1)d |b|^{-1}\\
& = & |b|^{-1}(\hat \hs_1 x d \hs_2 - \hat \hs_2 x d \hs_1).
\end{eqnarray*}
Hence $\epsilon^{i_0\cdots i_d} \hat \hs_{i_0} d\hat \hs_{i_1}\cdots d\hat \hs_{i_d}
= |b|^{-1}\epsilon^{i_0\cdots i_d}  \hat \hs_{i_0} d\hat \hs_{i_1}\cdots d\hat \hs_{i_{d-1}} d \hs_{i_d}$ and the statement follows iteratively.
\end{proof}
The following result can be found in literature (see, for instance, \cite{Golterman,QHZ,PS}. For the convenience of the reader we provide the details of the calculation.  
\begin{prop}
Let $d$ be any positive integer, and $b_i:\TM^d\to \R$ be given by $b_i(k) = \sin(k_i)$ if $1\leq i\leq d$ and 
$b_0(k) = m(k)$ for an even real function which does not vanish on the discrete set $\TRI := 
\{k\in\TM^d\,|\,\forall i : k_i\in \{0,\pi\}\}$. Then 
$$\int_{\TM^d}\epsilon^{i_0\cdots i_d} \hat \hs_{i_0} d\hat \hs_{i_1}\cdots d\hat \hs_{i_d}
= d! a_d a'_d \sum_{k'\in \TRI} \mbox{\rm sign}(m(k'))\prod_{i=1}^d \cos(k'_i) $$
where $a_d$ is the surface of the $d$-ball of radius $1$ in $\R^d$ and 
$a'_d=\frac{\pi}2\alpha_{d-1,d-1}$ (c.f.\ (\ref{eq-alpha})).
\end{prop}
We abbreviate 
\newcommand{\Ih}{I(m)}
$$\Ih = \frac12 \sum_{k'\in \TRI} \mbox{\rm sign}(m(k'))\prod_{i=1}^d \cos(k'_i)$$
and note that $I(m)$ must be integer, as it is one half of a sum of an even number of $\pm 1$'s. 
\begin{proof}
We know from the general theory that the integral on the left side is homotopy invariant, as long as $|b|$ remains invertible. We may therefore include a parameter $t>0$, consider
$b_0 = t m$ instead and perform the calculation in the limit $t\to 0$. Let $\delta>0$. 
Then, away from all $\delta$-balls with center in $\TRI$ we get
$$\lim_{t\to 0} \int_{\TM^d\backslash B_\delta(\TRI)} |b|^{-d-1}
\epsilon^{i_0\cdots i_d}  \hs_{i_0} d \hs_{i_1}\cdots d \hs_{i_d} =0
$$
as $\lim_{t\to 0} \hs_0 = \lim_{t\to 0} d\hs_0 = 0$ and $\lim_{t\to 0} |b| \geq \frac12\delta$ for small enough $\delta$.
On $B_\delta(k')$ with $k'\in\TRI$  we approximate up to order $\delta$ 
$$\sin(k_i)\cong  k_i-k'_i,\quad m(k) \cong m(k') .$$
Indeed, since $m$ is even,
the first order term in the Taylor expansion of $m$ at points of $\TRI$ vanishes. 
Since $\hs_i$ is of order $\delta$ if $i>0$ we have, up to order $\delta$ on $B_\delta(k')$
\begin{eqnarray*}
\epsilon^{i_0\cdots i_d}  \hs_{i_0} d \hs_{i_1}\cdots d \hs_{i_d} &\cong& \epsilon^{0 i_1\cdots i_d}  \hs_{0} d \hs_{i_1}\cdots d \hs_{i_d}\\
& \cong &
d! t m(k') \left(\prod_{i=1}^d \cos(k'_i)\right) dk_1\cdots dk_d 
\end{eqnarray*}
and hence,
up to order $\delta$,
\begin{eqnarray*}
\int_{B_\delta(k')} |b|^{-d-1}
\epsilon^{i_0\cdots i_d}  \hs_{i_0} d \hs_{i_1}\cdots d \hs_{i_d}
&\cong& d! \int_{B_\delta(k')}\frac{t m(k') 
\prod_{i=1}^d \cos(k'_i)}{(t^2 m^2(k')+ (k-k')^2)^\frac{d+1}2} dk_1\cdots dk_d \\
& = & d! \,\mbox{\rm sign}(m(k'))
\left(\prod_{i=1}^d \cos(k'_i)\right) \int_{0}^\delta 
\frac{\tilde t a_d r^{d-1}}{({\tilde t}^2 + {r}^2)^\frac{d+1}2} dr 
\end{eqnarray*}
where $a_d$ is the surface of the $d$-ball of radius $1$ and $\tilde t = t |m(k')|$.
We evaluate the integral in the limit $t\to 0$
$$ \int_{0}^\delta \frac{ t r^{d-1}}{({ t}^2 + {r}^2)^\frac{d+1}2} dr 
\stackrel{t\to 0}{\longrightarrow} 
\int_{0}^{+\infty} \frac{ r^{d-1}}{(1 + {r}^2)^\frac{d+1}2} dr = \int_0^{\frac{\pi}2} \sin^{d-1}(\theta) d\theta = \frac{\pi}2 \alpha_{d-1,d-1}.
$$
Putting everything together we obtain with Lemma~\ref{lem-jup} the result.
\end{proof}
For the Bott element on the torus we get $\Ih=1$, that is,
$\langle \ch'_d,b_{\TM^d}\rangle = \kappa^{d+1} d! a_d a'_d 2 $.
We therefore normalise 
\begin{equation}\label{eq-nor}
\ch_d  := \big(2\kappa^{d+1}  d! a_d a'_d \big)^{-1}\ch'_d
\end{equation}
calling $\ch_d$ the (top) standard Chern character of the $d$ dimensional torus. 
We thus see that, for 
a Hamiltonian of the form (\ref{eq-Ham}) $h=\sum_{i=0}^d \varphi(\rho_i) \hs_i$ ($\varphi=\varphi_{odd}$ or $\varphi_{ev}$ depending on whether $d$ is odd or even), the Connes pairing of $\ch_d$ with the corresponding van Daele class is equal to $I(m)$. 
Furthermore
\begin{cor}\label{cor-simple-pairings}
We have 
$$\langle \ch_d,KU_i(C(\TM^d)) \rangle = \left\{
\begin{array}{cl} 
\Z & \mbox{\rm if } i = d\: mod\: 2 \\
0 & \mbox{\rm otherwise}
\end{array}\right.
$$
and
$$\langle \ch_d,KO_i(C(\TM^d)^{\rf}) \rangle = \left\{
\begin{array}{cl} 
\Z & \mbox{\rm if } i = d\: mod\: 8 \\
2\Z & \mbox{\rm if } i = d+4 \: mod\: 8 \\
0 & \mbox{\rm otherwise}
\end{array}\right.
$$
where $\rf(f)(k) = \overline{f(-k)}$.
\end{cor}
\begin{proof} The corollary is a special case of Lemma~\ref{lem-simple-pairings} which we prove below.
\end{proof}

\subsection{Torsion valued expressions with $\ch_d$}
Here we consider the even and odd torsion valued pairings of $\ch_d$ with 
the $K$-theory class defined by a Hamiltonian $h\in A$ of the form (\ref{eq-Ham}) and $A$ as in (\ref{eq-alg}).
Again we see $h$ as the representative of the element 
$\sum_{i=0}^d\kg_i \hs_i \in C(\TM^d)\otimes \Cl_{d+1}$ in a representation of $\Cl_{d+1}$ on $\C^N$, however the representations and hence the Hamiltonians will be different from those of the last section; indeed this must be the case since the torsion valued pairing can only be non-trivial if the Connes pairing is trivial.

 \subsubsection{TRS but no chiral symmetry}\label{sec-oddTRS} 
Topological phases of hamiltonians which are invariant under a real structure $\rs$ (have TRS) but not protected by a chiral symmetry are classified by the $K$-group $\DK(A^\rs\otimes Cl_{1,0})$.  
As we have no protecting chiral symmetry we disregard the grading on $\Cl_{d+1}$ and denote the latter by $\Cl^u_{d+1}$. This algebra can be faithfully represented on $\C^N$ where $N = 2^{\nn(d+2)}$. We denote this representation by $\varphi^u:\Cl_{d+1}^u\to M_N(\C)$. It is, of course, only unique up to conjugation. 
As in Section~\ref{sec-RPS} we obtain a TRS invariant model by
considering on  $M_N(\C)$ a real structure $\rs'$ which extends the real structure $\rcl_{1,d}$ on $\Cl^u_{d+1}$. More precisely we consider $\kk_0$ to be real and all $\kk_i$ with $i>0$ to be imaginary. 

Given an $d$-dimensional cycle over $A$ of \tosi\ $\ss$, we extend it to a cycle over $A\otimes \Cl_1$ and conclude from Lemma~\ref{lem-p-ext} and 
conditions~(\ref{eq-cond+}) of Theorem~\ref{thm-TV-ev} that the even torsion valued expression can only be non-trivial if $\ss = -1$ and $d$ is even. Applied to our cycle on the Brillouin zone this means that $(-1)^{\mu(d)+d}=-1$ and $d$ is even. This is the case for $d=2$ mod $4$. Likewise with conditions~(\ref{eq-cond-}) of Theorem~\ref{thm-TV-odd} we find that the odd torsion valued expression can only be non-trivial if 
$(-1)^{\mu(d)+d}=1$ and $d$ is odd. This is the case for $d=3$ mod $4$. 
\bigskip

\paragraph{Let $d=2$ mod $4$} 
Then $\Gamma_{d+1}=(-i)^{\mu(d+1)}\gamma_0\cdots\gamma_d$ commutes with $h$ and is imaginary for the real structure $\rcl_{1,d}$. Choosing as trivial insulator $h_0=\kk_0$ we thus see that $h\otimes \rho$ and $h_0\otimes\rho$ commute with $\Gen{+}=\Gamma_{d+1}\otimes 1$ and hence satisfy property $Y^{+}_{\kk_0\otimes\rho}$. It follows that 
the torsion-valued expression $\Delta_{\ch'_d\#\jmath_2}^j([h\otimes\rho])$ 
is well-defined on van Daele classes of $B = A^\rs\otimes Cl_{1,0} $
and given by 
(\ref{eq-formula-even-b}) (note that 
$\tilde B = A^\rs\otimes Cl_{1,1}$ so that $\tilde \xi = \ch'_2\#\jmath_2$)
$$
\Delta^{j}_{\ch'_d\#\jmath_2} ([h\otimes\rho])  \eqm   \frac{1}2 
\int_{A^\rs}\jmath_1\Tr_N
\hat h \Gamma_{d+1} (d\hat h)^d \otimes\rho^{d+1} 
 \eqm    \frac{\kappa N}{2i}   \int_{\TM^d}
 \epsilon^{i_0 \cdots i_d} \hat \hs_{i_0} d\hat \hs_{i_1} \cdots d\hat \hs_{i_d}
$$
where $N = 2^{\frac{d+1}2}$. 
As $\frac{\kappa N}{2i\kappa^{d+1}}=1$ we obtain with our normalisation (\ref{eq-nor})
$$ \Delta^{(+)}_{\ch_d}([h]) := \Delta^{j}_{\ch_d\#\jmath_2}([h\otimes\rho]) \eqm \Ih .$$
Note that $\Ih = 1$ if we take 
$h=\varphi^u(b_{\TM^d})$ ($h$ is not the Bott hamiltonian $h^{(d)}_{Bott}=\varphi_{ev}(b_{\TM^2})$ as the representations $\varphi^u$ and $\varphi_{ev}$ are not isomorphic). 
The above expression for $ \Delta^{(+)}_{\ch_d}([h])$ has to be taken modulo the subgroup 
$$ \langle \ch_d\#\jmath_{3}, \DK(A^\rs\otimes Cl_{1,2})\rangle = \langle \ch_d, KO_2(A^\rs) \rangle.$$
This subgroup depends on the real structure $\rs$ 
 which, in turn, depends on the real structure $\rs'$ on $M_N(\C)$. As $d$ is even, $\rs'$ is not uniquely determined by the real structure $\rcl_{1,d}$ on the Clifford algebra and there are two inequivalent ways to extend it. If  
$M_N(\C)^{\rs'} = M_N(\R)$ (which means that TRS is even)
then $KO_2(A^\rs) = KO_2(C(\TM^d)^\rf)$ and, by Cor.~\ref{cor-simple-pairings},  
$\langle \ch_2, KO_2(C(\TM^2)^\rf) \rangle=\Z$ and $\langle \ch_6, KO_2(C(\TM^6)^\rf) \rangle=2\Z$. In the first case the torsion value pairing is trivial, in the second it is surjective onto $\Z_2$. On the other hand, if $M_N(\C)^{\rs'} = M_{\frac{N}2}(\HM)$ (TRS is odd) then
$KO_2(A^\rs) \cong KO_6(C(\TM^d)^\rf)$ and
 Cor.~\ref{cor-simple-pairings}  implies that $\langle \ch_2, KO_2(A^\rs) \rangle=2\Z$ whereas $\langle \ch_6, KO_2(A^\rs) \rangle=\Z$. Now the torsion value pairing is surjective onto $\Z_2$ in the first case and trivial in the second.

To compare the above with the invariant of Kane-Mele we consider now $d=2$ and use the representation of \cite{Bernevig},
$\varphi^u:\Cl^u_3 \to M_2(\C)\otimes M_2(\C)$,  
$\varphi^u(\rho_0) = 1\otimes \sigma_z$, $\varphi^u(\rho_1) = \sigma_z\otimes \sigma_x$,
and $\varphi^u(\rho_2) = 1\otimes \sigma_y$.
Furthermore we take 
the real structure $\rs=\Ad_{i\sigma_y\ot 1}\circ\cc$.
Then $h=\varphi^u(b_{\TM^2})$ is exactly what is referred to as the inversion symmetric Hamiltonian of HgTe in  \cite{Bernevig}[Chap.~11.3]. This shows that the Kane-Mele invariant is a special case of the torsion valued-pairing, simply because there are only two distinct strong topological phases in the Kitaev classification table \cite{Kitaev} and the inversion symmetric Hamiltonian of HgTe is known to have non-trivial Kane-Mele invariant. 

Also the representation $\varphi_{ev}:\Cl^u_3\to M_2(\C)$ is only unique up to conjugation, and when using $\varphi_{ev}(\rho_0) = \sigma_z$, $\varphi_{ev}(\rho_1) = \sigma_x$,
$\varphi_{ev}(\rho_2) = \sigma_y$ we see how $h=\varphi^u(b_{\TM^2})$ is related to the Bott hamiltonian $h^{(2)}_{Bott}=\varphi_{ev}(b_{\TM^2})$, namely
$$\varphi^u(b_{\TM^2}) = \begin{pmatrix} \varphi_{ev}(b_{\TM^2}) & 0 \\
0 &\rf( \varphi_{ev}(b_{\TM^2})) \end{pmatrix} .$$
Furthermore, the generator of $K_2(A^\rs)$ is $\begin{pmatrix} \varphi_{ev}(b_{\TM^2}) & 0 \\
0 &- \rf(\varphi_{ev}(b_{\TM^2})) \end{pmatrix} $. Here $\rf(f)(k) = \cc(f(-k))$ for a $2\times 2$ matrix valued function over $\TM^2$. 
\bigskip

\paragraph{Let $d=3$ mod $4$}
As $d$ the representation $\varphi^u$ is bijective and hence $\rs'$ uniquely determined by 
$\rcl_{1,d}$; the model has thus odd TRS  if $d=3$ and even TRS if $d=7$.
Now $\Gamma_{d+1}$ anti-commutes with $h$ of the form (\ref{eq-Ham}) and  is imaginary for the real structure $\rcl_{1,d}$. We thus see that $h\otimes \rho$ 
and $h_0\otimes\rho$ anti-commute with $\Gen{-}=\Gamma_{d+1}\otimes\rho$.
Again the relevant algebra is $B=A^\rs\otimes Cl_{1,0}$.
Hence $h\otimes \rho$ satisfies property $Y^{-}_{\kk_0\otimes\rho}$ and the torsion-valued pairing $\Delta_{\ch'_d\#\jmath_1}^c([h\otimes\rho])$ is well-defined. 
We obtain by (\ref{eq-formula-odd-b})
$$
\Delta^{c}_{\ch'_d\#\jmath_1}([h\otimes\rho]) \eqm  \frac{\kappa^{-1}}2 
\int_{\TM^d}\jmath_1 \Tr_N
\hat h \Gamma_{d+1} (d\hat h)^d \otimes\rho^{d+2} 
\eqm \frac{N}2  \int_{\TM^d}
 \epsilon^{i_0\cdots i_d} \hat \hs_{i_0} d\hat \hs_{i_1}\cdots d\hat \hs_{i_d}
$$
where $N=2^{\mu(d+2)} = 2^{\frac{d+1}2}$. Hence $\frac{N}{\kappa^{d+1}} = 1$ and 
we obtain from (\ref{eq-nor})
$$ \Delta^{(-)}_{\ch_d}([h]) := \Delta^{c}_{\ch_d\#\jmath_1}([h]) \eqm \frac{1}2\Ih $$
and, again, $I(m)=1$ is obtained by using the Hamiltonian $h=\varphi^u(b_{\TM^d})$.
These values are understood modulo 
$\langle \ch_d\#\jmath_2, \DK(A\otimes \Cl_{2})\rangle =  \langle \ch_d, KU_1(A) \rangle = \Z$. The torsion valued pairing in $d=3$ mod $4$ is thus surjective onto $\Z_2$. 
If $d=3$ mod $8$ the algebra $A^\rs$ is of quaternionic type whereas if $d=7$ mod $8$ it is of real type.
With similar arguments as in  the $2$-dimensional case 
we find that the $3$-dimensional strong Fu-Kane-Mele invariant is a special case of the odd torsion valued pairing. Indeed, when using the representation 
$\varphi^u:\Cl^u_4\to M_N(\C)$ of \cite{Bernevig} (extending the above by $\varphi^u(\rho_3) = \sigma_x\otimes \sigma_x$) we find that 
$h=\varphi^u(b_{\TM^3})$ is the Hamiltonian used in \cite{RosenbergFranz} as an exemple of a Hamiltonian with non-trivial strong Fu-Kane-Mele invariant,
and there are only two distinct strong phases for $3$-dimensional odd TRS invariant Hamiltonians.
 
\subsubsection{PHS but no chiral symmetry}  
Topological phases with this protecting symmetry are classified by $\DK(A^\rs\otimes Cl_{0,1})$. 
We use again the ungraded representation $\varphi^u:\Cl^u_{d+1}\to M_N(\C)$  with $N=2^{\mu(d+2)}$, but now with real structure $\rcl_{d,1}$ on $\Cl^u_{d+1}$ for which $\kk_0$ is imaginary and all other $\kk_i$ real. 
For odd $d$ this determines uniquely the real structure $\rs'$ on $M_N(\C)$ whereas for $d$ even we have a choice of extension.

A similar analysis as in the last section
can be performed to determine the dimensions $d$ for which the torsion valued pairing can be non-trivial. The different real structure which one has for PHS enters into Lemma~\ref{lem-p-ext}  and leads with the conditions~(\ref{eq-cond+}) and (\ref{eq-cond-}) 
now to the result that the even torsion valued expression can only be non-trivial if $d=0$ mod $4$ and the odd one if $d=1$ mod $4$.

We consider first the case $d=0$ mod $4$.
$\Gamma_{d+1}$ commutes with $h$ and is imaginary. We choose $h_0=-\kk_0$ and observe that $h\otimes\rho$ and $h_0\otimes \rho$ commute with the imaginary $\Gen{+}=\Gamma_{d+1}\otimes 1$. Thus $h\otimes\rho$ satisfies thus $Y^{+}_{\kk_0\otimes\rho}$ but the relevant algebra is now $B=A^\rs\otimes Cl_{0,1}$. 
It follows that $\Delta_{\ch'_d\#\jmath_1}^j([h\otimes\rho])$ 
is well-defined and given by 
(\ref{eq-formula-even-b}) 
$$\Delta^{j}_{\ch'_d\#\jmath_2} ([h\otimes\rho])  \eqm  \frac12 
\int_{\TM^d}\jmath_1 \Tr_N (\hat h -\hat h_0) \Gamma_{d+1} (d\hat h)^d \otimes\rho^{d+1} 
$$
If $d=0$ then $N=2$, $\Gamma_1 = \gamma_0$, $h = m\gamma_0$ so that 
$$\Delta^{j}_{\ch'_0\#\jmath_2} ([h\otimes\rho])  \eqm  \frac{\kappa}2 \Tr_2((\mbox{\rm sign}(m)+1)\gamma_0^2) \eqm  \kappa(\mbox{\rm sign}(m)+1)$$
If $d>0$ then 
$$
\Delta^{j}_{\ch'_d\#\jmath_2} ([h\otimes\rho])  \eqm  \frac{N\kappa } 2  \int_{\TM^d} 
 \epsilon^{i_0\cdots i_d} \hat \hs_{i_0}  d\hat \hs_{i_1}\cdots d\hat \hs_{i_d}
$$
As $N\kappa = \pm2\kappa^{d+1}$ we obtain with our normalisation
$$ \Delta^{(+)}_{\ch_d}([h]) := \Delta^{j}_{\ch_d\#\jmath_2}([h\otimes\rho]) \eqm \Ih $$
where $\Ih  = \frac12 (\mbox{\rm sign}(m)+1)$ if $d=0$.
The value $\Ih\eqm 1$ is obtained for $h = \varphi^u(b_{\TM^d})$.
We have to quotient out 
$ \langle \ch_d\#\jmath_3, \DK(A^\rs\otimes Cl_{0,3})\rangle =  \langle \ch_d, KO_4(A^\rs) \rangle$ which depends on the extension $\rs'$ of $\rcl_{d,1}$ to $M_N(\C)$. If $d=4$ mod $8$ and $M_N(\C)^{\rs'} = M_N(\R)$ (PHS is even), or, $d=0$ mod $8$ and $M_N(\C)^{\rs'} = M_\frac{N}2(\HM)$ (PHS is odd) then $\langle \ch_d, KO_4(A^\rs) \rangle =\Z$ whereas 
for the other two combinations $\langle \ch_d, KO_4(A^\rs) \rangle = 2\Z$. Hence the torsion valued pairing is surjective onto $\Z_2$ if $d=4$ mod $8$ and PHS is odd, or, $d=0$ mod $8$ and PHS is even. Note that the result for $d=0$ corresponds to the results of Section~\ref{sec-examples1} where the torsion valued pairing on $KO_2(\R)$ was computed.

We come to $d=1$ mod $4$. 
As the representation $\varphi^u$ is surjective in this case the real structure on $A$ is uniquely determined by $\rcl_{d,1}$. If $d=1$ mod $8$ the algebra $A^\rs$ is of real type (even PHS) and if $d=5$ mod $8$ it is of quaternionic type (PHS is odd). $\Gamma_{d+1}$ anti-commutes with $h$ and 
is real for the real structure $\rcl_{d,1}$. Hence $\Gen{-} = \Gamma_{d+1}\otimes \rho$ anti-commutes with $h\otimes\rho$ and $h_0\otimes\rho$ and is imaginary, as $\rho$ is imaginary (the relevant algebra is $B=A^\rs\otimes Cl_{0,1}$). We thus find that $h\otimes \rho$ satisfies  $Y^{-}_{\kk_0\otimes\rho}$ so that 
$\Delta^c_{\ch'_d\#\jmath_1}([h\otimes\rho])$ is well-defined and, by (\ref{eq-formula-odd-b}), 
$$
\Delta^{c}_{\ch'_d\#\jmath_1}([h])  \eqm  \frac{\kappa^{-1}}2 
\int_{\TM^d}\jmath_1\Tr_N
\hat h \Gamma_{d+1} d\hat h \otimes\rho^{d+2}  \eqm  \frac{N}{2i} \int_{\TM^1}\epsilon^{i_0\cdots i_d}\hs_{i_0}\cdots d\hs_{i_d}
$$
As $\frac{N}{\kappa^{d+1}} = (-i)^\frac{d+1}2$ we get 
$$ \Delta^{(-)}_{\ch_d}([h]) := \Delta^{c}_{\ch_d\#\jmath_1}([h]) \eqm \frac12 \Ih. $$
The group which is quotiented out is 
$\langle \ch_d\#\jmath_2, \DK(A\otimes \Cl_{2})\rangle =  \langle \ch_d, KU_1(A) \rangle=\Z$
so that the torsion valued pairing is thus surjective onto $\Z_2$.
\subsubsection{TRS and chiral symmetry} 
Topological phases with this protecting symmetry are classified by $\DK(A^\rs)$. 
We implement a chiral symmetry on $A$ by considering the natural grading on $\Cl_{d+1}$ and define $h=\varphi_{+1}(\sum_{i=0}^d\kg_i \hs_i)$ using the graded representation $\varphi_{+1}:\Cl_{d+1}\to M_N(\C)$ with $N=2^{\mu(d+3)}$ which is the composition of the inclusion $\Cl_{d+1}\hookrightarrow \Cl_{d+2}$ followed by the faithful representation $\varphi:\Cl_{d+2}\to M_N(\C)$ considered already above ($M_N(\C)$ with standard even grading).
We implement TRS on $h$ by declaring $\kg_0$ to be real and $\kg_i$ for $i=1,\cdots,d$ to be imaginary. This does not determine uniquely the real structure on $M_N(\C)$ and we now distinguish between the two cases $d$ even and $d$ odd.

Let $d$ be even. 
The grading operator on $A=C(\TM^d,M_N(\C))$ is $\Gamma_{d+2}$.
Moreover $\Gen{-} := \gamma_{d+1}$ anticommutes with $h$ and $\Gamma_{d+2}$.
Thus if $\Gen{-}$ is imaginary it fulfills the requirements of Theorem~\ref{thm-TV-odd}. 
We therefore extend the real structure to $\rcl_{1,d+1}$ on $\Cl_{d+2}$ 
so that $\gamma_{d+1}$ is imaginary. This fixes the real structure on $A$ as $\varphi$ is surjective. 
It follows that $\DK(A^\rs) = \DK(C(\TM^d)^\rf\otimes 
Cl_{1,d+1}) \cong KO_{d+1}(C(\TM^d)^\rf)$. Then
the grading operator is real if $d = 0$ mod $4$ and imaginary if $d=2$ mod $4$. 
Furthermore, the real subalgebra $A^\rs$ is of real type if $d\in\{0,6\}$ mod $8$.
For the other values $d\in \{2,4\}$ mod $8$ the algebra is of quaternionic type. 
We infer from Theorem~\ref{thm-TV-odd} that 
$\Delta^{c}_{ch'_d}([h])$ is well-defined and potentially non-trivial on 
$\DK(C(\TM^d)^\rf\otimes Cl_{1,d+1})$ and there given by
$$\Delta^{c}_{\ch'_d}([h])  \eqm  \frac{\kappa^{-1}}2 
\int_{\TM^d}\jmath_{d+2}
\kk_{d+1} \hat h (d\hat h)^d  \eqm  \frac{\kappa^{d+1}}2 \int_{\TM^d}
 \epsilon^{i_0\cdots i_d} \hat \hs_{i_0} d\hat \hs_{i_1}\cdots d\hat \hs_{i_d}
$$
and hence 
$$ \Delta_{\ch_d}^{(-)}([h]) := \Delta^{c}_{\ch_d}([h\otimes\rho]) \eqm \frac12\Ih $$
The group which is quotiented out is 
$\langle \ch_d\#\jmath_{d+3}, \DK(A\hot \Cl_{d+3})\rangle$. As $A = C(\TM^d)\otimes \Cl_{d+2}$ we have $\langle \ch_d\#\jmath_{d+3}, \DK(A\hot \Cl_{d+3})\rangle =  \langle \ch_d, KU_0(\TM^d) \rangle = \Z$. The torsion valued pairing is thus surjective onto $\Z_2$.

Let now $d$ be odd. Then $M_N(\C)$ is isomorphic to $\Cl_{d+3}$ via the representation $\varphi$ and the grading operator is 
$\Gamma_{d+3} $. Furthermore $\Gen{+} = i\gamma_{d+1}\gamma_{d+2}$ 
is an even self-adjoint unitary which commutes with $h$ and $\Gamma_{d+3}$ and hence fulfills the requirements of Theorem~\ref{thm-TV-ev}
provided it is imaginary. There are two real structures on $\Cl_{d+3}$ extending $\rcl_{1,d}$
which achieve this, namely $\rcl_{3,d}$ and $\rcl_{1,d+2}$.
With the first choice, $\rs'=\rcl_{3,d}$, we have $\DK(A^\rs) = DK(C(\TM^d)^\rf\otimes 
Cl_{3,d}) \cong KO_{d-2}(C(\TM^d)^\rf)$ and and with the second  
$\DK(A^\rs) = DK(C(\TM^d)^\rf\otimes 
Cl_{1,d+2}) \cong KO_{d+2}(C(\TM^d)^\rf)$. Furthermore,
the grading operator is real if $d = 3$ mod $4$ and imaginary if $d=1$ mod $4$. We now infer from Theorem~\ref{thm-TV-ev} that $\Delta^{j}_{\ch_d\#\jmath_{1}}([h\otimes\rho])$ is well-defined and potentially non-trivial on $\DK(A^\rs)$ and given by ($d$ is odd)
$$
\Delta^{j}_{\ch'_d\#\jmath_{1}}([h\otimes\rho])  \eqm  \frac{1}2 
\int_{\TM^d}\jmath_{d+4}
i\kk_{d+1}\kk_{d+2} \hat h (d\hat h)^d\otimes \rho^{d+2} 
 \eqm  \frac{i\kappa^{d+3}} 2 \int_{\TM^d}
 \epsilon^{i_0\cdots i_d} \hat \hs_{i_0} d\hat \hs_{i_1}\cdots d\hat \hs_{i_d}
$$
As $i\kappa^2 = -2$ we have
$$ \Delta_{\ch_d}^{(+)}([h]) := \Delta^{j}_{\ch_d\#\jmath_{d+3}}([h\otimes\rho]) =  \Ih .$$
The group which we have to quotiented out is 
$$\langle \ch_d\#\jmath_2, \DK(A^\rs \hot Cl_{0,2})\rangle =  \langle \ch_d, KO_d(C(\TM^d)^\rf) \rangle =\Z$$ if we use $\rs'=\rcl_{3,d}$ and 
$$\langle \ch_d\#\jmath_2, \DK(A^\rs \hot Cl_{0,2})\rangle =\langle \ch_d, KO_{d+4}(C(\TM^d)^\rf) \rangle=2\Z$$ if we use $\rs'=\rcl_{1,d+2}$ as real structure. In the first case we thus get a trivial pairing, whereas in the second case the torsion value pairing is surjective onto $\Z_2$. In this second case, i.e.\  $\rs'=\rcl_{1,d+2}$, the algebra $A^\rs$ is quaternionic for $d\in\{1,3\}$ mod $8$ and real if $d\in\{5,7\}$ mod $8$, and the grading operator $\Gamma_{d+3}$ is imaginary if $d=1$ mod $4$ and real if $d=3$ mod $4$.

\subsection{Tabular summary}
We summarize the various possibilities we have discussed above. The Hamiltonian is given by $h = \sum_{i=0}^d\varphi(\rho_i)\hs_i \in C(\TM^d,M_N(\C))$ where $d$ is the dimension and 
$\varphi:\Cl_{d+1}\to M_N(\C)$ is a representation of the Clifford algebra $\Cl_{d+1}$. Which representation we take depends on which symmetry class we want to realise. The latter is determined by a graded real structure on $\Cl_{d+1}$ which is pushed forward to $M_N(\C)$ and possibly extended. 
In the presence of chiral symmetry the grading is the standard grading on the Clifford algebra and the standard even grading on the matrix algebra, and the representation $\varphi$ is graded. Otherwise the algebras and the representation are ungraded.

The topological phase of $h$ is classified by its van Daele class, 
$[h]$ if there is CS, $[h\otimes\rho]$ if not, in the relevant $K$-group. 
This $K$-group is isomorphic to $KO_i(C(\TM^d)^\rf)$ (or $KU_i(C(\TM^d))$, if there is no real protecting symmetry) the index $i$ depending on the symmetry class as explained in \cite{Kel1}. 

We present two tables. 
For Table~\ref{tab-si} the real structures and the representation are chosen such that the van Daele class associated to $h$ can have non-trivial Connes pairing with $\ch_d$. In particular the pairing of the $d$-dimensional Bott Hamiltonian with $\ch_d$ is $1$. The range of the pairing of $\ch_d$ with the relevant $K$-group is $\Z$. The torsion valued pairing is trivial (or undefined). The representation is the (up to conjugation unique) graded bijective representation $\varphi^{odd}:\Cl_{d+1}\to M_{\mu(d+2)}(\C)$ if $d$ is odd, or the ungraded surjective representation $\varphi^{ev}:\Cl^u_{d+1}\to M_{\mu(d+1)}(\C)$ if $d$ is even.

%
%

\begin{table}[h]
\begin{center}
\caption{{\bf Parameters for non-trivial Connes pairing with standard top chern character.}
The table presents the dimension $d$, the type of the algebra $A$ or $A^\rs$, the
symmetry type, and the $K$-group classifying the topological phase of the Hamiltonian.}  
\label{tab-si}
\begin{tabular}{|c|c|c|c|c|c|c|c|c|c|c|c|}
\hline
$d$ & type of alg. & CS & TRS & PHS & $K$-group \\
\hline
\hline
$0$ mod $2$ & complex &- & - & -& $KU_0(C(\TM^d))$\\
\hline
$1$ mod $2$ &  complex & yes & - & -& $KU_1(C(\TM^d))$\\
\hline
$0$ mod $8$& real &- &even & -& $KO_0(C(\TM^d)^\rf)$\\
\hline
$1$ mod $8$  & real &real & even &  even & $KO_1(C(\TM^d)^\rf)$\\
\hline
$2$ mod $8$ &  real &- &-& even & $KO_2(C(\TM^d)^\rf)$\\
\hline
$3$ mod $8$ &  quaternionic & imag.& odd & even & $KO_3(C(\TM^d)^\rf)$\\
\hline
$4$ mod $8$ &  quaternionic & - &odd & -& $KO_4(C(\TM^d)^\rf)$\\
\hline
$5$ mod $8$ & quaternionic & real & odd & odd & $KO_5(C(\TM^d)^\rf)$\\
\hline
$6$ mod $8$ &  quaternionic & - &- & odd & $KO_6(C(\TM^d)^\rf)$\\
\hline
$7$ mod $8$ & real & imag.& even & odd & $KO_7(C(\TM^d)^\rf)$\\
\hline
\end{tabular}
\end{center}
\end{table}

In Table~\ref{tab-sii} the real structure and the representation are chosen such that 
$\Delta^{(+)}_d(h)$ or $\Delta^{(-)}_d(h)$ can be non-trivial. This requires the Connes pairing with the van Daele class associated to $h$ to be $0$. The range of the pairing on the relevant $K$-group is $\Z_2$. The representation is the (up to conjugation unique) graded injective representation $\varphi^{+1}:\Cl_{d+1}\to M_{\mu(d+3)}(\C)$ if we have CS,
or the ungraded injective representation $\varphi^{u}:\Cl^u_{d+1}\to M_{\mu(d+2)}(\C)$ if there is no CS. 

\begin{table}[h]
\begin{center}
\caption{{\bf Parameters for a non-trivial torsion value pairing with the standard top chern character}. We present the dimension $d$, the type of the real subalgebra $A^\rs$, the
symmetry type, the type of the torsion valued pairing, and $K$-group classifying the topological phase of the Hamiltonian.}
\label{tab-sii}
\begin{tabular}{|c|c|c|c|c|c|c|c|c|c|c|c|}
\hline
$d$ mod $8$ &  type of $A^\rs$ & CS & TRS & PHS &  $\Delta^{(\pm)}_{d}$ & $K$-group  \\
\hline
\hline
$0$    & real &- &-& even & even & $KO_2(C(\TM^d)^\rf)$  \\
\hline
$1$  &  real & - &- & even & odd & $KO_2(C(\TM^d)^\rf)$ \\
\hline
$2$  &  quaternionic &- &odd & -& even & $KO_4(C(\TM^d)^\rf)$  \\
\hline
$3$ & quaternionic &- &odd & -& odd & $KO_4(C(\TM^d)^\rf)$  \\
\hline
$4$    & quaternionic &- &-& odd & even & $KO_6(C(\TM^d)^\rf)$ \\
\hline
$5$  &  quaternionic &- &- & odd & odd & $KO_6(C(\TM^d)^\rf)$ \\
\hline
$6$  &  real &- &even & -& even & $KO_0(C(\TM^d)^\rf)$  \\
\hline
$7$ & real &- &even & -& odd & $KO_0(C(\TM^d)^\rf)$  \\
\hline
$0$ &  real & real & even &  even & odd & $KO_1(C(\TM^d)^\rf)$ \\
\hline
$1$  & quaternionic & imag.& odd & even & even & $KO_3(C(\TM^d)^\rf)$ \\
\hline
$2$  & quaternionic & imag.& odd & even & odd & $KO_3(C(\TM^d)^\rf)$ \\
\hline
$3$  & quaternionic & real & odd & odd & even & $KO_5(C(\TM^d)^\rf)$ \\
\hline
$4$ &  quaternionic & real & odd &  odd & odd & $KO_5(C(\TM^d)^\rf)$ \\
\hline
$5$  & real & imag.& even & odd & even & $KO_7(C(\TM^d)^\rf)$ \\
\hline
$6$  & real & imag.& even & odd & odd & $KO_7(C(\TM^d)^\rf)$ \\
\hline
$7$  & real & real & even & even & even & $KO_1(C(\TM^d)^\rf)$ \\
\hline
\end{tabular}
\end{center}
\end{table}

\section{Two dimensional aperiodic tight binding models with odd TRS}
\label{sec-aperiodic}
The methods developped in this paper apply also to aperiodic tight binding models for which there is no underlying Brillouin zone so that we cannot perform a Bloch transformation and the algebra $A$ becomes fundamentally non-commutative. In this context a couple of structural questions have to be solved, in particular the question about the domain of the torsion valued pairings, that is, the size of $\ker j_*$ and $\ker c_*$. This is partly adressed in this section where we mainly  restrict ourselfs to two dimensional models which have an odd time reversal invariance, and hence consider only $\ker j_*$.
A more comprehensive discription of tight binding models of any dimension and with all types of protecting symmetries will be given elsewhere.

\newcommand{\hull}{\Xi}
\subsection{Observable algebra for aperiodic solids}
The natural generalisation of the observable algebra for tight binding models describing aperiodic solids without external magnetic field is the crossed product algebra 
$$A:=C(\hull,M_N(\C))\rtimes_\alpha\Z^d$$ where 
$\hull$ is the space of microscopic configurations on $\Z^d$ which can be realisations of the material, for instance by associating to a point in $\Z^d$ its atomic orbital type, and $\alpha$ the action of $\Z^d$ by shift of the configuration\footnote{we do not distinguish notationally the action of $\Z^d$ on $\hull$ by homeomorphisms from its pull back action on $C(\hull)$} \cite{Bel86}.  
Depending on the circumstances one wants to describe, the elements of $\hull$ are disorder configurations, or quasiperiodic configurations,
or even periodic configurations as for a periodic crystal. In the last case $\hull$ may be taken to be a single point and $\alpha=\id$, as $M_N(\C)\rtimes \Z^d\cong C(\TM^d,M_N(\C))$. 
We allow only finitely many possibilities of atomic orbitals at a point and in this case it is most natural to equip $\hull$ with a compact totally disconnected topology. The reason for this is that $\hull$ can be understood as an inverse limit of finite sets, namely the sets of configurations of finite size of which there are, for any given size, only finitely many.
It is usually assumed that $\hull$ carries an ergodic invariant probability measure $\PM$ and contains a dense $\Z^d$-orbit which lies in the support of the measure. Physically this may be justified by saying that we consider deformations only in a fixed thermodynamic phase.
This measure gives rise to a trace on $C(\hull)$ which extends to a positive trace $\Tr_A$ on $A$.

Elements of $A$ can be approximated by  finite sums of the form $\sum_{n\in\Z^d}  a_n u^n$ 
where $a_n\in C(\hull,M_N(\C))$, $u^n = u_1^{n_1} \cdots u_d^{n_d}$ with $d$ commuting unitaries $u_1, \cdots, u_d$, and multiplication and $*$-structure are given by
\begin{equation}\label{eq-crossed}
a_nu^n a_mu^m = a_n\alpha_n(a_m) u^{n+m},\quad (a_nu^n)^* = \alpha_{-n}(a_n^*) u^{-n} .
\end{equation}
In particular the action $\alpha$ on $C(\hull,M_N(\C))$ is induced by conjugation with the unitaries $u_1, \cdots,u_d$. The trace $\Tr_A$ is given by $\Tr_A(a_n u^n) =  \delta_{0n}\int_\Xi \Tr_N(a_0)d \PM$ and is thus invariant under the action $\alpha$.

The finite dimensional algebra $M_N(\C)$ is used to describe the internal degrees of freedom. Protecting symmetries are defined by real structures $\rs'$ and/or a grading $\gamma$ on $M_N(\C)$ as in the periodic case (last section) and then extended to the crossed product by acting trivially on the unitaries $u_1, \cdots, u_d$, $\rs(a_n u^n) = \rs'(a_n)u^n$. 

\bigskip

The chern character $\ch_d$ can be generalised to the aperiodic case as follows. Consider the derivations $\partial_1,\cdots,\partial_d$,
$$\partial_j(\sum_{n\in\Z^n} a_n u^n) = \sum_{n\in\Z^n} i n_j a_n u^n.$$
They commute among each other and satisfy $\Tr_A\circ \partial_j = 0$.  
Let $\Omega=A\ot\Lambda\C^d$ with $\Z_2$-grading $\gamma\ot \id$ and $\Z$-grading corresponding to the usual grading of the Grassmann algebra $\Lambda\C^d$. 
Define the differential $d:A\to A\ot \Lambda^1\C^d$ by $da = \sum_{i=1}^d \partial_i(a)\ot \gr_i$ where $\{\gr_i\}_{i}$ is a base of $\C^d$. As usual the differential extends uniquely to all of $\Omega$.  
We define a $*$-structure on $ \Lambda\C^n$ by declaring the elements $\gr_i$ to be self-adjoint. 
Let
$\imath:A\ot\Lambda\C^n \to A$ be given by 
$$\imath(a\ot \gr_1\wedge\cdots \wedge \gr_n) = a$$
and  $\imath(a\ot w) = 0$ for all $w\in\Lambda\C^d$ of degree less than $d$.
Then $(A\ot\Lambda\C^n, d, \Tr_A\circ\imath)$ is a $d$-dimensional cycle over $A$. 
The Fr\'echet algebra $\Aa\subset A$ of infinite sums  $\sum_{n\in\Z^2}  a_n u^n$ for which $n \mapsto \|a_n\|$ is rapidly decreasing is a dense subalgebra which is closed under holomorphic functional calculus  \cite{Rennie,PS} and thus a domain algebra for this cycle.

Under the Fourier-Bloch transformation and with the correct normalisation, the above cycle corresponds to the de Rham cycle over the Brillouin zone considered in the last section if $\hull$ is a single point.  

Suppose now that the observable algebra $A$ carries a real structure $\rs$ of the form above, $\rs(a_n u^n) = \rs'(a_n)u^n$. Then 
$$ \partial_j(\rs(a)) = -\rs(\partial_j(a)) $$
and we define a real structure on the Grassmann algebra by
$$ \rs'' (\lambda_j) = - \lambda_j $$
so as to guarantie that $d$ commutes with the real structure $\tilde \rs =  \rs\ot\rs''$ on $A\ot\Lambda\C^n$. Note that the trace $\Tr_A$ satisfies $\Tr_A(\rs(a)^*) = {\Tr_A(a)}$. 
It follows that $(A\ot\Lambda\C^n,\tilde\rs, d, \Tr_A\circ\imath)$ is a $(*,\rs)$-cycle of \tosi\ $$\ss = (-1)^{\nn(d)+d}.$$
We denote its character by $\ch'_d$. 

The observable algebra $A$ contains the element 
$$\tilde b_{\TM^d} : = \rho_0 \left(1 + \frac12\sum_{i=1}^d  (u_i+u_i^*-2) \right) + 
\frac{1}{2i}\sum_{i=1}^d  \rho_i (u_i - u_i^*)$$
for any choice of configuration space $\hull$. If the latter is reduced to one point, so that $A\cong C(\TM^d)\otimes \Cl_{d+1}$, then $\tilde b_{\TM^d}$ becomes $b_{\TM^d}$ from (\ref{eq-Bott-torus}), the element defining the Bott element on the torus. We therefore normalise the chern character again as $$\ch_d := \langle \ch'_d,[\tilde b_{\TM^d}]\rangle^{-1}\ch'_d$$  


\subsection{$K$-theory of the observable algebra}\label{sec-K}
The tool to compute the 
$K$-theory of $C(\hull)\rtimes_\alpha\Z^d$ and of its real subalgebra 
$C(\hull,\R)\rtimes_\alpha\Z^d$
is the Pimsner Voiculescu exact sequence \cite{PV,Sch}.  
We recall some details. Associated to an action $\alpha$ on a (trivially graded) \CA\ $B$ there is a short exact sequence, the so-called Toeplitz extension 
\begin{equation}\label{eq-Toeplitz}
0 \to B\otimes \Kk \to \Tt(B,\alpha)\stackrel{q}\to B\rtimes_\alpha\Z\to 0.
\end{equation}
$\Tt(B,\alpha)$ is the universal \CA\ generated by $B$ and a coisometry $S$, i.e.\ an element $S$ satisfying $S S^*=1$, such that $S a S^* = \alpha(a)$,  $a\in B$. Moreover $q(aS) = au_1$ with $u_1$ as in (\ref{eq-crossed}) (but for a $\Z$-action). Any real structure $\rs$ on $B$ which commutes with the action $\alpha$ can be extended to the crossed product algebra
$B\rtimes_\alpha\Z$ by $\tilde\rs(b\ot u_1) = \rs(b)\ot u_1$ for $b\in B$, and similarily it can be extended to the Toeplitz algebra $\Tt(B,\alpha)$ by $\tilde\rs(b\ot S) = \rs(b)\ot S$. The above short exact sequence 
(\ref{eq-Toeplitz}) is then equivariant w.r.t\ the real structure $\tilde\rs$ and restricting to $\tilde\rs$-invariant elements we obtain the Toeplitz extension of the real crossed product
$B^\rs\rtimes_\alpha\Z$. 
The important result of \cite{PV} which has been adapted to the real case in \cite{Sch} is that
the above short exact sequence gives rise to an  
exact sequence in complex or real $K$-theory (the Pimsner-Voiculescu exact sequence) which can be cut into short exact sequences, for each degree one. 
In the real case the short exact sequence in degree $i$ is\footnote{as is customary, we use also the notation $K_i = KU_i$ if $\FM = \CM$ and $K_i = KO_i$ if $\FM = \HM$.}
\begin{equation}\label{eq-PV}
0 \to C_{\alpha} KO_i(B^\rs) \stackrel{i_*}\to KO_i(B^\rs\rtimes_\alpha\Z) \stackrel{\delta}\to I_{\alpha} KO_{i-1}(B^\rs) \to 0.
\end{equation}
Here $C_{\alpha} KO_i(B^\rs):=  KO_i(B^\rs)/\sim_\alpha$ is the quotient module of coinvariant elements, that is the module $KO_i(B^\rs)$ modulo elements of the form $[x] - [\alpha(x)]$, $[x]\in KO_{i}(B^\rs)$,
and $I_\alpha KO_{i-1}(B^\rs) := \{[x]\in KO_{i-1}(B^\rs):[\alpha(x)]=[x]\}$ is the submodule of invariant elements. 
The quotient map $\delta$ in (\ref{eq-PV}) is the boundary map coming from the short exact sequence (\ref{eq-Toeplitz}). The other map $i_*$ of (\ref{eq-PV}) is induced by the inclusion $i:B\to B\rtimes_\alpha\Z$.

Since crossed products with $\Z^d$ can be seen as iterated crossed products with $\Z$ the above can be iterated to compute in principle the $K$-theory of the observable algebra, however the final result becomes more an more complicated with higher $d$ and can be given in closed form only if the short exact sequences (\ref{eq-PV}) at each stage split.
On the other hand, the pairing with $\ch_d$ can be simply expressed.
\begin{lemma}\label{lem-simple-pairings}
Let $\hull$ be a compact metrisable totally disconnected space with a continuous $\Z^d$-action $\alpha$. Suppose that the action has a dense orbit.
Let $\ch_d$ be the character of the cycle over $C(\hull)\rtimes_\alpha\Z^d$ introduced above.
We have 
$$\langle \ch_d,KU_i(C(\hull)\rtimes_\alpha \Z^d) \rangle = \left\{
\begin{array}{cl} 
\Z & \mbox{\rm if } i = d\: mod\: 2 \\
0 & \mbox{\rm otherwise}
\end{array}\right.
$$
and
$$\langle \ch_d,KO_i(C(\hull,\R)\rtimes_\alpha \Z^d) \rangle = \left\{
\begin{array}{cl} 
\Z & \mbox{\rm if } i = d\: mod\: 8 \\
2\Z & \mbox{\rm if } i = d+4 \: mod\: 8 \\
0 & \mbox{\rm otherwise}
\end{array}\right.
$$
\end{lemma}
\begin{proof}
Let $\FM$ be $\C$ or $\R$. 
Applying $d$ times iteratively the the Pimsner Voiculescu exact sequence \cite{PV,Sch}) one obtains the exact sequence
$$ 0 \to \ker \delta^{(d)}  \to K_i( C(\Xi,\FM) \rtimes_\alpha \Z^d)\stackrel{\delta^{(d)}}\to I_\alpha K_{i-d}(C(\hull,\FM)) \to 0$$
where $\delta^{(d)}$ is the composition of the $d$ boundary maps of the individual Pimsner Voiculescu exact sequences and we have written $I_\alpha = I_{\alpha_d}\cdots I_{\alpha_1}$. Since $\hull$ is totally disconnected $C(\hull)$ is the direct limit of finite dimensional commutative algebras. By continuity of the $K$-functor we thus have
$K_{i-d}(C(\hull,\FM))\cong C(\hull,K_{i-d}(\FM))$ and under this isomorphism the action $\alpha_*$ on $K_{i-d}(C(\hull,\FM))$ becomes the usual pull back action on functions over $\hull$. Since the action is transitive only constant functions are $\alpha$-invariant and thus  $I_\alpha K_{i-d}(C(\hull,\FM))\cong K_{i-d}(\FM)$.

We have $\langle \ch_d,\ker\delta^{(d)} \rangle = 0$ as 
the inclusion $\ker\delta^{(d)}\hookrightarrow K_i( C(\Xi,\FM)\rtimes_\alpha \Z^d)$ is composed of the inclusion maps $C(\Xi,\FM)\rtimes \Z^n\hookrightarrow C(\Xi,\FM)\rtimes \Z^{n+1}$ and hence the elements in its image vanish under $\partial_d$. There must therefore be a morphism $\varphi:K_{i-d}(\FM) \to \C$ such that, for all $[x]\in K_i(C(\Xi,\FM)\rtimes_\alpha \Z^d)$ we have  $\langle \ch_d,[x] \rangle = \langle \varphi,\delta^{(d)}([x]) \rangle$. 
Since $K_{i-d}(\FM)$ has at most one generator,  there is only one morphism up to normalisation. Clearly if $K_{i-d}(\FM)$ is pure torsion, or trivial, we have $\varphi=0$. 
Otherwise $\varphi=c\,\ch_0$
for some $c\in \C$. Up to a sign, the value of $c$ can be obtained as follows. 

Consider first the case that $\FM=\C$ and $\hull$ is a single point so that $C(\Xi,\FM)\rtimes_\alpha \Z^d \cong C(\TM^d)$. We determined in the last section that the smallest non-vanishing absolute value for $\langle \ch_d,K_d(C(\TM^d)) \rangle$ is given by 
$|\langle \ch_d,[b_{\TM^d}]\rangle|=1$ and the calculations made in Section~\ref{sec-Examples} show that $\langle \ch_0,K_{0}(\C) \rangle=\Z$. Hence in this case $\delta^{(d)}([\tilde b_{\TM^d}])$ is a generator of $K_{0}(\C)$ and  $c=\pm1$. Now since $C(\hull,M_2(\C))\rtimes_\alpha \Z^d$ contains 
$\tilde b_{\TM^d}$, $\delta^{(d)}([\tilde b_{\TM^d}])$ must also be a generator of 
$I_\alpha K_0(C(\Xi,\Z)) = K_{0}(\C)$ for general $\hull$. Thus for $\FM=\C$ we have 
$\langle \ch_d,\cdot\rangle = \pm \langle  \ch_0,{\delta^{(d)}}(\cdot)\rangle$. Since the Toeplitz sequence is equivariant w.r.t.\ the real structures the result remains true if we restrict it to the classes of the elements of
the real sub-algebras. In particular, 
$\langle \ch_d,KO_{d}(C(\hull,\R)\rtimes_\alpha \Z^d) \rangle = \langle \ch_0,KO_0(\R)) \rangle = \Z$ and
$\langle \ch_d,KO_{d+4}(C(\hull,\R)\rtimes_\alpha \Z^d) \rangle =
 \langle \ch_d,KO_d(C(\hull,\HM)\rtimes_\alpha \Z^d) \rangle = \langle \ch_0,KO_0(\HM)) \rangle = 2\Z$.
\end{proof}
Corollary~\ref{cor-simple-pairings} is obtained from the above result upon taking $\hull$ to be a single point, as $\R\rtimes \Z^d \cong C(\TM^d)^\rf$.
\subsection{Two dimensional systems with odd TRS}
We now restrict our analysis to two dimensional (possibly aperiodic) systems with odd time reversal invariance. Our aim is to determine the domain of the torsion valued pairing $\Delta^{(+)}_{\ch_2}=\Delta_{\ch_2\#\kappa_1}^j$, that is, the kernel of $j_*$ on the relevant $K$-group.

Following \cite{Kel1} we implement odd time reversal on 
$A = C(\hull,M_N(\C))\rtimes_\alpha\Z^d$ through a real structure $\rs$ 
of the form $\rs = \Ad_\Theta\circ\rf$ where $\Theta\in A$ 
is a unitary which satisfies $\Theta\rf(\Theta) = -1$ and $\rf$ is the reference real structure
$ \rf(a_n u^n) = \cc(a_n) u^n$.
Here $\cc(a_n)$ is complex conjugation of the matrix elements in ${a_n}\in M_N(\C)$. 
We know from the general theory of \cite{Kel1} that, under mild conditions on the spectrum of $\Theta$ (for instance if the spectrum is finite), $(M_2(A),\rs_2)$ is 
conjugate to $(M_2(\C)\ot A,\rh\ot \rf)$ where $\rh = \Ad_{i\sigma_y}\circ\cc$ is the quaternionic real structure on $M_2(\C)$ whose real subalgebra is the algebra of quaternions
$M_2(\C)^\rh=\HM$. We may thus work from the beginning with the observable algebra 
$A=M_2(\C)\ot C(\hull,M_n(\C))\rtimes_\alpha\Z^2$ 
equipped with the real structure $ \rs = \rh \ot\rf$. Its real subalgebra is
$A^\rs=\HM \ot C(\hull,M_n(\R))\rtimes_\alpha\Z^2$.


A Hamiltonian describing an insulator with odd TRS
corresponds to a self-adjoint invertible element $h\in A^\rs$. It can thus be expressed as a  
$2\times 2$ matrix 
$$ h = \begin{pmatrix} h_1 & R \\ R^* & h_2 \end{pmatrix} $$
whose entries belong to $C(\hull,M_n(\R))\rtimes_\alpha\Z^2$ and satisfy 
$\rf(h_1) = h_2$ and $\rf(R) = -R^*$.
The models discussed in Section~\ref{sec-oddTRS} (for $d=2$) and the 
 Kane-Mele model are of the above type if one choses $\Xi$ to be one point and $n=2$.

For the calculation of the $K$-theory of $A$ and $A^\rs$, the value of $n$ is not important and we will set it to $1$.
\begin{prop}\label{prop-K-groups}
Let $\hull$ be a compact metrisable totally disconnected space with a continuous $\Z^2$-action $\alpha$. 
Let $\FM=\CM$ or $\FM=\HM$. For $i=0$ and $i=2$ we have the exact sequences
$$0\to C_{\alpha}C(\hull,K_i(\FM)) \to 
 K_i(C(\hull,\FM)\rtimes_{\alpha}\Z^2) \stackrel{\delta^{(2)}}\to I_{\alpha}C(\hull,K_{i-2}(\FM))\to 0.$$
Furthermore, $KO_1(C(\hull,\HM)\rtimes_{\alpha}\Z^2)$ is torsion free.
\end{prop}
Before giving the proof we remark that, if $\FM=\CM$ then there is no distinction between $i=0$ and $i=2$. Furthermore, the fact that  $K_2(\HM) = 0$ simplifies the exact sequence in the case $\FM=\HM$ and $i=2$.
\begin{proof} We view the $C(\hull,\FM)\rtimes_{\alpha}\Z^2$ 
as a double crossed product $(C(\hull,\FM)\rtimes_{\alpha_1}\Z)\rtimes_{\alpha_2}\Z$ and apply (\ref{eq-PV}) twice. For $\FM=\CM$ this calculation can be found in \cite{vanElst}. We consider here the case $\FM=\HM$.
We obtain first from the action $\alpha_1$ on $C(\hull,\HM)$
\begin{equation}\label{eq-PV1}
0 \to C_{\alpha_1} KO_i(C(\hull,\HM)) \to KO_i(C(\hull,\HM)\rtimes_{\alpha_1}\Z) \to 
I_{\alpha_1} KO_{i-1}(C(\hull,\HM))\to 0.
\end{equation}
Since $\hull$ is totally disconnected 
we have $KO_i(C(\hull,\HM)) \cong C(\hull,KO_i(\HM))$.
It is known that $KO_i(\HM)$ is $\Z,0,0,0,\Z,\Z_2,\Z_2,0$ in degrees $i=0,1,\cdots,7$. By (\ref{eq-PV1}) 
\begin{eqnarray*}
KO_{-1}(C(\hull,\HM)\rtimes_{\alpha_1}\Z)& \stackrel{\delta_1}\cong& I_{\alpha_1}C(\hull,KO_{-2}(\HM))\\
 KO_0(C(\hull,\HM)\rtimes_{\alpha_1}\Z)& \cong& C_{\alpha_1}C(\hull,KO_0(\HM))\\
  KO_1(C(\hull,\HM)\rtimes_{\alpha_1}\Z)& \stackrel{\delta_1}\cong& I_{\alpha_1}C(\hull,KO_0(\HM))\\
   KO_2(C(\hull,\HM)\rtimes_{\alpha_1}\Z)& = & 0
\end{eqnarray*}
where $\delta_1$ is the boundary map for the action $\alpha_1$.
We insert this into the exact sequence (\ref{eq-PV}) for the second action $\alpha_2$ on $B^\rs = C(\hull,\HM)\rtimes_{\alpha_1}\Z$. For $i=0,2$ we obtain precisely the exact sequences stated in the proposition. The quotient map $\delta^{(2)}$ is thus the composition of the boundary maps $\delta_1$ and $\delta_2$ for the two actions $\alpha_1$ and $\alpha_2$. 
For $i=1$ we obtain  
$$0\to C_{\alpha_2}I_{\alpha_1}C(\hull,K_0(\HM)) \to 
 K_1(C(\hull,\HM)\rtimes_{\alpha}\Z^2) \stackrel{\delta_2\circ\delta_1}\to I_{\alpha_2}C_{\alpha_1}C(\hull,K_{0}(\HM))\to 0.$$
Since $ I_{\alpha_2}C_{\alpha_1}C(\hull,\Z)$ and 
$C_{\alpha_2}I_{\alpha_1}C(\hull,\Z)$ are always torsion free $K_1(C(\hull,\HM)\rtimes_{\alpha}\Z^2$ is also torsion free.
\end{proof}

\begin{cor}\label{prop-j1}
Let $\hull$ be a compact metrisable totally disconnected space with a continuous $\Z^2$-action $\alpha$. Suppose that the action has a dense orbit.
Then $K_0(C(\hull,\HM)\rtimes_{\alpha}\Z^2)$ is determined by the exact sequence
$$0\to C_{\alpha}C(\hull,\Z) \to 
 K_0(C(\hull,\HM)\rtimes_{\alpha}\Z^2) \stackrel{\delta^{(2)}}\to \Z_2\to 0$$
and $  j_*: KO_0(C(\hull,\HM)\rtimes_{\alpha}\Z^2)\to KO_1(C(\hull,\HM)\rtimes_{\alpha}\Z^2)$ is the $0$-map. In particular, 
$\Delta^{(+)}_{\ch_2}$ is defined on all of 
$K_0(C(\hull,\HM)\rtimes_{\alpha}\Z^2)$.
Furthermore,
\begin{eqnarray*}  KO_2(C(\hull,\HM)\rtimes_{\alpha}\Z^2)& \cong & \Z
\end{eqnarray*}
with generator given by $\begin{pmatrix} \varphi_{ev}(\tilde b_{\TM^2}) & 0 \\ 0 & -\rf(\varphi_{ev}(\tilde b_{\TM^2})) \end{pmatrix}$.
\end{cor}
\begin{proof} 
The exact sequence is just the specialisation of that of Prop.~\ref{prop-K-groups} to $i=0$ taking into account that $I_{\alpha}C(\hull,K_{-2}(\HM))\cong KO_2(\R)=\Z_2$, as the action has a dense orbit.
We have seen that the image of $j_*$ is pure torsion. But
$KO_1(C(\hull,\HM)\rtimes_{\alpha}\Z^2)$ is torsion free and hence the image of $j_*$ on the $K_0$-group trivial. 
Since $KO_2(\HM)=0$ Prop.~\ref{prop-K-groups} yields that $\delta^{(2)}:KO_2(C(\hull,\HM)\rtimes_{\alpha}\Z^2) \to I_{\alpha}C(\hull,KO_0(\HM))$ is an isomorphism. 
As $ I_{\alpha}C(\hull,KO_0(\HM)) \cong KO_0(\HM)\cong\Z$ the generator of $KO_2(C(\hull,\HM)\rtimes_{\alpha}\Z^2)$ must be the same as in the periodic case where it was determined in the last section.
\end{proof}

\begin{cor}\label{prop-K-ch} 
Let $\hull$ be a compact metrisable totally disconnected space with a continuous $\Z^2$-action $\alpha$ which has a dense orbit. Then
$$\Delta^j_{\ch_2\#\kappa_1}:KO_0(C(\hull,\HM))\to \langle \ch_2,KO_0(C(\hull,\HM)\rtimes_{\alpha}\Z^2\rangle /\langle \ch_2,KO_2(C(\hull,\HM)\rtimes_{\alpha}\Z^2\rangle$$ is surjective onto $\Z_2$.
\end{cor}
\begin{proof}
By Lemma~\ref{lem-simple-pairings}  $\langle \ch_2,KO_2(C(\hull,\HM)\rtimes_{\alpha}\Z^2\rangle=2\Z$ and we already computed that  $\Delta^j_{\ch_2\#\kappa_1}([\varphi^u(\tilde b_{\TM^2})]) = 1$ mod $2$.
\end{proof}

Theorem~\ref{thm-WK} in combination with the above corollary tells us that any element of $KO_0(C(\hull,\HM)\rtimes_{\alpha}\Z^2$ admits a representative which has an extra spin symmetry $\Gen{+}$ and formula (\ref{eq-formula-even}) tells us how to express the torsion-valued pairing with the help of that symmetry.
If this spin symmetry is internal, that is, does not depend on the unitaries $u_j$ then 
(\ref{eq-formula-even}) simplifies to (\ref{eq-formula-even-b}) and 
$\Delta^{(+)}_{\ch_2}([h])$ is one half of the difference of the Chern number of the projection of $h$ onto the spin up sector minus the Chern number of its projection onto the spin down sector. 
We refer the reader to \cite{Prodan} for a thorough discussion of this interpretation as a so-called spin Chern number. It does not really need exact commutation and also the effect of strong disorder can be included \cite{Prodan}.

\newcommand{\eps}{\epsilon}
\newcommand{\Hf}{H^{\mbox{\footnotesize eff}}}
\newcommand{\tP}{P_{z_0,z_1}}

\subsection{Periodically driven models}
We provide an application to periodically driven insulators (Floquet insulators) with 
observable algebra $A = C(\hull,M_N(\C))\rtimes_\alpha\Z^2$ and odd
time reversal invariance. In the case that these are crystalline, that is, $\hull$ equal to a single point, they have been discussed by Carpentier {\it et al.}. Their work   \cite{Gawedzki} was actually the source of inspiration for our construction in Section~\ref{sec-tor}. 

Let $H(t)\in A$ be a continuously differentiable function of self-adjoint elements. $H(t)$ should be thought of as a time dependent Hamiltonian describing the material subject to an external force. We suppose that the time dependence is periodic, of period $T>0$. Let $U(t)$ be the unitary time evolution operator, which is the solution of the initial value problem
$$ i\dot U(t) = H(t)U(t),\quad U(0) = 1 .$$
It follows that $U(t)$ is $T$-periodic up to multiplication by $U(T)$,
$$ U(t+T) = U(T) U(t).$$
For simplicity we now set $T=1$. 
We are interested in the topological properties of the spectral projections of $U(1)$ which belong to $A$. The spectrum of $U(1)$ is a subset of the circle $S^1$ of complex numbers of modulus $1$. We need to assume that it is not all of $S^1$ but has at least two gaps so that we can define the spectral projection onto the spectral part between the gaps.
More precisely,  if $z_0,z_1$ are two distinct points in $S^1$ we denote by $[z_0,z_1]$ the subset of points in $S^1$ which are counter-clockwise to the left of $z_0$ and to the right of $z_1$, and define $\tP$ to be
the spectral projection of $U(1)$ onto $[z_0,z_1]$. If $z_0$ and $z_1$ do not belong to the spectrum of $U(1)$ then $\tP$ is a continuous function of $U(1)$ and thus lies in $A$. Consequently it defines an element in $KU_0(A)$ or, equivalently, $x_{z_0,z_1}:=(2\tP-1)\ot \kg$ an element of the van Daele $K$-group $\DK_0(A\otimes \Cl_{1})$. We are interested in pairings of this element with chern characters, in particular, in dimension two, with the standard chern character $\ch_2$ described above. But note that 
we are considering spectral projections of the time $1$ evolution operator $U(1)$ and not spectral projections of the Hamiltonian itself.

The spectral projection $\tP$ can be computed at follows: Fix a point $z\in S^1$ to define a domain for a complex logarithm $\log:\C\backslash \R^+z\to\C$. Choose a branch of that logarithm, that is, a real number $\epsilon$ such that $e^{i\epsilon } = z$ so that $\log_\epsilon(e^{i\varphi})$ is the imaginary number $i\varphi_\epsilon$ which 
satisfies $\epsilon \leq \varphi_\epsilon < \epsilon +2\pi$ and $\varphi_\epsilon-\varphi\in 2\pi\Z$. Then
$$\Hf_\eps : = i\log_\epsilon U(1)$$
is a selfadjoint operator which can be seen as an effective time independent Hamiltonian, because at times $t=n$, $n\in \Z$, the time evolution of $H(t)$ and of $\Hf_\eps$ coincide,
$U(n) = e^{-i n \Hf_\eps}$.
In between these times the difference of their time evolution is given by the periodized time evolution operator
$$V_\epsilon(t) = U(t) e^{i  t \Hf_\eps}$$
which satisfies $V_\epsilon(t+1)=V_\epsilon(t)$. If $z$ does not belong to the spectrum of $U(1)$ then $\Hf_\eps$ is a continuous function of $U(1)$ and hence an element of $A$.

Now suppose that $z_0$ and $z_1$ do not belong to the spectrum of $U(1)$. Choose $\eps_i$ so that $e^{i\eps_i } = z_i$ and $0\leq (\epsilon_1 - \epsilon_0) <2\pi $. Then, for $z\neq z_1,z_2$ we have 
$\log_{\epsilon_1}(z) -\log_{\epsilon_0}(z) = 
2\pi i$ if $z\in [z_1,z_2]$ while this expression is $0$ otherwise. It follows from functional calculus that
$$2\pi i \tP  = \log_{\epsilon_1}(U(1)) -\log_{\epsilon_0}(U(1)) = 
-i(\Hf_{\epsilon_1}-\Hf_{\epsilon_0}).$$
The last expression may be interpreted in $K$-theory as follows. Let $\beta:KU_0(A)\to KU_1(SA)$ be the Bott map from (\ref{eq-BS}). Then, using 
$\exp(-2\pi i s \tP) = e^{is (\Hf_{\epsilon_1}-\Hf_{\epsilon_0})} = V_{\epsilon_1}(s) V^*_{\epsilon_0}(s)$ we find 
$\beta([x_{z_0,z_1}]) = [Y]$ where $Y$ is the loop in $M_2(A)$,
$$ Y(s) = \begin{pmatrix} 0 & V_{\epsilon_1}(s) V^*_{\epsilon_0}(s) \\
V_{\epsilon_0}(s) V^*_{\epsilon_1}(s) & 0 \end{pmatrix} .$$

\subsubsection{Odd time reversal symmetry} 
We now consider the effect of time reversal symmetry. Recall that time reversal symmetry is implemented by a real structure $\rs$ on $A=M_2(\C)\otimes C(\Xi,M_{n}(\C))\rtimes_\alpha\Z^d$.
A time dependent Hamiltonian $H(t)$ has time reversal symmetry if 
$$\rs(H(t)) = H(-t).$$
This implies that $\rs(U(t)) = U(-t)$, $\rs(\Hf_\eps) = \Hf_\eps$, $\rs(\tP) = \tP$, and 
$\rs(V_\eps(t)) = V_\eps(-t)$.
In particular, $\tP$ defines an element of $KO_0(A^\rs)$. 

For periodic models with odd time reversal symmetry
Carpentier {\it et al.}\ \cite{Gaw,Gawedzki} 
proposed to associate to $\tP$ an invariant which 
we  wish to describe  in our framework. Indeed, we will show that their invariant corresponds to the torsion valued pairing of $\ch_2$ with $[x_{z_0,z_1}]$ and thus generalises to aperiodic models. 
For periodic models 
$A = M_{2n}(\C)\rtimes_\alpha\Z^2$ with trivial action $\alpha$ and this algebra 
is isomorphic to  $M_{2n}(\C)\ot C(\TM^2)$. Under the isomorphism the real structure becomes
$\rs(f)(k) = \rh(f(-k))$
for $f:\TM^2\to M_2(\C)\ot M_n(\C)$. This Fouriertransformation understood, we view now $\tP = \tP(k)$ and
$V_\eps=V_\eps(t,k)$ as continuous functions 
$$\tP : \TM^2\to M_2(\C)\ot M_n(\C),\quad V_\eps : \TM^3\to M_2(\C)\ot M_n(\C)$$ 
which satisfy 
$ \Ad_{\sigma_y\ot 1}\tP(k) = \overline{\tP(-k)}$
and 
$$ \Ad_{\sigma_y\ot 1}V_\eps(t,k) = \overline{V_\eps(-t,-k) }.$$
The invariant associated to $\tP$ in \cite{Gawedzki} is obtained by first modifying $V_\eps$.
Let 
\begin{equation} \label{eq-hatV}
\hat V_{\eps}(t,k) =\left\{\begin{array}{cc}
 V_{\eps}(t,k) & \mbox{for } t\in [0,\frac12]\\
\hat V_{\eps}(t,k) & \mbox{for } t\in [\frac12,1]
\end{array}\right.
\end{equation}
where for $[\frac12,1]\times\TM^2 \ni (t,k) \mapsto \hat V_{\eps}(t,k)\in U_N(\C)$ is any continuously differentiable function which satisfies the symmetry constreint
 $$ \Ad_{\sigma_y\ot 1}\hat V_\eps(t,k) = \overline{\hat V_\eps(t,-k) }$$
and the boundary conditions
$$\hat V_\eps(\frac12,k) = V_{\eps}(\frac12,k),\quad \hat V_\eps(1,k) = 1.$$
A substantial part of \cite{Gawedzki} is devoted to the proof that such a function $\hat V_\eps$ exists.
The degree of $\hat V_\eps$ is defined to be
$$\mathrm{deg}\hat V_\eps := \frac1{24\pi^2}\int_{\TM^3} \Tr(\hat V_\eps^*d\hat V_\eps)^3$$
where $d$ is the exterior derivative and $ \int_{\TM^3}$ the standard integral over $\TM^3$.
Finally,  the invariant associated to $\tP$ is the difference modulo $2$
$$ K(\tP) := \mathrm{deg}\hat V_{\eps_1} - \mathrm{deg}\hat V_{\eps_0} 
\quad \mbox{mod } 2.$$ 
Carpentier {\it et al.}\ then show that $ K(\tP) $ coincides with the Kane-Mele invariant of $\tP$.
Here we point out that $ K(\tP) $ corresponds exactly to 
$\Delta^{j}_{\ch_2\#\kappa_1}([x_{z_0,z_1}])$. 
Indeed, we can split the integral 
$$\frac1{24\pi^2}\int_{\TM^3} \Tr(\hat V_\eps^*d\hat V_\eps)^3 =
\frac1{24\pi^2}\int_{[0,\frac12]\times\TM^2} \Tr( V_\eps^*d V_\eps)^3+\frac1{24\pi^2}\int_{[\frac12,1]\times \TM^2} \Tr(\hat V_\eps^*d\hat V_\eps)^3.$$
Integrating out the time variable $t\in [0,\frac12]$ one finds that
$$\frac1{24\pi^2}\int_{[0,\frac12]\times\TM^2} \Tr( V_{\eps_1}^*d V_{\eps_1})^3
-\frac1{24\pi^2}\int_{[0,\frac12]\times\TM^2} \Tr( V_{\eps_0}^*d V_{\eps_0})^3$$
is proportional to $\langle \ch_2, [x_{z_0,z_1}]\rangle$ which vanishes, as follows from the last statement of Thm.~\ref{thm-TV-ev} (the dimension, parity, and \tosi\ of $\ch_2$ satisfy (\ref{eq-cond+})).
Let $$F(t) = \begin{pmatrix}0 & \hat V_{\eps_1}(\frac{t+1}2)\hat V^*_{\eps_0}(\frac{t+1}2)\\
\hat V_{\eps_0}(\frac{t+1}2)\hat V^*_{\eps_1}(\frac{t+1}2) & 0 \end{pmatrix}.$$ 
$F$ is a homotopy between 
$\begin{pmatrix}0 & \tP - \tP^\perp \\ \tP - \tP^\perp & 0 \end{pmatrix}$ and 
$\begin{pmatrix}0 & 1 \\ 1 & 0 \end{pmatrix}$. 
Identifying $\rho$ with $\begin{pmatrix}0 & 1 \\ 1 & 0 \end{pmatrix}$ we thus find that
\begin{eqnarray*}
\Delta^{j}_{\ch_2\#\kappa_1}([x_{z_0,z_1}]) & \eqm & c_2^{-1} (8\pi \kappa^3)^{-1}{\ch'_2}_{[0,1]}\# \jmath_2 (F,F,F)\\
& \eqm &  \frac1{24\pi^2} \int_{[0,1]\times \TM^2}  
\Tr\big(F dF\big)^3\\
& \eqm &  \frac1{24\pi^2} \int_{[\frac12,1]\times \TM^2}  
\Tr\big(\hat V_{\eps_1} d\hat V_{\eps_1}\big)^3
- \frac1{24\pi^2} \int_{[\frac12,1]\times \TM^2}  
\Tr\big(\hat V_{\eps_0} d\hat V_{\eps_0}\big)^3
\end{eqnarray*}
($\eqm$ means equality modulo $2$). This shows the result.

\end{document}